\documentclass{m2an}
\usepackage{rotating}
\begin{document}
\title{Isogeometric discretizations of the Stokes problem on trimmed geometries}\thanks{We gratefully thank Annalisa Buffa for her teaching and advice and Rafael V\'azquez Hern\'andez for the many discussions and his precious support.}\thanks{This work was partially supported by ERC AdG project CHANGE n. 694515.}
\author{Riccardo Puppi}\address{Chair of Modelling and Numerical Simulation, \'Ecole Polytechnique F\'ed\'erale de Lausanne, Station 8, 1015 Lausanne, Switzerland.}
%
\date{\today}
\begin{abstract} 
The isogeometric approximation of the Stokes problem in a trimmed domain is studied. This setting is characterized by an underlying mesh unfitted with the boundary of the physical domain making the imposition of the essential boundary conditions a challenging problem. A very popular strategy is to rely on the so-called Nitsche method~\cite{MR3264337}. We show that the Nitsche method lacks stability in some degenerate trimmed domain configurations, potentially polluting the computed solutions. After extending the stabilization procedure of~\cite{MR4155233} to incompressible flow problems, we show that we recover the well-posedness of the formulation and, consequently, optimal a priori error estimates. Numerical experiments illustrating stability and converge rates are included.

 \end{abstract}
%
%
\maketitle
\section*{Introduction}
In Computer Aided Design (CAD) complex geometries are usually constructed as a set of simple spline, or more generally NURBS, geometries plus a set of Boolean operations, among which there is set difference, commonly called \emph{trimming} in this context. Isogeometric analysis (IGA)~\cite{MR2152382} is a numerical method for approximating the solutions of PDEs to allieviate the need to remodel the geometries via a tetrahedral or hexahedral mesh, one of the major bottlenecks of the classical Finite Element Method (FEM). Following the \emph{isoparametric paradigm} we wish to use the same basis functions employed for the parameterization of the geometric domain for the discretization and analysis of the differential problem taken into consideration.
In recent years, great progress has been made in the domain of IGA to improve the usability of CAD geometries in the solution of PDEs and, in this respect, volumetric representations (V-rep) are a major contribution~\cite{MR3513034,MR3982623}. However further efforts are needed for it to be considered a mature field. In this respect trimming is a major challenge for isogeometric methods: on one hand, it is a fundamental element in the design phase of the geometric domain, while on the other, it is an obstacle to the analysis, in particular to the development of robust and reliable solution methods. For further details about the challenges of trimming in IGA we refer the interested readers to the review article~\cite{MR3867694} and the references therein. 

In practice, the simplest trimmed geometry is an object endowed with a tensor product mesh which is cut by an arbitrary line in 2D or surface in 3D. Similar challenges were considered, already long ago (see the pioneering work of Barrett and Elliott~\cite{MR752608,MR853660}), in the context of FEM.

Hence all the techniques developed by the finite element community come now to help. Among the most successful approaches relying on a solid mathematical foundation, there is CutFEM, a construction developed in the last decade by Burman and collaborators~\cite{MR3416285,MR2738930}.

Let us outline the setting we consider in this paper. We set $\Omega_0 \subset \mathbb R^d$ (here $d=2,3$), a domain parametrized by a bijective spline map $\F : \left(0, 1\right)^d \to \Omega_0$, i.e., a \emph{patch} in the isogeometric terminology, and let $\Omega_1,\dots,\Omega_N$ be Lipschitz domains in $\mathbb R^d$ . We assume that $\Omega_i$, $i=1,\dots,N$,  are to be cut away from $\Omega_0$ and that our computational domain reads:
\begin{equation}\label{trimmed_domain}
	\Omega = \Omega_0\setminus\bigcup_{i=1}^N \overline\Omega_i.
\end{equation}
In this paper, we contribute to the development of robust isogeometric numerical methods for PDEs in a trimmed domain such as~\eqref{trimmed_domain}.  It is known that the main sources of issues are: integration, conditioning, and stability~\cite{MR4155233}. Here, we address the stability of incompressible flow problems in trimmed geometries. For what concerns integration, namely the construction of suitable quadrature rules on the cut elements, we rely on the technique developed in~\cite{MR3982623}. The conditioning issue is also out of the scope of this work: herein we limit ourselves to applying a block diagonal rescaling of the basis functions. Our contribution lies on the theoretical side and extends the analysis developed in~\cite{MR4155233} to incompressible flow problems. We introduce the Raviart-Thomas, N\'ed\'elec, and Taylor-Hood mixed isogeometric elements (first studied in~\cite{MR2808112}) and adapt their definitions to trimmed domains. For the weak imposition of the Dirichlet boundary conditions, we rely on Nitsche's method. After empirically demonstrating the lack of stability of the Nitsche formulation, we propose our stabilization. On the one hand, just as in the elliptic case, we modify the evaluation of the normal derivatives of the velocities at the ``badly'' cut elements. At the same time, we apply the stabilization to the whole space of pressures, by substituting the degrees of freedom (DOFs) at the bad elements with a linear combination from a ``good'' neighboring element. Let us observe that a mathematical proof of the inf-sup condition for the stabilized formulation is still missing. However, numerical experiments indicate that our approach works.

Let us now briefly review the literature about isogeometric methods on trimmed domains. First of all, we observe that most of the contributions to this topic propose to weakly enforce Dirichlet boundary conditions using the Nitsche method~\cite{MR1365557} and they adopt the adaptive quadrature strategy developed in~\cite{RANK2012104}. For instance, let us refer to the papers in immersogeometric methods for fluid-structure interactions~\cite{MR3310316,MR3589879} where the structure is immersed in the background mesh of the fluid which is cut in an arbitrary fashion. In a series of articles~\cite{MR3610105,MR3801783,MR3873865} Hoang et al. employ different families of stable (in the boundary fitted case) isogeometric elements, namely the Taylor-Hood element~\cite{MR2250029,MR2808112,MR2846758,MR3047946}, the Raviart-Thomas element~\cite{MR2808112,MR2792397,MR3021778}, the N\'ed\'elec element~\cite{MR2808112} and the Subgrid element~\cite{MR3047946}, and empirically show that the well-posedness and good conditioning in the unfitted case can be recovered by adding to the variational formulation a consistent term penalizing the jumps of high-order derivatives of the pressure, in the spirit of~\cite{MR2571349,MR3264337}.  In~\cite{MR3610101} the authors develop an overlapping Additive-Schwarz preconditioner tailored for incompressible flow problems, without addressing the stability issue.

The manuscript is structured as follows. After having introduced in Section~1 some notations and the strong formulation of the Stokes problem, in Section~\ref{stokes_trimming:sec:ig} we provide the basic notions on IGA: we introduce three families of isogeometric elements (Raviart-Thomas, N\'ed\'elec, and  Taylor-Hood) and use them to discretize the considered equations using Nitsche's method for the imposition of the essential boundary conditions. In Section~\ref{stokes_trimming:section:lack_stab} we show through numerical experiments that the Nitsche formulation is not stable for the trimming operation. Then, in Section~\ref{stokes_trimming:sec:stabilization} we introduce our stabilized Nitsche formulation and then move on to its numerical analysis in Section~\ref{stokes_trimming:section5}. We demonstrate that the stabilized formulation is well-posed and, in Section~\ref{stokes_trimming:sec:apriori}, that we have optimal a priori error estimates. In Section~\ref{stokes_trimming:sec:numerical_experiments} we provide several numerical experiments to validate the effectiveness of our stabilization procedure.

\section{Notation and model problem}\label{stokes_trimming:sec:notation}
We briefly introduce some useful notations for the forthcoming analysis. Let $D$ be a Lipschitz-regular \emph{domain} (subset, open, bounded, connected) of $\mathbb R^d$, $d\in\{2,3\}$. Let $L^2(D)$ denote the space of square integrable functions on $D$, equipped with the usual norm $\norm{\cdot}_{L^2(D)}$. Let $L^2_0(D)$ be the subspace of $L^2(D)$ of functions with zero average. We denote by $H^n(D)$, for $n\in\N$, the standard Sobolev space of functions in $L^2(D)$ whose $n$-th order weak derivatives belong to $L^2(D)$. We define, for $\varphi:D\to \R$ sufficiently regular and $\bm\eta$ multi-index with $\abs{\bm\eta}:=\sum_{i=1}^d\eta_i$, $D^{\bm\eta}\varphi:=\displaystyle{\frac{\partial^{\abs{\bm\eta}}\varphi}{\partial x_1^{\eta_1}\dots\partial x_d^{\eta_d}}}$, $\norm{\varphi}^2_{H^n(D)}:=\sum_{\abs{\bm \eta}\le n} \norm{D^{\bm\eta}\varphi}^2_{L^2(D)}$. Sobolev spaces of fractional order $H^k(D)$, $k\in\R$, can be defined by interpolation techniques, see~\cite{MR2424078}. Let us denote $\bm L^2(D):=\left(L^2(D)\right)^d$ and $\bm H^k(D):= \left( H^k(D) \right)^d$. We define the Hilbert space $\bm H(\dive;D)$ of vector fields in $\bm L^2(D)$ with divergence in $L^2(D)$, endowed with the graph norm $\norm{\cdot}_{H(\dive;D)}$. Let $H^\frac{1}{2}(\partial D)$ be the range of the trace operator of functions in $H^1(D)$ and, for a non-empty open subset of the boundary $\omega$, we define its restriction $H^{\frac{1}{2}}(\omega)$. Both $H^\frac{1}{2}(\partial D)$ and $H^{\frac{1}{2}}(\omega)$ can be endowed with an intrinsic norm, see~\cite{MR2328004}. The dual space of $H^{\frac{1}{2}}(\omega)$ is denoted $H^{-\frac{1}{2}} (\omega)$. Finally, let $\bm H^\frac{1}{2}(\partial D):= \left( H^\frac{1}{2}(\partial D) \right)^d$, $\bm H^{\frac{1}{2}}(\omega):= \left( H^{\frac{1}{2}}(\omega) \right)^d$ and $\bm H^{-\frac{1}{2}} (\omega):= \left( H^{-\frac{1}{2}} (\omega) \right)^d$.

 For the sake of convenience, we are going to employ the same notation $\abs{\cdot}$ for the volume (Lebesgue), the surface (Hausdorff) measures of $\R^d$ and the cardinality of a set. We also denote as $\mathbb Q_{r,s,t}$ the vector space of polynomials of degree at most $r$ in the first variable, at most $s$ in the second and at most $t$ in the third one (analogously for th case $d=2$), $\mathbb P_u$ the vector space of polynomials of degree at most $u$. We may write $\mathbb Q_k$ instead of $\mathbb Q_{k,k}$ or $\mathbb Q_{k,k,k}$. Given $E\subset \R^d$, the notation $\operatorname{int} E$ denote its interior.

Let us assume $\Omega$ to be a Lipschitz domain obtained via trimming operations, as in~\eqref{trimmed_domain} with $N=1$, namely $\Omega=\Omega_0\setminus \overline\Omega_1$. Let $\Gamma$ be its boundary such that $\Gamma=\overline\Gamma_D\cup\overline\Gamma_N$, where $\Gamma_D$ and $\Gamma_N$ are non-empty, open, and disjoint. We denote the \emph{trimming curve} (respectively surface if $d=3$) as $\Gamma_{T}=\Gamma \cap \partial \Omega_1$.

 Note that throughout this document $C$ will denote generic constants that may change at each occurrence, but that are always independent of the local mesh size and on the way the background mesh is cut by the trimming curve, unless otherwise specified. The equivalence $cx\le y\le Cx$ is denoted by $x\sim y$.

The Stokes equations are a linear system that can be derived as a simplification of the Navier-Stokes equations. They describe the flow of a fluid under incompressibility and slow motion regimes. Given the body force $\f:\Omega\to\R^d$, the mass production rate $g:\Omega\to\R$, the Dirichlet datum $\u_D:\Omega\to\R$ and the Neumann datum $\bm\sigma_N:\Omega\to\R$, we look for the \emph{velocity} $\u:\Omega\to\R^d$ and \emph{pressure} $p:\Omega\to\R$ such that
\begin{equation}
\begin{aligned}\label{stokes_trimming:stokes}
-\dive \bm\sigma(\u,p) = \f,  \qquad&\text{in}\;\Omega, \\
\dive \u = g,\qquad &\text{in}\;\Omega, \\
\u  = \u_D,\qquad &\text{on}\;\Gamma_D,\\
\bm\sigma(\u,p) \n   = \bm\sigma_N,\qquad &\text{on}\;\Gamma_N,
\end{aligned}
\end{equation}
where $\mu\in\R$, $\mu>0$ is the \emph{viscosity coefficient}, $\bm\sigma(\u,p)= \mu D\u - p \mathbf{I}$ is the \emph{Cauchy stress tensor}. The first equation is known as the \emph{conservation of the momentum} and is nothing else than the Newton's Second Law, relating the external forces acting on the fluid to the rate of change of its momentum, the second one is the \emph{conservation of mass} (when $g\equiv 0$). In what follows, we assume $\mu\equiv 1$ for the sake of simplicity of the notation.
\section{The isogeometric discretization}\label{stokes_trimming:sec:ig}
\subsection{Univariate B-splines}
 For a more detailed introduction to isogeometric analysis, we refer the interested reader to the review article~\cite{MR3202239}. Given two positive integers $k$ and $n$, we say that $\Xi:=\{\xi_1,\dots,\xi_{n+k+1} \}$ is a \emph{k-open knot vector} if
\begin{equation*}
\xi_1=\dots=\xi_{k+1}<\xi_{k+2}\le\dots \le \xi_n<\xi_{n+1}=\dots=\xi_{n+k+1}.
\end{equation*}
We assume $\xi_1=0$ and $\xi_{n+k+1}=1$. We also introduce $Z:=\{\zeta_1,\dots,\zeta_M \}$, the set of \emph{breakpoints}, or knots without repetitions, which forms a partition of the interval $(0,1)$. Note that
$$
\Xi=\{\underbrace{\zeta_1,\dots,\zeta_1}_{m_1\;\text{times}},\underbrace{\zeta_2,\dots,\zeta_2}_{m_2\;\text{times}},\dots,\underbrace{\zeta_M,\dots,\zeta_M }_{m_M\;\text{times}}\},
$$
where $m_j$ is the multiplicity of the breakpoint $\zeta_j$ and $\sum_{i=1}^Mm_i=n+k+1$. Moreover, we assume $m_j\le k$ for every internal knot and we denote $I_i:=(\zeta_i,\zeta_{i+1})$ and its measure $h_i:=\zeta_{i+1}-\zeta_{i}$, $i=1,\dots, M-1$.

We denote as $\hat{B}_{i,k}:[0,1]\to\R$ the i-th \emph{B-spline} of degree $k$, $1\le i\le n$, obtained using the \emph{Cox-de Boor formula}, see for instance~\cite{MR3202239}. Moreover, let $S^k_{\bm \alpha}(\Xi):=\operatorname{span}\{\hat{B}_{i,k}: 1\le i\le n \}$ be the vector space of univariate splines of degree $k$, which can also be characterized as the space of piecewise polynomials of degree $k$ with $\alpha_j:=k-m_j$ continuous derivatives at the breakpoints $\zeta_j$, $1\le j\le M$ (the \emph{Curry-Schoenberg Theorem}). The number of continuous derivatives at the breakpoints are collected in the \emph{regularity vector} $\bm\alpha:=\left(\alpha_j \right)_{j=1}^M$. A knot multiplicity $m_j=k+1$ corresponds to a regularity $\alpha_j=-1$, \emph{i.e.}, a discontinuity at the breakpoint $\zeta_j$. Since the knot vector is open, it holds $\alpha_1=\alpha_M=-1$. For the sake of simplicity of the notation we assume that the basis functions have the same regularity at the internal knots, namely $\alpha_j=\alpha$ for $2\le j\le M-1$. 
 Moreover, given an interval $I_j=\left( \zeta_j,\zeta_{j+1}\right)=(\xi_i,\xi_{i+1})$, we define its \emph{support extension} $\tilde I_j$ as
\begin{equation*}
	\tilde I_j:=\operatorname{int}\bigcup\{\operatorname{supp}(\hat{B}_{\ell,k}): \operatorname{supp}(\hat{B}_{\ell,k})\cap I_j\ne\emptyset, 1\le \ell\le n \}=\left(\xi_{i-k},\xi_{i+k+1}\right).
\end{equation*}

\subsection{Multivariate B-splines}
Let $d\in\{2,3\}$ denote the space dimension and $M_\ell,n_\ell,k_\ell\in\N$, $\Xi_\ell=\{\xi_{\ell,1},\dots,\xi_{\ell,n_l+k+1} \}$, $Z_\ell=\{\zeta_{\ell,1},\dots,\zeta_{\ell,M_\ell} \}$ be given, for every $1\le \ell\le d$. We set the degree vector $\mathbf{k}:=(k_1,\dots,k_d)$, the regularity vectors $\bm\alpha_\ell$, $1\le \ell \le d$, and the multivariate knot-vector $\mathbf{\Xi}:=\Xi_1\times\dots\times\Xi_d$. As in the univariate case, we assume that the same regularity holds at the internal knots for every parametric direction, hence we drop the bold font once for all and write $\alpha_\ell$, $1\le \ell \le d$. Note that the breakpoints of $Z_\ell$ form a Cartesian grid in the parametric domain $\hat \Omega_0 := (0,1)^d$, namely the \emph{parametric Bézier mesh}
\begin{equation*}
\hat{\mathcal{M}}_{0,h}=\{Q_{\mathbf{j}}=I_{1,j_1}\times\dots\times I_{d,j_d}: I_{\ell,j_\ell}=(\zeta_{\ell,j_\ell},\zeta_{\ell,j_\ell+1}): 1\le j_\ell\le M_\ell-1 \},
\end{equation*}
where each $Q_{\mathbf{j}}$ is called \emph{parametric B\'ezier element}, with $h_{Q_\mathbf j}:=\operatorname{diam}(Q_\mathbf j)$. We require the following hypothesis that allows us to assign $h_Q$ as a unique measure to each element and to use the Bramble-Hilbert type results developed in~\cite{MR2250029}.
	\begin{assumption}\label{shape_regularity}
		The family of meshes $\{\hat{\mathcal M}_{0,h}\}_h$ is assumed to be \emph{shape-regular}, that is, the ratio between the smallest edge of $Q\in\hat{\mathcal M}_{0,h}$ and its diameter $h_Q$ is uniformly bounded with respect to $Q$ and $h$.
	\end{assumption}
	\begin{remark}
		The shape-regularity hypothesis implies that the mesh is \emph{locally-quasi uniform}, \emph{i.e.}, the ratio of the sizes of two neighboring elements is uniformly bounded (see \cite{MR2250029}).
\end{remark}
Let $\mathbf{I}:=\{\mathbf{i}=(i_1,\dots,i_d): 1\le i_\ell\le n_\ell \}$ be a set of multi-indices. For each $\i=(i_1,\dots,i_d)$, we define the set of \emph{multivariate B-splines}
$
\{\hat{B}_{\mathbf{i},\mathbf k}(\mathbf{\zeta})=\hat{B}_{i_1,k_1}(\zeta_1)\dots\hat{B}_{i_d,k_d}(\zeta_d): \mathbf{i}\in\mathbf{I}  \}.
$
Moreover, for a generic B\'ezier element $Q_j\in\hat {\mathcal M}_{0,h}$, we define its \emph{support extension} $\tilde Q_j=\tilde I_{1,j_1}\times\dots\times\tilde I_{d,j_d}$,
where $\tilde I_{\ell,j_\ell}$ is the univariate support extension of the univariate case defined above. The \emph{multivariate spline space} in $\hat{\Omega}_0$ is defined as
$
S^{\k}_{\alpha_1,\dots,\alpha_d}(\mathbf{\Xi})=\operatorname{span}\{\hat{B}_{\i,\k}(\zeta):\i\in\mathbf{I} \},
$
which can also be seen as the space of piecewise multivariate polynomials of degree $\k$ and with regularity across the internal B\'ezier elements given by $\alpha_1,\dots,\alpha_d$. Note that $S^{\k}_{\alpha_1,\dots,\alpha_d}(\mathbf{\Xi})=S^{\k}_{\alpha_1,\dots,\alpha_d}(\Xi_1,\dots,\Xi_d)=\bigotimes_{i=1}^d S^{k_i}_{\alpha_i}(\Xi_{i})$.
\begin{remark}
	What has been said so far can be easily generalized to the case of Non-Uniform Rational B-Splines (NURBS) basis functions. See, for instance,~\cite{MR3618875}. 
\end{remark}	
\subsection{Isogeometric spaces for the Stokes problem on trimmed geometries}\label{47}
As already said in the introduction, the principle of isogeometric methods is to assume the physical untrimmed domain $\Omega_0$ to be the image of the unit $d$-dimensional cube through $\mathbf{F}\in \left(S^{\k}_{ \alpha_{1},\dots,\alpha_d}(\mathbf{\Xi}) \right)^d$, namely $\Omega_0=\mathbf{F}(\hat{\Omega}_0)$. Note that the \emph{isogeometric map} $\F$ is given from the CAD description of the geometry. We define the \emph{physical B\'ezier mesh} as the image of the elements in $\hat{\mathcal M}_{0,h}$ through $\F$, $\mathcal M_{0,h}:=\{K:K=\F(Q),\ Q\in\hat {\mathcal M}_{0,h}\}$. We denote $h_K:=\operatorname{diam} (K)$, and the support extension $\tilde K:=\F(\tilde Q)$, for each $K\in\mathcal M_{0,h}$ such that $K=\F(Q)$.
To prevent the existence of singularities in the parametrization we make the following assumption.
\begin{assumption}\label{1}
	The parametrization $\mathbf{F}:\hat{\Omega}_0\to\Omega_0$ is bi-Lipschitz. Moreover, $\restr{\mathbf{F}}{\overline{Q}}\in C^\infty(\overline{Q})$, for every $Q\in\hat{\mathcal{M}}_{0,h}$, and $\restr{\mathbf{F}^{-1}}{\overline{K}}\in C^\infty(\overline{K})$, for every $K\in\mathcal{M}_{0,h}$.
\end{assumption}
Some consequences of Assumption \ref{1} are the following.
\begin{enumerate}
	\item $h_Q\approx h_K$, \emph{i.e.}, there exist $C_1>0, C_2>0$ such that $C_1 h_K\le h_Q \le C_2 h_K$, for every $K\in\mathcal M_{0,h}$, $Q\in\hat{\mathcal M}_{0,h}$, $K=\F(Q)$.
	\item There exists $C>0$ such that, for every $Q\in\hat{ \mathcal{M}}_{0,h}$ such that $\mathbf{F}(Q)=K$, it holds $\norm{D\mathbf{F}}_{L^{\infty}(Q)}\le C$ and $\norm{D\mathbf{F}^{-1}}_{L^{\infty}(K)}\le C$.
	\item There exist $C_1>0, C_2>0$  such that, for every ${\bm \zeta} \in \hat \Omega_0$, $C_1\le\abs{\operatorname{det}(D\mathbf F( {\bm \zeta}))}\le C_2$.
\end{enumerate}
Moreover, note that Assumption~\ref{1} implies that if the parametric mesh is shape-regular,
then the physical mesh is shape-regular too. With an abuse of notation, we may denote
$h := \max_{K\in\mathcal M_{0,h}} h_K$, even if the same symbol has been used for the maximum diameter of the
parametric mesh. 

To define the isogeometric spaces in the parametric domain, we need to resort to the following cumbersome notation. For every $1\le \ell\le d$, let $\Xi_\ell$ be a knot vector of degree $k$ and regularity vector $\alpha$, with $\bm\Xi=\Xi_1\times \dots\times \Xi_\ell$ is the knot vector used for the geometry. We construct $\check{\Xi}_\ell$, a knot vector of degree $k+1$ and regularity $\alpha+1$,  from $\Xi_\ell$ by adding one repetition of the first and last knots. By increasing the multiplicity of the internal knots $\check\Xi_\ell$ by one, we obtain $\tilde\Xi_\ell$, a knot vector of degree $k+1$ and regularity $\alpha$.
\begin{align*}
	\hat V_{0,h}^\text{RT}:=&
	\begin{cases}
		S^{k+1,k}_{ \alpha+1,\alpha}(\check\Xi_1,\Xi_2)\times S^{k,k+1}_{\alpha,\alpha+1}(\Xi_1,\check\Xi_2),&\text{if}\;d=2,\\
		S^{k+1,k,k}_{\alpha+1,\alpha, \alpha}(\check\Xi_1,\Xi_2,\Xi_3)\times S^{k,k+1,k}_{\alpha,\alpha+1,\alpha}(\Xi_1,\check\Xi_2,\Xi_3)\times S^{k,k,k+1}_{\alpha,\alpha,\alpha+1}(\Xi_1, \Xi_2,\check\Xi_3),\;\;\;\;\;\qquad\qquad\qquad&\text{if}\;d=3,
	\end{cases}\\
	\hat V_{0,h}^\text{N}:=&
	\begin{cases}
		S^{k+1,k+1}_{\alpha+1,\alpha}(\check\Xi_1,\tilde\Xi_2)\times S^{k+1,k+1}_{\alpha,\alpha+1}(\tilde\Xi_1,\check\Xi_2),&\text{if}\;d=2,\\
		S^{k+1,k+1,k+1}_{\alpha+1,\alpha,\alpha}(\check\Xi_1,\tilde\Xi_2,\tilde\Xi_3)\times S^{k+1,k+1,k+1}_{\alpha,\alpha+1,\alpha}(\tilde\Xi_1,\check\Xi_2,\tilde\Xi_3)\times S^{k+1,k+1,k+1}_{\alpha,\alpha,\alpha+1}(\tilde\Xi_1,\tilde\Xi_2,\check\Xi_3),&\text{if}\;d=3,
	\end{cases}\\
	\hat V_{0,h}^\text{TH}:=&
	\begin{cases}
		S^{k+1,k+1}_{\alpha,\alpha}(\tilde\Xi_1,\tilde\Xi_2)\times S^{k+1,k+1}_{\alpha,\alpha}(\tilde\Xi_1,\tilde\Xi_2),&\text{if}\;d=2,\\
		S^{k+1,k+1,k+1}_{\alpha,\alpha,\alpha}(\tilde\Xi_1,\tilde\Xi_2,\tilde\Xi_3)\times S^{k+1,k+1,k+1}_{\alpha,\alpha,\alpha}(\tilde\Xi_1,\tilde\Xi_2,\tilde\Xi_3)\times S^{k+1,k+1,k+1}_{\alpha,\alpha,\alpha}(\tilde\Xi_1,\tilde\Xi_2,\tilde\Xi_3),&\text{if}\;d=3,
	\end{cases}\\
	\hat Q_{0,h}:=&
	\begin{cases}
		S^{k,k}_{\alpha,\alpha}(\Xi_1,\Xi_2), \qquad&\text{if}\;d=2,\\
		S^{k,k,k}_{\alpha,\alpha,\alpha}(\Xi_1,\Xi_2,\Xi_3), \qquad&\text{if}\;d=3.
	\end{cases}
\end{align*}
It holds $\hat V_h^\text{RT}\subset \hat V_h^\text{N}\subset \hat V_h^\text{TH}$, see~\cite{MR2808112}.

We observe that the previous construction as well as what follows could be done in a more general setting, by considering different degrees and regularities for each parametric direction.
We also note that, for $\alpha=-1$, $\hat V^\text{RT}_h$ and $\hat V^\text{N}_h$ recover the classical Raviart-Thomas finite element and N\'ed\'elec finite element of the second kind, respectively. For $\alpha=0$, $\hat V_h^\text{TH}$ represents the classical Taylor-Hood finite element space. Henceforth we assume $\alpha\ge 0$, otherwise $\hat V_h^\square$, $\square\in\{\text{RT},\text{N},\text{TH}\}$, is a discontinuous space (of \emph{jump type}) and it does not provide a suitable discretization for the velocity solution of the Stokes problem, since it is not $\bm H^{1}$-conforming.

The isogeometric spaces in the untrimmed domain $\Omega_0$ read as follows:
\begin{equation*}
	\begin{aligned}
		V_{0,h}^{\text{RT}} :=& \{\vv_h:\iota_v \left(\vv_h\right) \in \hat V_{0,h}^\text{RT} \},\;
		V_{0,h}^{\text{N}} := \{\vv_h:\iota_v \left(\vv_h\right) \in \hat V_{0,h}^\text{N} \},\;
		V_{0,h}^{\text{TH}} :=\{\vv_h:\vv_h\circ\F \in \hat V_{0,h}^\text{TH} \},\;\\
		Q^{\text{RT}}_{0,h} =& Q^{\text{N}}_{0,h}:=  \{q_h :\iota_p(q_h)  \in \hat Q_{0,h} \},\; Q^{\text{TH}}_{0,h} := \{q_h :q_h\circ\F  \in \hat Q_{0,h} \},
	\end{aligned}
\end{equation*}
where $\iota_v$ and $\iota_p$ are, respectively, the divergence-preserving and integral-preserving transformations, defined as 
\begin{alignat*}{3}
	\iota_v:&\bm H(\dive;\Omega_0)\to \bm H(\dive;\hat\Omega_0),\qquad  &&\iota_v \left(\vv\right):=\det \left( D\F\right) D\F^{-1} \left( \vv\circ\F\right),\\
	\iota_p: & L^2(\Omega_0)\to L^2(\hat \Omega_0),\qquad &&\iota_p(q):=\det \left( D\F\right)\left( q\circ\F\right).
\end{alignat*}
 Let us restrict them to the the active part of the domain, \emph{i.e.},
\begin{equation*}
	\begin{aligned}
		V_h^\mathrm{RT} :=& \{\restr{\vv_h}{\Omega}:\vv_h\in \ V_{0,h}^{\mathrm{RT}} \},\;
		V_h^\mathrm{N} := \{\restr{\vv_h}{\Omega}:\vv_h \in  V_{0,h}^{\mathrm{N}} \},\;
		V_h^\mathrm{TH} :=\{\restr{\vv_h}{\Omega}:\vv_h \in  V_{0,h}^{\mathrm{TH}} \},\;\\
		Q_h^\mathrm{RT} =Q_h^\mathrm{N}:=& \{\restr{q_h}{\Omega}:q_h \in  Q^\mathrm{RT}_{0,h} \},\; 	Q_h^\mathrm{TH} := \{\restr{q_h}{\Omega}:q_h \in  Q^\mathrm{TH}_{0,h} \}.
	\end{aligned}
\end{equation*}
We observe that in general, in the physical domain, $V_h^\mathrm{RT}\subset V_h^\mathrm{N}\not\subset  V_h^\mathrm{TH}$.

To alleviate the notation, we may omit the superscript $\square\in\{\mathrm{RT},\mathrm{N},\mathrm{TH}\}$ when what said does not depend from the particular isogeometric element choice.

It will be convenient to define the \emph{active parametric B\'ezier mesh} $\hat{\mathcal M}_h=\{Q\in\hat{\mathcal M}_{0,h}: Q\cap\hat\Omega\ne\emptyset\}$, where $\hat\Omega=\mathbf F^{-1}(\Omega)$, and, similarly, the \emph{active physical B\'ezier mesh} $\mathcal M_h=\{K: K=\mathbf F (Q), Q\in\hat{\mathcal M}_h\}$.
For every $K\in\mathcal M_h$, let $h_K:=\operatorname{diam}(K)$ and $h:=\max_{K\in\mathcal M_h}h_K$. We define $\h:\Omega\to (0,+\infty)$ to be the piecewise constant mesh-size function of $\mathcal M_h$, assigning to the active part of each element $K\in\mathcal M_h$ its whole diameter, namely $\restr{\h}{K\cap\Omega}:=h_K$.
The elements whose interiors are cut by the trimming curve (or surface) are denoted as $\mathcal G_h$, namely, $\mathcal G_h:=\{K\in\mathcal M_h: \Gamma_K\ne \emptyset\}$, where $\Gamma_K:=\Gamma\cap K$. 

The following result holds since $\Gamma$ is assumed to be Lipschitz-regular, hence not too oscillating.
\begin{lemma}\label{lemma:boundary}
There exists $C>0$ such that, for every $K\in\mathcal G_h$, it holds $\abs{\Gamma_K}\le Ch_K^{d-1}$.
\end{lemma}
\begin{proof}
See~\cite{MR3806650}.
\end{proof}
We endow the discrete spaces of the velocities with the scalar product
\begin{align*}
	\left( \w_h,\vv_h\right)_{1,h}:= \int_{\Omega}  D\w_h: D\vv_h  +\int_{\Gamma_D}  \mathsf h^{-1}  \w_h\cdot\vv_h, \qquad \w_h,\vv_h\in  V_h,
\end{align*}
inducing the mesh-dependent norm
\begin{alignat*}{3}
	\norm{\vv_h}_{1,h}^2: =& \norm{ D\vv_h}^2_{L^2(\Omega)}+\norm{\h^{-\frac{1}{2}}\vv_h}^2_{L^2(\Gamma_D)},\qquad&&\ \vv_h\in V_h.
\end{alignat*}
We also equip the discrete spaces of the pressures with the mesh-dependent norm
\begin{alignat*}{3}
	\norm{q_h}_{0,h}^2: =& \norm{q_h}^2_{L^2(\Omega)}+\norm{\h^{\frac{1}{2}}q_h}^2_{L^2(\Gamma_D)},\qquad&&\ q_h\in Q_h.
\end{alignat*}
We consider the following \emph{Nitsche's formulations} as discretizations of problem~\eqref{stokes_trimming:stokes}.

Find $\left(\u_h,p_h\right)\in  V_h\times Q_h$ such that
\begin{equation}
	\begin{aligned}\label{stokes_trimming:stability:prob6}
		a_h(\u_h,\vv_h) + b_1(\vv_h,p_h)=F_h (\vv_h), \qquad &\forall \ \vv_h\in  V_h, \\
		b_m(\u_h,q_h)  = G_m(q_h), \qquad & \forall\ q_h\in Q_h,
	\end{aligned}
\end{equation}
where $m\in\{0,1\}$ and
\begin{alignat*}{3}
	a_h(\w_h,\vv_h):=&\int_\Omega D\w_h:D\vv_h-\int_{\Gamma_D} D\w_h\n\cdot\vv_h - \int_{\Gamma_D} \w_h\cdot D\vv_h\n\\
	& + \gamma \int_{\Gamma_D}\mathsf h^{-1} \w_h\cdot\vv_h,  \qquad && \w_h,\vv_h\in  V_h,\\
	b_m(\vv_h,q_h):=& -\int_\Omega q_h \dive\vv_h+m \int_{\Gamma_D} q_h \vv_h\cdot\n, && \vv_h\in V_h,q_h\in Q_h,\\
	F_h(\vv_h):=&  \int_\Omega \f\cdot\vv_h + \int_{\Gamma_N} \u_N\cdot \vv_h- \int_{\Gamma_D} \u_D\cdot D\vv_h\n + \gamma\int_{\Gamma_D}\mathsf h^{-1}  \u_D\cdot\vv_h, \quad && \vv_h\in V_h,\\
	G_m(q_h):=&  -\int_\Omega g q_h+ m\int_{\Gamma_D} q_h \u_D\cdot\n, \ \qquad && q_h\in Q_h,
\end{alignat*}
$\gamma>0$ being a penalty parameter.
\begin{remark}
	In the literature, the Nitsche formulation of the Stokes problem was introduced in~\cite{fcd86cec3aaf40f58936b8118fca3f44} with $m=1$ and allows to weakly impose the Dirichlet boundary conditions without manipulating the discrete velocity space. The choice $m=0$ allows for an exactly divergence-free numerical solution for the velocity field in the case of $g\equiv 0$ and the Raviart-Thomas isogeometric element, see Remark~\ref{stokes_trimming:rmk:not_commuting}.
\end{remark}	
\begin{remark}\label{stokes_trimming:remark:bc}
	We observe that, in order to simplify the presentation, in formulation~\eqref{stokes_trimming:stability:prob6} we impose Dirichlet conditions weakly everywhere. In the case where there is $\tilde \Gamma\subset \Gamma_D$ such that $\F^{-1}(\tilde \Gamma)$ is a union of full faces of $\hat \Omega_0$, then one could have strongly imposed Dirichlet's conditions on $\tilde\Gamma$ by appropriately modifying the discrete velocity spaces: the traces for $V_h^{\mathrm{TH}}$ and the normal components for $V_h^{\mathrm{RT}}$ and $V_h^{\mathrm{N}}$ (the tangential components are weakly imposed in the spirit of~\cite{MR3048532}).
\end{remark}
\begin{remark}
	The imposition of the Neumann boundary conditions does not pose any particular problem in a mesh that is not aligned with $\Gamma_N$. These kinds of conditions are natural for the Stokes problem, \emph{i.e.}, they can be enforced through a boundary integral as long as suitable quadrature rules in the cut elements are available, see~\cite{MR3982623}.
\end{remark}
Motivated by the previous remark, we henceforth assume that $\Gamma_N\cap \Gamma_{T}=\emptyset$, so that $\Gamma_T\subseteq \Gamma_D$, \emph{i.e.}, we impose Dirichlet boundary conditions on the trimming curve.
\section{Lack of stability of the Nitsche method}\label{stokes_trimming:section:lack_stab}
Throughout this section, we want to show with some numerical experiences that the Nitsche formulation of the Stokes problem~\eqref{stokes_trimming:stability:prob6}, discretized with Raviart-Thomas, N\'ed\'elec, and Taylor-Hood elements, lacks stability when working on trimmed geometries. It is well-known (see~\cite{MR972452,MR650055}) that the following are necessary conditions for the well-posedness of formulation~\eqref{stokes_trimming:stability:prob6} for both $m\in\{0,1\}$.
\begin{enumerate}
	\item There exists $\overline\gamma>0$ such that, for every fixed $\gamma\ge\overline\gamma$, there exists $M_a>0$ such that
	\begin{align}
		\abs{a_h(\w_h,\vv_h)}&\le M_a \norm{\w_h}_{1,h}\norm{\vv_h}_{1,h},\,\qquad\qquad\forall\ \w_h,\vv_h\in   V_h\label{stokes_trimming:lackstab:eq1}.
	\end{align}
	\item There exist $\beta_1>0$, $\beta_0>0$ such that
	\begin{align}
		\inf_{q_h\in   Q_h}\sup_{\vv_h\in   V_h}\frac{b_1(\vv_h,q_h)}{\norm{q_h}_{0,h}\norm{\vv_h}_{1,h}}&\ge\beta_1,\label{stokes_trimming:lackstab:eq4_0}\\
		\inf_{q_h\in   Q_h}\sup_{\vv_h\in   V_h}\frac{b_0(\vv_h,q_h)}{\norm{q_h}_{0,h}\norm{\vv_h}_{1,h}}&\ge\beta_0.\label{stokes_trimming:lackstab:eq4_1}
	\end{align}
\end{enumerate}
\begin{remark}
	We observe that $M_a$ depends (and grows dependently) on $\gamma$, which has to be taken sufficiently large, \emph{i.e.}, $\gamma\ge\overline\gamma$, but at the same time as small as possible, \emph{i.e.}, close to $\overline\gamma$, in order not to end up with a too-large continuity constant.
\end{remark}	
We want to show that the stability constants in the previous estimates, $M_a$, $\beta_1$, $\beta_0$, can be arbitrarily negatively influenced by the relative position between the mesh and the trimming curve; hence they are not uniform with respect to the trimming operation. Note that in the following essential boundary conditions are enforced on the whole boundary, and they are weakly imposed on the parts unfitted with the mesh. Let us proceed in order.
\begin{enumerate}
	\item The breakdown example for the robustness of the continuity constant $M_a$ is the following. Let $\Omega_0=\left(0,1\right)^2$, $\Omega_1=\left(0,1\right)\times \left(0.75+\eps,1\right)$ and $\Omega=\Omega_0\setminus \overline\Omega_1$, as illustrated in Figure~\ref{stokes_trimming:lackstab}\subref{stokes_trimming:lackstab:fig3}. Note that the continuity constant of $a_h(\cdot,\cdot)$ corresponding to $\gamma=1$ is smaller than the one related to $\gamma>1$, \emph{i.e.}, $M_a^{\gamma}>M_a^1$ for every $\gamma>1$. Hence, in order to verify that the continuity constant also degenerates with the cut, it is sufficient to show that $M_a^1$ grows as $\eps$ gets smaller. $M_a^1$ can be estimated as the largest eigenvalue of the subsequent generalized eigenvalue problem. 
	
	Find $\left(\u_h,\lambda_h\right)\in  V_h\setminus\{0\}\times\R$ such that
	\begin{align}\label{stokes_trimming:lackstab:eq5}
		a_h(\u_h,\vv_h) = \lambda_h \left(\u_h,\vv_h \right)_{1,h},\qquad\forall\ \vv_h\in {V}_h.
	\end{align}
	\begin{figure}[!ht]
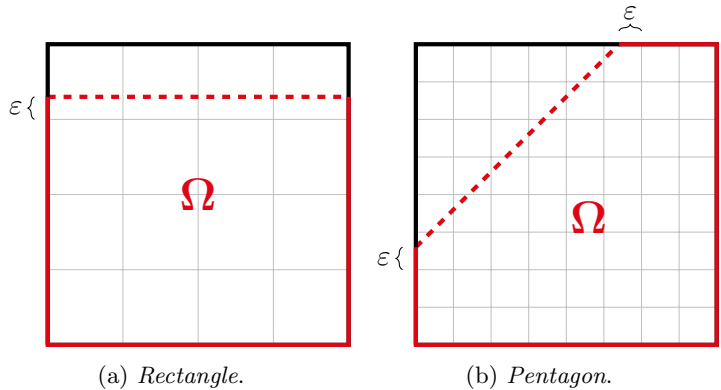

		\centering	
		\subfloat[\emph{Rectangle}.]
		{
			\includegraphics[]{mesh_epsilon.tex}\label{stokes_trimming:lackstab:fig3}
		}
		\subfloat[\emph{Pentagon}.]
		{
			\includegraphics[]{mesh_epsilon_pentagon.tex}\label{stokes_trimming:lackstab:fig1}  
		}
		\caption{The trimmed geometries.}\label{stokes_trimming:lackstab}
	\end{figure}	
	Assume that the mesh is uniform and let us fix the degree $k=2$.
	
	Then, let us compute $\lambda_h^{\max}$ for different values of $\eps$: in each configuration we refine the mesh, see Figure~\ref{stokes_trimming:lackstab:fig3bis}. We can clearly see how the largest eigenvalue grows unboundedly as $\eps$ goes to zero, implying that the continuity constant can be made arbitrarily large by reducing $\eps$. As already observed in the literature in the case of the Poisson problem~\cite{MR4155233}, this is due to the lack of an inverse inequality robust with respect to the trimming operation, namely,
	\begin{equation*}
		\norm{\h^{\frac{1}{2}}D\vv_h\n}_{L^2(\Gamma_K)}\le C \norm{D \vv_h}_{L^2(K\cap\Omega)},
	\end{equation*}
	with $C$ independent of the shape and diameter of $K\cap\Omega$.
	\begin{figure}[!ht]
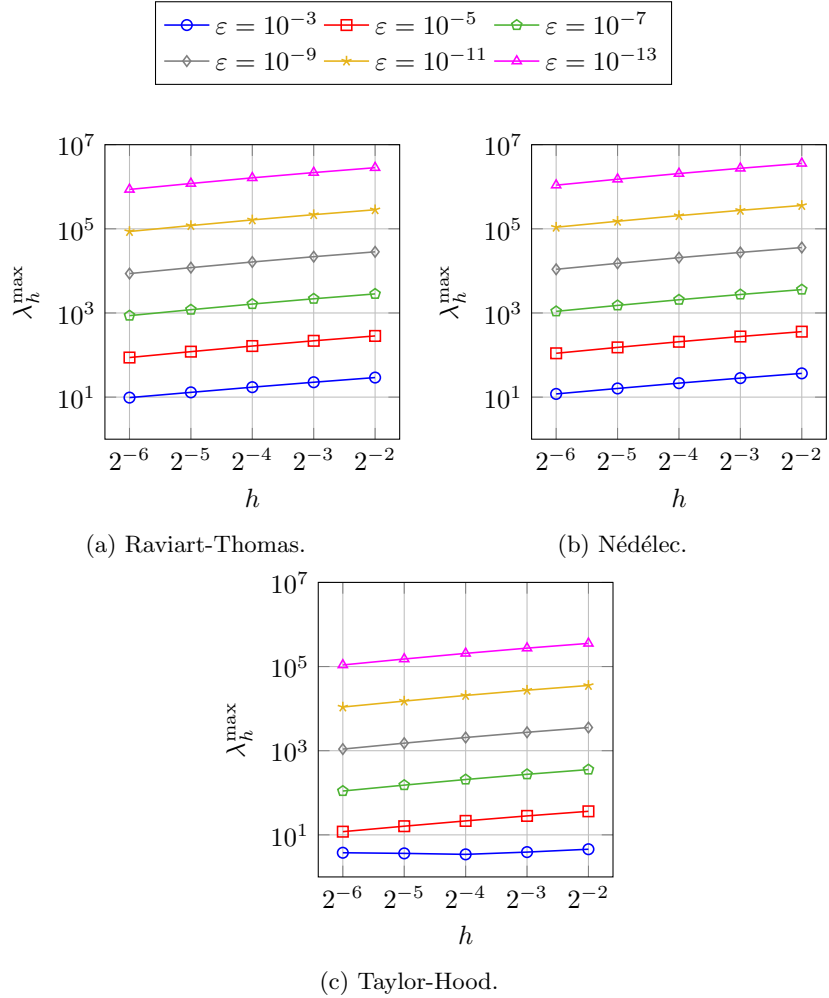

		\centering
		\includegraphics[]{legend_eps.tex}  \\
		\vspace{0.5cm}
		\subfloat[Raviart-Thomas.]
		{	
			\includegraphics[]{stokes_continuity_square_RT.tex}  
		}
		\subfloat[N\'ed\'elec.]
		{	
			\includegraphics[]{stokes_continuity_square_N.tex}  
		} \\
		\subfloat[Taylor-Hood.]
		{	
			\includegraphics[]{stokes_continuity_square_TH.tex}  
		}
		\caption{Maximum generalized eigenvalue of~\eqref{stokes_trimming:lackstab:eq5} for the trimmed \emph{rectangle}.}\label{stokes_trimming:lackstab:fig3bis}
	\end{figure}
	\item Now, we consider a different setting. Let $\Omega_0=\left(0,1\right)^2$, $\Omega_1$ be the triangle with vertices $(0,0.25+\eps)-(0,1)-(0.75-\eps,1)$ and $\Omega=\Omega_0\setminus \overline\Omega_1$, see Figure~\ref{stokes_trimming:lackstab}\subref{stokes_trimming:lackstab:fig1}. We want to study the values of $\beta_m$, $m\in\{0,1\}$, with respect to the trimming parameter $\eps$. The inf-sup constants are numerically evaluated as explained in~\cite{MR1813187}. In Figure~\ref{stokes_trimming:lackstab:fig2} we plot $\beta_m$ for $k=2$ and different values of the trimming parameter $\eps$ and the mesh-size $h$. The numerical experiments show the dependence of $\beta_m$ on $\eps$. This negative result is due to the presence of a \emph{spurious pressure mode} $p_h^\eps$ (technically speaking it is not spurious since, even if $\beta_m(p_h^\eps)\ll 1$, it still holds $\beta_m(p_h^\eps)\neq 0$) whose support is concentrated in trimmed elements with a very small overlap with the physical domain $\Omega$. 
	
	\begin{figure}[!ht]
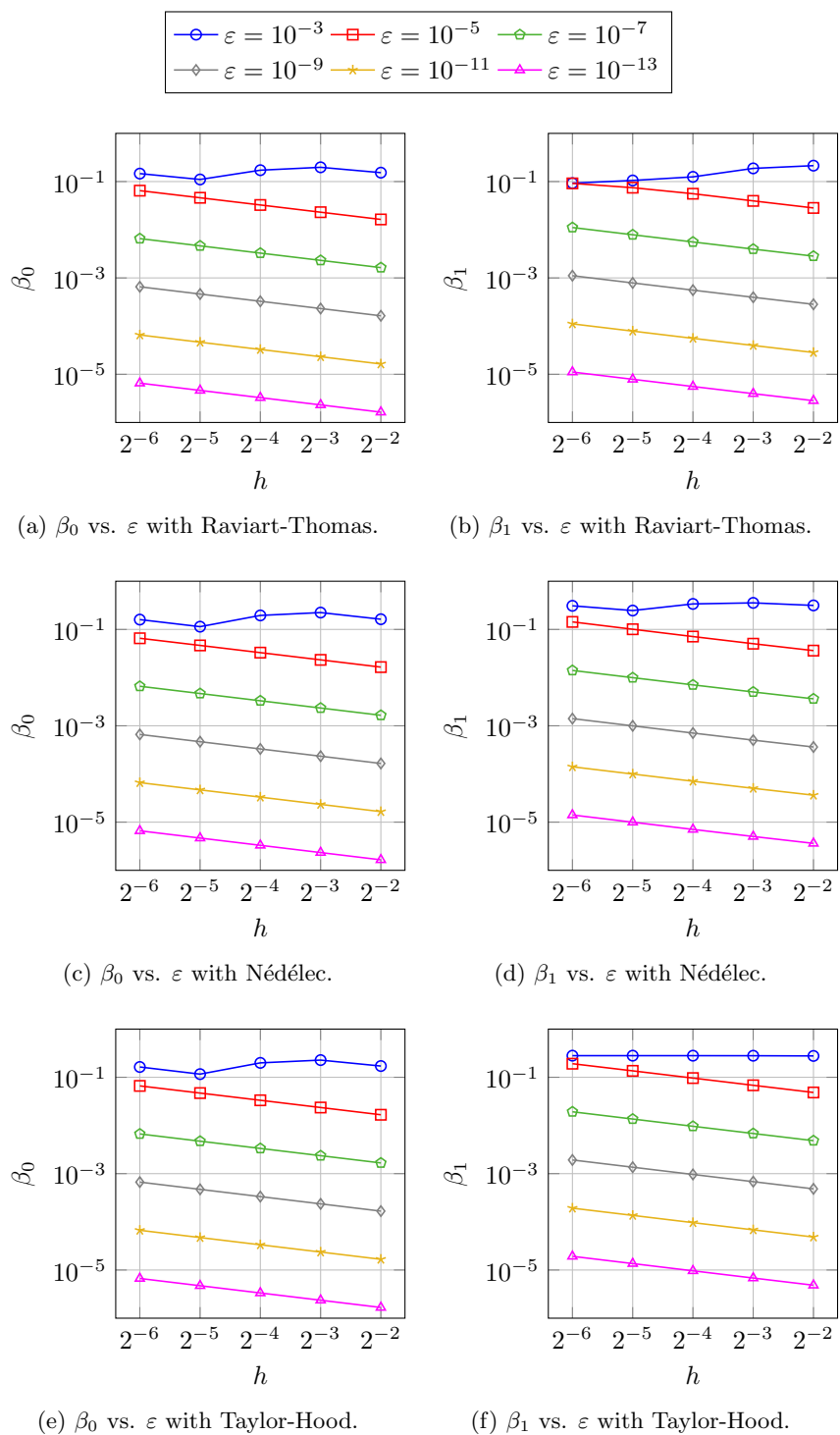

		\centering
		\includegraphics[]{legend_eps.tex}  \\
		\vspace{0.5cm}
		\subfloat[$\beta_0$ vs. $\eps$ with Raviart-Thomas.]
		{	
			\includegraphics[]{inf_sup_b0_pentagon_RT.tex}  
		}
		\subfloat[$\beta_1$ vs. $\eps$ with Raviart-Thomas.]
		{
			\includegraphics[]{inf_sup_b1_pentagon_RT.tex}  	
		} \\
		\vspace{0.5cm}
		\subfloat[$\beta_0$ vs. $\eps$ with N\'ed\'elec.]
		{	
			\includegraphics[]{inf_sup_b0_pentagon_N.tex}  
		}
		\subfloat[$\beta_1$ vs. $\eps$ with N\'ed\'elec.]
		{
			\includegraphics[]{inf_sup_b1_pentagon_N.tex}  	
		} \\
		\vspace{0.5cm}
		\subfloat[$\beta_0$ vs. $\eps$ with Taylor-Hood.]
		{	
			\includegraphics[]{inf_sup_b0_pentagon_TH.tex}  
		}
		\subfloat[$\beta_1$ vs. $\eps$ with Taylor-Hood.]
		{
			\includegraphics[]{inf_sup_b1_pentagon_TH.tex}  	
		} 
		\caption{Inf-sup constants for the trimmed \emph{pentagon}.}\label{stokes_trimming:lackstab:fig2}
	\end{figure}	
\end{enumerate}
\begin{remark}
	Let us observe that in the previous numerical counterexamples, in order to validate the lack of stability of the formulation, one should have constructed a sequence of spaces depending on $\eps$ rather than changing the domain (as it is done in~\cite{MR4155233}, where the lack of stability of the Nitsche formulation for the Poisson problem on trimmed domains is shown). However, both constructions lead to the same results, and we believe that our choice makes the presentation more fluent. We also note that in the first counterexample, the inf-sup condition is not violated, and, similarly, the second configuration is not a counterexample for the continuity.
\end{remark}

\section{The stabilized Nitsche formulation}\label{stokes_trimming:sec:stabilization}
\subsection{Stabilization procedure}
We start by subdividing, for each $h>0$, the elements of the active physical B\'ezier mesh $\mathcal M_{h}$ into two disjoint collections: the one of the good elements $\mathcal M_h^g$, those with sufficient overlap with the physical domain, and the one of the bad elements $\mathcal M_h^b$, a small portion of which intersect $\Omega$. Then, for each bad element $K$, we select a good neighbor $K'$.
\begin{definition}\label{poisson_trimming:def_goodandbad}
	Let $\theta\in (0,1]$ be the \emph{volume-ratio threshold} and $Q\in\hat{\mathcal{M}}_h$. We say that $Q$ is a \emph{good element} if
	\begin{equation*}\label{poisson_trimming:good_bad_def}
		\frac{\abs{\hat\Omega\cap Q}}{\abs{Q}}\ge\theta.
	\end{equation*}
	Otherwise, $Q$ is a \emph{bad element}. Thanks to the regularity Assumption~\ref{shape_regularity} on $\F$, this classification on the parametric elements naturally induces a classification on the physical elements. $\mathcal M_h^g$ stands for the collection of the good physical B\'ezier elements and $\mathcal M_h^b$ for the one of the bad physical elements. Note that $\mathcal M_h \setminus \mathcal G_h \subseteq \mathcal M_h^g$ and $\mathcal M_h^b\subseteq \mathcal G_h$. We denote the set of \emph{neighbors} of $K$ as
	\begin{equation}\label{poisson_trimming:neighbors}
		\mathcal N(K):=\{K'\in\mathcal M_h: \operatorname{dist}\left(K,K'\right)\le C h\}\setminus\{K\},
	\end{equation}
	where $C$ does not depend on the mesh size nor on the trimming configuration.
\end{definition}
The following assumption is not restrictive and is satisfied whenever the mesh is sufficiently refined, and we take $C$ large enough in~\eqref{poisson_trimming:neighbors}.
\begin{assumption}\label{poisson_trimming:assumption_neigh}
	We assume that for any $K\in\mathcal M_h^b$, there exists $K'\in\mathcal N (K) \cap \mathcal M _h^g$. From now on we will refer to such $K'$ as a \emph{good neighbor} of $K$.
\end{assumption}
 We also define $\overline\Omega_{I,h}=\bigcup_{K\in \mathcal M_{h}\setminus \mathcal G_h}\overline K$, the region occupied by untrimmed elements, and $S_h:=  \Omega \setminus \bigcup_{K\in\mathcal M_h^g} \overline K =\operatorname{int} \bigcup_{K\in \mathcal M_h^b} \overline K\cap \overline\Omega$, the region occupied by bad elements.

It is well known that formulation~\eqref{stokes_trimming:stability:prob6} is stable if $\Omega=\Omega_{I,h}$. In the general case $\Omega_{I,h}\subsetneq \Omega$, the goal of the stabilization is, informally speaking, to extend the stability of the discrete problem from the internal elements of the domain to the cut ones.
\begin{remark}
	We observe that choosing $\theta=1$ in Definition~\ref{poisson_trimming:def_goodandbad} corresponds to stabilizing all cut elements, in which case it holds $\mathcal M^b_h =\mathcal G_h$, $\overline \Omega_{I,h} = \bigcup_{K\in\mathcal M_h^g}\overline K$ and $S_h = \Omega\setminus \overline \Omega_{I,h}$.
\end{remark}
Let us start by stabilizing the pressures. We define the operator $R^p_h: Q_h\to L^2(\Omega)$ locally as $\restr{R^p_h(q_h)}{K}:=R^p_K(q_h)$, for every $K\in\mathcal M_h$ and all $q_h\in Q_h$, as follows:
\begin{itemize}
	\item if $K\in\mathcal M_{h}^g$, then
	\begin{equation*}
		R^p_K(q_h):=\restr{q_h}{K},
	\end{equation*} 	
	\item if $K\in\mathcal M_{h}^b$, then
	\begin{equation*}
		R^p_K(q_h):=\restr{ \mathcal E_{K',K}\left(\Pi_{K'}\left(\restr{q_h}{K'}\right)\right)}{K},
	\end{equation*}
	where $\Pi_{K'}:L^2\left(K'\right)\to\mathbb{Q}_k\left(K'\right)$ is the local $L^2$-projection and $\mathcal E_{K',K}:\mathbb Q_{k}(K')\to \mathbb Q_k(K'\cup K)$ is the canonical polynomial extension. $K'$ is a good neighbor of $K$.
\end{itemize}
\begin{proposition}[Stability property of $R^p_h$]\label{stokes_trimming:prop:stab_prop_press}
	Given $\theta\in (0,1]$, there exist $C_1,C_2>0$ such that, for every $K\in\mathcal M_h$ and $h>0$,
	\begin{align*}
		\norm{\mathsf h^{\frac{1}{2}}R^p_h(q_h)}_{L^2(\Gamma_ K)}\le &C_1\norm{q_h}_{L^2(K'\cap\Omega)},\qquad\forall\ q_h\in  Q_h,\\
		\norm{R^p_h(q_h)}_{L^2(K\cap\Omega)}\le &C_2\norm{q_h}_{L^2(K'\cap\Omega)},\qquad\forall\ q_h\in  Q_h,
	\end{align*}
	where $K'$ is a good neighbor if $K\in\mathcal M_{h}^b$, $K'=K$ and if $K\in\mathcal M^g_{h}$.
\end{proposition}	
\begin{proof}
	Let $q_h\in Q_h^\square$, $\square\in\{\mathrm{TH},\mathrm{RT},\mathrm{N}\}$, and $K\in\mathcal M_h$. We first assume $K\in\mathcal M_h^g$ and let $Q=\F^{-1}(K)$, $\hat q_h=q_h\circ\F$ if $\square=\mathrm{TH}$, $q_h=\det (D\F)^{-1}\hat q_h\circ \F^{-1}$ if $\square\in\{\mathrm{RT},\mathrm{N}\}$. H\"older's inequality, Lemma~\ref{appendix:50}, and Lemma~\ref{lemma:boundary} imply
	\begin{equation*}
		\begin{aligned}
			\norm{\h^{\frac{1}{2}}q_h}_{L^2(\Gamma_K)} = &h_K^{\frac{1}{2}} \norm{q_h}_{L^2(\Gamma_K)} \le h_K^{\frac{1}{2}} \abs{\Gamma_K}^{\frac{1}{2}} \norm{q_h}_{L^\infty(\Gamma_K)} \le h_K^{\frac{1}{2}} \abs{\Gamma_K}^{\frac{1}{2}} \norm{q_h}_{L^\infty(K)}\\
			\le & C_\square h_K^{\frac{1}{2}} \abs{\Gamma_K}^{\frac{1}{2}} \norm{\hat q_h}_{L^\infty(Q)}
			\le  C_\square h_K^{\frac{1}{2}} \abs{\Gamma_K}^{\frac{1}{2}} h_K^{-\frac{d}{2}} \norm{\hat q_h}_{L^2(Q\cap\hat\Omega)} \le C C_\square \overline C_\square \norm{q_h}_{L^2(K\cap\Omega)},
		\end{aligned}
	\end{equation*}
	where $C_\mathrm{TH}=1$, $C_\text{RT}=C_\text{N}=\norm{\det D \F^{-1}}_{L^\infty(K)}$, $\overline C_\mathrm{TH}=\norm{\det D \F^{-1}}^{\frac{1}{2}}_{L^\infty(K\cap\Omega)}$, $\overline C_\mathrm{RT}=\overline C_\mathrm{N}=\norm{\det D\F}^{\frac{1}{2}}_{L^\infty(Q\cap\hat\Omega)}$, and $C$ depends on $k$, and on $\theta$.
	Now, let $K\in\mathcal M_h^b$ with good neighbor $K'$. We employ, respectively, H\"{o}lder's inequality, Lemma~\ref{lemma:boundary}, and Lemma~\ref{appendix:51}, and we get
	\begin{align*}
		\norm{\h^{\frac{1}{2}}R_K^p(q_h)}_{L^2(\Gamma_K)} =& h_K^{\frac{1}{2}}\norm{\mathcal E_{K',K}\left(\Pi_{K'}\left(q_h\right) \right)}_{L^2(\Gamma_K)}\le h_K^{\frac{1}{2}}\abs{\Gamma_K}^{\frac{1}{2}}\norm{\mathcal E_{K',K}\left(\Pi_{K'}\left(q_h\right) \right)}_{L^\infty(\Gamma_K)} \\
		\le &C h_K^{\frac{d}{2}}\norm{\mathcal E_{K',K}\left(\Pi_{K'}\left(q_h\right) \right)}_{L^\infty(K)}	\le C  h_K^{\frac{d}{2}}\norm{\Pi_{K'}\left(q_h\right)}_{L^\infty(K')}.
	\end{align*}	 
	We can now use Lemma~\ref{appendix:lemma:schwab}, the boundedness of the $L^2$-projection with respect to $\norm{\cdot}_{L^2}$ and the local quasi-uniformity of the mesh, to obtain
	\begin{align*}
		\norm{\h^{\frac{1}{2}}R_K^p(q_h)}_{L^2(\Gamma_K)} \le & C h_K^{\frac{d}{2}} h_{K'}^{-\frac{d}{2}} \norm{\Pi_{K'}(q_h)}_{L^2(K')} \le C \norm{q_h}_{L^2(K')},
	\end{align*}
	with $C$ depending on $k$. By applying H\"older's inequality, moving to the parametric domain, using Lemma~\ref{appendix:50}, and moving back to the physical domain, we get
	\begin{equation}
		\begin{aligned}\label{stokes_trimming:eq_felipe}
			\norm{\h^{\frac{1}{2}}R_K^p(q_h)}_{L^2(\Gamma_K)} \le & C h_{K'}^{\frac{d}{2}} \norm{q_h}_{L^\infty(K')} \le C_\square h_{K'}^{\frac{d}{2}} \norm{\hat q_h}_{L^\infty(Q')}\\
			\le &C C_\square h_{K'}^{\frac{d}{2}} h_{Q'}^{-\frac{d}{2}}  \norm{\hat q_h}_{L^\infty(Q'\cap\hat\Omega)} \le C C_\square \overline C_\square  \norm{ q_h}_{L^2(K'\cap\Omega)},	
		\end{aligned}
	\end{equation}
	where $C$ depends, in particular, on $k$ and $\theta$, and $C_\square$, $\overline C_\square$ have been defined above.
	
	Let us move to the proof of the other inequality of the statement. If $K\in\mathcal G_h^g$, then there is nothing to prove. Let $K\in\mathcal G_h^b$ and $K'$ its good neighbor.
	\begin{align*}
		\norm{R_h^p(q_h)}_{L^2(K\cap\Omega)} \le & \abs{K\cap\Omega}^{\frac{1}{2}} \norm{R_h^p(q_h)}_{L^\infty(K\cap\Omega)} \le \abs{K\cap\Omega}^{\frac{1}{2}} \norm{R_h^p(q_h)}_{L^\infty(K)}\\
		\le &C \abs{K\cap \Omega}^{\frac{1}{2}} \norm{\Pi_{K'}(q_h)}_{L^\infty(K')},
	\end{align*}	
	where we have used, respectively, Hölder's inequality and Lemma~\ref{appendix:51}. 
	Note that it is trivial to check that, for every $u\in L^2(K')$, $\norm{\Pi_{K'}(u)}_{L^2(K')}\le \norm{u}_{L^2(K')}$. On the other hand, by using Lemma~\ref{appendix:lemma:schwab}, and the $L^2$-stability of the $L^2$-projection, we have $\norm{\Pi_{K'}(q_h)}_{L^\infty(K')}\le C h_{K'}^{-\frac{d}{2}} \norm{\Pi_{K'}(q_h)}_{L^2(K')}\le C h_{K'}^{-\frac{d}{2}}\norm{q_h}_{L^2(K')}$. Hence, 
	\begin{align*}
		\norm{R_h^p(q_h)}_{L^2(K\cap\Omega)} \le  C \abs{K\cap \Omega}^{\frac{1}{2}} h_{K'}^{-\frac{d}{2}}  \norm{q_h}_{L^2(K')}  \le C \abs{K\cap\Omega}^{\frac{1}{2}} \abs{K}^{-\frac{1}{2}}  \norm{q_h}_{L^2(K')}\le C \norm{q_h}_{L^2(K')},
	\end{align*}
	where in the second last passage of both lines we used the shape-regularity and quasi-local uniformity of the mesh, entailing $h_{K'}^{-\frac{d}{2}} \sim \abs{K'}^{-\frac{1}{2}}$, $\abs{K'}^{-\frac{1}{2}}\sim \abs{K}^{-\frac{1}{2}}$. We observe that the constant $C$ depends on $k$, since we relied on Lemma~\ref{appendix:lemma:schwab}. We conclude as in~\eqref{stokes_trimming:eq_felipe}.
\end{proof}
Now, let us move to the velocities and define, for $\square\in\{\mathrm{RT},\mathrm{N},\mathrm{TH}\}$, the operator $R^v_h: V_h^\square\to \bm L^2(\Omega)$ locally as $\restr{R^v_h(\vv_h)}{K}:=R^v_K(\vv_h)$ for every $K\in\mathcal M_h$ and all $\vv_h\in V_h^\square$:
\begin{itemize}
	\item if $K\in\mathcal M_{h}^g$, then
	\begin{equation*}
		R^v_K(\vv_h):=\restr{\vv_h}{K},
	\end{equation*} 	
	\item if $K\in\mathcal M_{h}^b$, then
	\begin{equation*}
		R^v_K(\vv_h):=\restr{\bm{\mathcal E}_{K',K}\left(\bm\Pi_{K'}\left(\restr{\vv_h}{K'}\right)\right)}{K},
	\end{equation*}
	where $\bm \Pi_{K'}:\mathbf L^2\left(K'\right)\to \mathbb V_k(K')$ is the $L^2$-orthogonal projection onto 
	\begin{align*}
		\mathbb V_k(K'):=&
		\begin{cases}
			\mathbb S_k(K'),\qquad\qquad\qquad\qquad\qquad\qquad\qquad\qquad\qquad\,&\text{if}\; \square=\mathrm{RT},\\
			\left(\mathbb Q_{k+1}(K')\right)^d,\qquad&\text{if}\; \square\in \{\mathrm{N},\mathrm{TH}\},\\
		\end{cases}\\
		\mathbb S_k(K'):=&
		\begin{cases}
			\mathbb Q_{k+1,k}(K')\times \mathbb Q_{k,k+1}(K'),\qquad\qquad\qquad\qquad\qquad\,&\text{if}\; d=2,\\
			\mathbb Q_{k+1,k,k}(K')\times \mathbb Q_{k,k+1,k}(K')\times \mathbb Q_{k,k,k+1}(K'),\qquad&\text{if}\; d=3,
		\end{cases}
	\end{align*}
	and $\bm{\mathcal E}_{K',K}: \mathbb V_h(K')\to \mathbb V_h(K\cup K')$ is the canonical polynomial extension. Here, $K'\in\mathcal M_h^g$ denotes a good neighbor of $K$.
\end{itemize}
\begin{proposition}[Stability property of $R^v_h$]\label{stokes_trimming:prop:stab_prop_vel}
	Given $\theta\in (0,1]$, there exists $C>0$ such that, for every $K\in\mathcal M_h$,
	\begin{equation*}
		\norm{\mathsf h^{\frac{1}{2}}DR^v_h(\vv_h)\n}_{L^2(\Gamma_ K)}\le C\norm{D\vv_h}_{L^2(K'\cap\Omega)},\qquad\forall\ \vv_h\in  V_h,
	\end{equation*}
	where $K'\in\mathcal M_h^g$ is a good neighbor of $K$ if $K\in\mathcal M_{h}^b$, $K'=K$ if $K\in\mathcal M_{h}^g$.
\end{proposition}	
\begin{proof}
	We refer the reader to the proof of Theorem~6.8 of~\cite{MR4155233}. The constant $C$ will depend on $\F$ accordingly to the element choice.
\end{proof}
As we saw in Section~\ref{stokes_trimming:section:lack_stab}, due to the unfitted configuration, the Nitsche formulation~\eqref{stokes_trimming:stability:prob6} may present some serious instabilities. Our remedy is twofold. On the one hand, we locally change the evaluation of the normal derivatives of the velocities in the weak formulation; on the other, we modify the space of the discrete pressures.

We introduce the following \emph{stabilized pressure space}
\begin{align*}
	\overline Q_h:=\big\{ \varphi_h\in L^2(\Omega): \exists\ q_h\in Q_h\ \text{such that}\ \restr{\varphi_h}{\Omega\setminus \overline S_h}=\restr{q_h}{\Omega\setminus \overline S_h}\ \text{and}\ \restr{\varphi_h}{S_h} =\restr{R^p_h(q_h)}{S_h} \big\}.
\end{align*}
\begin{remark}\label{stokes_trimming:rmk:dimensions}
	Let us stress that, while $\dim \overline Q_h\le \dim Q_h$, in general, we have that $\overline Q_h$ is not a subspace of $Q_h$ since its elements are discontinuous functions. However, we observe that the discontinuities are located across the facets in the region of bad elements $\overline S_h$ and, for $q_h\in Q_h$ and $R_h^p(q_h)\in \overline Q_h$, it holds
	\begin{align*}
		\restr{q_h}{\Omega\setminus \overline S_h}= 	\restr{R_h^p(q_h)}{\Omega\setminus \overline S_h}. 
	\end{align*}	
\end{remark}
\begin{remark}\label{stokes_trimming:norm_equivalence}
	Proposition~\ref{stokes_trimming:prop:stab_prop_press} entails that $\norm{\cdot}_{0,h}$ and $\norm{\cdot}_{L^2(\Omega\setminus\overline S_h)}$ are equivalent norms on $\overline Q_h$, namely there exist $c_1,c_2>0$, independent on $h$ and on the the way the mesh is cut by trimming curve, but in general depending on the fixed parameter $\theta\in(0,1]$, such that
	\begin{align*}
		c_1 \norm{q_h}_{L^2(\Omega\setminus\overline S_h)} \le \norm{q_h}_{0,h} \le c_2 \norm{q_h}_{L^2(\Omega\setminus\overline S_h)},\qquad\forall\ q_h\in \overline Q_h.
	\end{align*}
\end{remark}
\begin{remark}
	Let $Q_h=\{\restr{B_\k}{\Omega}: \k\in\mathbf K\}$ and define $\tilde Q_h=\{\restr{B_\i}{\Omega}: \i \in \mathbf I\}$, where $\mathbf I:= \{\mathbf i\in \mathbf K: \exists\ K\in\mathcal M_h^g\ \text{such that}\ K\subset \operatorname{supp} B_{\mathbf i}  \}$. This time $\tilde Q_h$ is a subspace of $Q_h$. Moreover, let us observe that $\overline Q_h$ and $\tilde Q_h$ are isomorphic (as normed vector spaces) when equipped with $\norm{\cdot}_{L^2(\Omega\setminus \overline S_h)}$.
\end{remark}
We introduce the following stabilized version of formulation~\eqref{stokes_trimming:stability:prob6}.

Find $\left(\u_h,p_h\right)\in  V_h\times \overline Q_h$ such that
\begin{equation}
	\begin{aligned}\label{stokes_trimming:stability:prob6_stab}
		\overline a_h(\u_h,\vv_h) + b_1(\vv_h,p_h)=\overline F_h(\vv_h),\qquad &\forall \ \vv_h\in  V_h, \\
		b_m(\u_h,q_h)  =G_m(q_h),\qquad& \forall\ q_h\in  \overline Q_h,
	\end{aligned}
\end{equation}
where $m\in\{0,1\}$ and
\begin{alignat*}{3}
	\overline a_h(\w_h,\vv_h):=&\int_{\Omega} D\w_h:D\vv_h - \int_{\Gamma_D}  D R_h^v\left(\w_h\right)\n\cdot\vv_h - \int_{\Gamma_D} \w_h\cdot D R_h^v\left(\vv_h\right)\n \\
	& + \gamma \int_{\Gamma_D}\mathsf h^{-1} \w_h\cdot\vv_h, \quad&& \w_h,\vv_h\in  V_h,\\
	\overline F_h(\vv_h):=&\int_{\Omega}\f\cdot\vv_h +\int_{\Gamma_N}\bm\sigma\cdot\vv_h  - \int_{\Gamma_D}  \u_D\cdot DR_h^v\left(\vv_h\right)\n\ + \gamma \int_{\Gamma_D}\mathsf h^{-1} \u_D\cdot\vv_h,  \quad && \vv_h\in  V_h.
\end{alignat*}
\begin{remark}
	We believe that this strategy is still consistent with the stabilization procedure of~\cite{MR4155233} since the modification does not affect the space of the velocities, but just the one of pressures, the latter being discontinuous objects from a physical point of view.
\end{remark}	

\subsection{Interpolation and approximation properties of the discrete spaces}
From~\cite{MR2923416}, there exist $\bm E:\bm H^{t}(\Omega)\to\bm H^{t}(\R^d)$, $t\ge 1$, and $E:  H^{r}\left(\Omega\right)\to  H^{r}\left(\R^d\right)$, $r\ge 1$, universal (degree-independent) Sobolev-Stein extensions such that $\dive\circ \bm E = E \circ \dive$.  We define, for $\square\in\{\mathrm{RT},\mathrm{N},\mathrm{TH}\}$ and $t\ge 1$,
\begin{align*}
	\Pi_{V_h}^\square: \bm H^{t}(\Omega)\to V_h^\square,\qquad \vv \mapsto \restr{\Pi^{\square}_{V_{0,h}} \left( \restr{\bm E\left(\vv \right)}{\Omega_0} \right)}{\Omega},
\end{align*}
where $\Pi^{\square}_{V_{0,h}}$ is the spline quasi-interpolant onto $V_{0,h}^{\square}$. Similarly, for the pressures, given $r\ge 1$, we introduce
\begin{equation*}
	\begin{aligned}
		\Pi^\square_{Q_h}:& H^{r}\left(\Omega\right)\to Q^\square_h,\qquad q\mapsto \restr{\Pi^\square_{Q_{0,h}} \left( \restr{E\left(q\right)}{\Omega_{0}}    \right)     }{\Omega},
	\end{aligned}
\end{equation*}
and further compose it with the stabilization operator for the pressures, 
\begin{equation*}
	\begin{aligned}
		\overline	\Pi^\square_{Q_h}:& H^{r}\left(\Omega\right)\to\overline Q_h^\square,\qquad q\mapsto R^p_h\left(\Pi^\square_{Q_h} q_h\right),
	\end{aligned}
\end{equation*}
where $\Pi^\square_{Q_{0,h}}$ is the spline quasi-interpolant onto $Q_{0,h}^\square$. Let us recall that, in the Raviart-Thomas case, $\Pi^\mathrm{RT}_{V_{0,h}}$ and $\Pi^\mathrm{RT}_{Q_{0,h}}$ are defined so that the first diagram in~\eqref{stokes_trimming:comm_diag} commutes (see~\cite{MR2792397}). Our construction implies that also the diagram on the right commutes.
\begin{equation}\label{stokes_trimming:comm_diag}
	\begin{CD}
		\bm H(\dive;\Omega_0)  @>\dive>> L^2(\Omega_0) \\
		@VV \Pi_{V_{0,h}}^{\mathrm{RT}}V @VV \Pi_{Q_{0,h}}^\mathrm{RT}V \\
		V_{0,h}^{\mathrm{RT}} @>\dive>>  Q_{0,h}^\mathrm{RT}
	\end{CD}	
	\qquad\quad	
	\begin{CD}
		\bm	H(\dive;\Omega)  @>\dive>> L^2(\Omega) \\
		@VV \Pi_{V_{h}}^\mathrm{RT}V @VV \Pi_{Q_h}^\mathrm{RT}V \\
		V_h^\mathrm{RT} @>\dive>>  Q_h^\mathrm{RT}
	\end{CD}	
\end{equation}	
\begin{remark}\label{stokes_trimming:rmk:not_commuting}
	Note that the commutativity of the right-hand diagram in~\eqref{stokes_trimming:comm_diag} is lost when instead of $Q_h$ we use the stabilized space $\overline Q_h$.
\end{remark}
\begin{proposition}[Approximation property of~$R_h^v$]\label{stokes_trimming:prop:approx_prop_vel}
	There exists $C>0$ such that, for every $\vv\in\bm H^{t}(\Omega)$, $t\ge 2$, and $K\in\mathcal G_h$,
	\begin{align*}
		\norm{\h^{\frac{1}{2}}D\left(\vv-R_h^v( \Pi^\square_{V_h}\vv) \right)\n}_{L^2(\Gamma_D)}\le C h^s \norm{\vv}_{H^{t}(\Omega)},
	\end{align*}
	where $s:=\min\{k,t-1\}$ if $\square = \mathrm{RT}$ and $s:=\min\{k+1,t-1\}$ if $\square\in\{\mathrm{N},\mathrm{TH} \}$.
\end{proposition}
\begin{proof}
	It is sufficient to apply the vectorial version of Proposition~6.9 of~\cite{MR4155233} and to sum over the cut elements in $\mathcal G_h$. The constant $C$ depends on $\F$ accordingly to the element choice.
\end{proof}
\begin{lemma}\label{stokes_trimming:lemma:stab_interp}
	There exists $C>0$ such that
	\begin{align*}
		\norm{ \Pi_{V_h}\vv}_{1,h}\le C\norm{\vv}_{H^1(\Omega)},\qquad\forall\ \vv\in\bm H^1_{0,\Gamma_D}(\Omega).	
	\end{align*}	
\end{lemma}	
\begin{proof}
	Let $\vv \in\bm H^1_{0,\Gamma_D}(\Omega)$. Using the $H^1$-stability for the quasi-interpolant in the boundary-fitted case~\cite{MR2792397}, we have
	\begin{align}\label{stokes_trimming:eq_rognes1}
		\norm{\Pi_{V_h}\vv}^2_{1,h} = &\norm{D \Pi_{V_{0,h}} \left(\bm E(\vv)\right)}^2_{L^2(\Omega)} + \norm{\h^{-\frac{1}{2}} \Pi_{V_{0,h}} \left(\bm E(\vv)\right)}^2_{L^2(\Gamma_D)}\\
		\le& C\norm{\bm E\left( \vv\right)}_{H^1(\Omega_0)} + \sum_{K\in\mathcal G_h}h_K^{-1}\norm{\Pi_{V_{0,h}}\left(\bm E\left(\vv\right)\right)}^2_{L^2(\Gamma_K)}.
	\end{align}
	By using $\restr{\bm E(\vv)}{\Gamma_D}=0$, Lemma~\ref{appendix:lemma:disc_trace_ineq}, and the optimal approximation properties of the quasi-interpolants on boundary-fitted meshes, it holds
	\begin{equation}
		\begin{aligned}\label{stokes_trimming:eq_rognes2}
			\sum_{K\in\mathcal G_h}h_K^{-1}&\norm{\Pi_{V_{0,h}}\left(\bm E\left(\vv\right)\right)}^2_{L^2(\Gamma_K)}=  \sum_{K\in\mathcal G_h}h_K^{-1} \norm{\Pi_{V_{0,h}}\left( \bm E \left(\vv\right)\right) -\bm E(\vv)}^2_{L^2(\Gamma_K)}\\
			\le&  C  \sum_{K\in\mathcal G_h} h_K^{-1}\norm{\Pi_{V_{0,h}}\left( \bm E \left(\vv\right)\right) -\bm E(\vv)}_{L^2(K)}\norm{\Pi_{V_{0,h}}\left( \bm E \left(\vv\right)\right) -\bm E(\vv)}_{H^1(K)}\\
			\le&  C \norm{\bm E\left(\vv\right)}^2_{H^1(\Omega_0)}.	
		\end{aligned}	
	\end{equation}
	We conclude by combining~\eqref{stokes_trimming:eq_rognes1} and~\eqref{stokes_trimming:eq_rognes2}, and using the boundedness of the Sobolev-Stein extension operator.
\end{proof}
\begin{theorem}\label{stokes_trimming:thm:approx}
	There exist $C_v,C_q>0$ such that for every $\left(\vv,q\right)\in  \bm H^{t}(\Omega)\times H^{r}(\Omega)$, $t\ge 1$ and $r\ge 1$, it holds
	\begin{equation*}	
		\begin{aligned}	
			\norm{\vv- \Pi_{V_h}^\square\vv}_{1,h}\le C_v h^s\norm{\vv}_{H^{t}(\Omega)},\qquad \norm{q-\overline\Pi_{Q_h} q}_{0,h}\le C_q h^\ell \norm{q}_{H^{r}(\Omega)},
		\end{aligned}	
	\end{equation*}
	where $s:=\min\{k,t-1\}$ if $\square=\mathrm{RT}$, $s:=\min\{k+1,t-1\}$  if $\square\in\{\mathrm{N},\mathrm{TH}\}$, and $\ell:=\min\{k+1,r\}$.
\end{theorem}	
\begin{proof}
	For the velocities, we proceed by employing the trace inequality of Lemma~\ref{appendix:lemma:disc_trace_ineq} componentwise, the standard approximation properties of $\Pi_{V_{0,h}}$, and the boundedness of the Sobolev-Stein extension operator.
	\begin{align*}
		\norm{\vv- \Pi^\square_{V_h}\vv}^2_{1,h}
		\le& \norm{D\left(\bm E(\vv)- \Pi^{\square}_{V_{0,h}}\bm E(\vv)\right)}^2_{L^2(\Omega_{0})}\\
		&+C\sum_{K\in\mathcal G_h}\norm{\h^{-1}\left(\bm E(\vv)- \Pi^{\square}_{V_{0,h}}\bm E(\vv)\right)}_{L^2(K)}\norm{\bm E(\vv)- \Pi^{\square}_{V_{0,h}}\bm E(\vv)}_{H^1(K)}\\
		\le &Ch^{2s}\norm{\bm E(\vv)}^2_{H^{t}(\Omega_0)}\\
		&+C\sum_{K\in\mathcal G_h}\norm{\h^{-1} \left(\bm E(\vv)- \Pi^{\square}_{V_{0,h}}\bm E(\vv)\right)}_{L^2(K)}
		\norm{\bm E(\vv)- \Pi^{\square}_{V_{0,h}}\bm E(\vv)}_{H^1(K)}\\
		\le& C h^{2s}\norm{\vv}^2_{H^{t}(\Omega)} + \norm{\h^{-1}\left(\bm E(\vv)- \Pi^\square_{V_h}\vv\right)}_{L^2(\Omega_0)}
		\norm{\bm E(\vv)- \Pi^\square_{V_h}\vv}_{H^1(\Omega_0)}\\
		\le& C  h^{2s}\norm{\vv}^2_{H^{t}(\Omega)},
	\end{align*}
	where $s:=\min\{k,t-1\}$ if $\square=\mathrm{RT}$, $s:=\min\{k+1,t-1\}$  if $\square\in\{\mathrm{N},\mathrm{TH}\}$. For the pressure term, we have
	\begin{align}
		\norm{q-\overline \Pi_{Q_h} q}^2_{0,h}= \norm{\left(q- \overline\Pi_{Q_h} q\right)}^2_{L^2(\Omega)} + \sum_{K\in\mathcal G_h}\norm{\h^{\frac{1}{2}}\left(q- \overline{\Pi}_{Q_h} q\right)}^2_{L^2(\Gamma_K)}.\label{stokes_trimming:eq:interp_press}
	\end{align}	
	For the volumetric term we may proceed analogously to the case of the velocities. Let us focus on the boundary part of~\eqref{stokes_trimming:eq:interp_press} and take $K\in\mathcal M_h^g$. We employ Lemma~\ref{appendix:lemma:disc_trace_ineq}:        
	\begin{align*}
		\norm{\mathsf h^{\frac{1}{2}}\left(E(q)- \overline\Pi_{Q_h} q\right)}^2_{L^2(\Gamma_K)}
		\le& C \norm{E(q)- \Pi_{Q_{0,h}} E\left(q\right)}_{L^2(K)}\norm{\mathsf h\left(E(q)- \Pi_{Q_{0,h}} E\left(q\right)\right)}_{H^1(K)} \\
		\le &C h^{2\ell}\norm{E(q)}_{H^{r}(\tilde K)},
	\end{align*}
	where $\ell:=\min\{k+1,r\}$. Now, let us suppose $K\in\mathcal M_h^b$, with $K'\in\mathcal M_h^g$ its good neighbor. Let $\varphi\in\mathbb Q_k(B_K)$, where $B_K$ is the minimal bounding box enclosing $K$ and $K'$, so that $R_K^p(\varphi)=\varphi$. We have
	\begin{align*}
		\norm{\mathsf h^{\frac{1}{2}}\left(q- \overline\Pi_{Q_h} q\right)}_{L^2(\Gamma_K)}=& \norm{\mathsf h^{\frac{1}{2}}\left(q-R^p_K\left(\Pi_{Q_h} q\right)\right)}_{L^2(\Gamma_K)}\\
		\le& \underbrace{\norm{\mathsf h^{\frac{1}{2}}\left(q-\varphi\right)}_{L^2(\Gamma_K)}}_{\RomanNumeralCaps 1}+\underbrace{\norm{\mathsf h^{\frac{1}{2}}R^p_K\left(\varphi-\Pi_{Q_h} q\right)}_{L^2(\Gamma_K)}}_{\RomanNumeralCaps 2}.\label{stokes_trimming:eq:approx_eq5}
	\end{align*}
	By using Lemma~\ref{appendix:lemma:disc_trace_ineq}, we obtain
	\begin{equation*}\label{stokes_trimming:eq:approx_eq4}
		\begin{aligned}
			\RomanNumeralCaps 1\le&C \norm{E(q)-\varphi}^{\frac{1}{2}}_{L^2(K)}	\norm{\mathsf h\left(E(q)-\varphi\right)}^{\frac{1}{2}}_{H^1(K)} \le C\norm{E(q)-\varphi}^{\frac{1}{2}}_{L^2(B_K)}	\norm{\mathsf h\left(E(q)-\varphi\right)}^{\frac{1}{2}}_{H^1(B_K)} .
		\end{aligned}
	\end{equation*}
	On the other hand, from Proposition~\ref{stokes_trimming:prop:stab_prop_press} and triangular inequality we have
	\begin{equation*}\label{stokes_trimming:the_bay}
		\begin{aligned}
			\RomanNumeralCaps2=&\norm{\mathsf h^{\frac{1}{2}}R_K^p\left(\varphi-\Pi_{Q_h} q\right)}_{L^2(\Gamma_K)}\le C \norm{\left(\varphi-\Pi_{Q_h} q\right)}_{L^{2}( K')} \\
			\le &C\left( \norm{\varphi-q}_{L^{2}( K')}+ \norm{q-\Pi_{Q_h} q}_{L^{2}( K')} \right)\\
			\le& C\left( \norm{\varphi-E(q)}_{L^{2}(B_K)}+ \norm{E(q)-\Pi_{0,Q_h}\left(E\left( q\right)\right)}_{L^{2}(K')} \right).
		\end{aligned}
	\end{equation*}
	Let us choose $\varphi$ such that the Deny-Lions Lemma (Theorem 3.4.1 of~\cite{MR1299729}) holds on $B_K$ and use the optimal approximation properties of $\Pi_{Q_{0,h}}$. Thus
	\begin{equation}\label{stokes_trimming:eq:approx_eq2}
		\begin{aligned}	
			\norm{\mathsf h^{\frac{1}{2}}\left(q- \overline\Pi_{Q_h} q\right)}_{L^2(\Gamma_K)} \le \RomanNumeralCaps 1 + \RomanNumeralCaps 2 \le Ch^{\ell}\left( \norm{E(q)}_{H^{r}(B_K)} +   \norm{E(q)}_{H^{r}(\tilde K')}  \right),
		\end{aligned}	
	\end{equation}	
	where $\ell:=\min\{k+1,r\}$ and $C$ depends on the shape-regularity of $B_K$ (through Theorem~3.4.1 of~\cite{MR1299729}), on $\F$, and on the shape-regularity of the parametric B\'ezier mesh (through the approximation properties of $\Pi_{Q_{0,h}}$). Hence, we conclude by taking the sum over the cut elements. The final constant will depend on $k$, $d$, on the constant appearing in~\eqref{poisson_trimming:neighbors}, on the shape-regularity of the parametric mesh, on $\F$, and on the boundedness of the Sobolev-Stein extension.
\end{proof}	
\section{Well-posedness of the stabilized formulation}\label{stokes_trimming:section5}
The following result gives the necessary and sufficient conditions for the existence, uniqueness and stability of the solution of~\eqref{stokes_trimming:stability:prob6_stab}. Let us denote $K_{m}:=\{\vv_h\in  V_h: b_m(\vv_h,q_h)=0\quad\forall\ q_h\in   \overline Q_h\}$, for $m=0,1$. Even if not explicitly stated in order to keep the notation lighter, the following stability constants are required to be independent of the mesh-size $h$ and on the way $\mathcal M_h$ has been cut by $\Gamma_T$.
\begin{proposition}\label{stokes_trimming:prop:ex_uniq_stab}
	Let us  fix $m\in\{0,1\}$, \emph{i.e.}, we choose either the symmetric or the non-symmetric version of~\eqref{stokes_trimming:stability:prob6_stab}.
	\begin{enumerate}[(i)]
		\item There exists $\overline\gamma>0$ such that, for every $\gamma\ge\overline\gamma$, there exists $M_a>0$ such that
		\begin{align}\label{stokes_trimming:eq:continuity_a} 
			\abs{\overline a_h(\w_h,\vv_h)}\le M_a \norm{\w_h}_{1,h}\norm{\vv_h}_{1,h},\qquad\forall\ \w_h,\vv_h\in   V_h.	
		\end{align}
		\item There exist $M_{b_1}>0$, $M_{b_0}>0$ such that
		\begin{alignat}{3}
			\abs{b_1(\vv_h,q_h)}&\le M_{b_1} \norm{\vv_h}_{1,h}\norm{q_h}_{0,h},\qquad&&\forall\ \vv_h\in   V_h,\forall\ q_h\in  \overline Q_h, \label{stokes_trimming:eq:continuity_b1} \\
			\abs{b_0(\vv_h,q_h)}&\le M_{b_0} \norm{\vv_h}_{1,h}\norm{q_h}_{0,h},\qquad&&\forall\ \vv_h\in   V_h,\forall\ q_h\in  \overline Q_h. \label{stokes_trimming:eq:continuity_b0}
		\end{alignat}
		\item There exist $\overline\gamma>0$ and $\alpha_m>0$ such that, for every $\gamma\ge\overline\gamma$, it holds
		\begin{equation}\label{stokes_trimming:eq:coerc1}
			\inf_{\vv_h\in K_m}\sup_{\w_h\in K_1}\frac{\overline a_h(\w_h,\vv_h)}{\norm{\w_h}_{1,h}\norm{\vv_h}_{1,h}}\ge\alpha_m,
		\end{equation}
		and, for all $\w_h\in K_1\setminus \{0\}$,
		\begin{equation}\label{stokes_trimming:eq:coerc2}
			\sup_{\vv_h\in K_m} \overline a_h(\w_h,\vv_h)>0.
		\end{equation}
		\item There exist $\beta_1>0$, $\beta_0>0$ such that
		\begin{align}
			\inf_{q_h\in   \overline Q_h}\sup_{\vv_h\in   V_h}\frac{b_1(\vv_h,q_h)}{\norm{q_h}_{0,h}\norm{\vv_h}_{1,h}}&\ge\beta_1,\label{stokes_trimming:eq:infsup_b1}\\
			\inf_{q_h\in   \overline Q_h}\sup_{\vv_h\in   V_h}\frac{b_0(\vv_h,q_h)}{\norm{q_h}_{0,h}\norm{\vv_h}_{1,h}}&\ge\beta_0.\label{stokes_trimming:eq:infsup_b0}
		\end{align}
		Conditions~$(i)-(iv)$ hold if and only if there exists a unique solution $\left(\u_h,q_h\right)\in  V_h\times  \overline Q_h$ to~\eqref{stokes_trimming:stability:prob6_stab}. Moreover,
		\begin{equation}\label{stokes_trimming:eq:stab_est}	
			\begin{aligned}
				\norm{\u_h}_{1,h}&\le  \frac{1}{\alpha}\norm{\overline F_h}_{-1,h} + \frac{1}{\beta_m}\left(\frac{M_a}{\alpha}+1 \right)\norm{G_m}_{-0,h}, \\
				\norm{p_h}_{0,h}&\le\frac{1}{\beta_1}\left(1+ \frac{ M_a}{\alpha_m} \right)\norm{\overline F_h}_{-1,h} +  \frac{ M_a}{\beta_m\beta_1}\left(\frac{M_a}{\alpha_m}+1 \right)\norm{G_m}_{-0,h},
			\end{aligned}
		\end{equation}
		where $\norm{\cdot}_{-1,h}$ and $\norm{\cdot}_{-0,h}$ denote the dual norms with respect to $\norm{\cdot}_{1,h}$ and $\norm{\cdot}_{0,h}$, respectively.
	\end{enumerate}
\end{proposition}	
\begin{proof}
	We refer the interested reader to~\cite{MR972452,MR650055}.
\end{proof}	
\begin{remark}
	We observe that condition~\eqref{stokes_trimming:eq:coerc2} can be replaced by $\dim K_m = \dim K_1$. If $m=1$, then conditions~\eqref{stokes_trimming:eq:coerc1} and~\eqref{stokes_trimming:eq:coerc2} can be summarized in the coercivity of $\overline a_h(\cdot,\cdot)$ on $K_1$. Moreover, if $g\equiv 0$, then we are no more bound to satisfy~\eqref{stokes_trimming:eq:infsup_b0} when $m=0$.	
\end{remark}	
\begin{lemma}\label{stokes_trimming:lemma:global_coercivity}
	There exist $\overline\gamma>0$ and $\alpha>0$ such that, for every $\gamma\ge\overline\gamma$, it holds
	\begin{align*}
		\alpha \norm{\vv_h}^2_{1,h}\le \overline a_h(\vv_h,\vv_h),\qquad\forall\ \vv_h\in V_h,
	\end{align*}
	and, for every $\gamma\ge\overline\gamma$, there exists $M_a>0$ such that
	\begin{align*}
		\abs{\overline a_h(\w_h,\vv_h)}\le M_a \norm{\w_h}_{1,h}\norm{\vv_h}_{1,h},\qquad\forall\ \w_h,\vv_h\in   V_h.
	\end{align*}		
\end{lemma}	
\begin{proof}
	This proof is based on Proposition~\ref{stokes_trimming:prop:stab_prop_vel} and follows the same lines of Theorem~5.3 of~\cite{MR4155233}.	
\end{proof}	
During the review process of the manuscript to which this chapter refers (see~\cite{puppi_stokes}), we encountered an error in the proof of the conditions~\eqref{stokes_trimming:eq:infsup_b1},~\eqref{stokes_trimming:eq:infsup_b0}. Due to the lack of time, we are compelled to require them in the form of the following assumption. The search for suitable techniques to derive such properties will be the subject of a future study. 
\begin{assumption}\label{stokes_trimming:assumption_bernarda}
	Given $\theta\in (0,1]$, there exist $\beta_0>0$ and $\beta_1>0$ such that	
	\begin{align*}
		\inf_{q_h\in   \overline Q_h}\sup_{\vv_h\in   V_h}\frac{b_1(\vv_h,q_h)}{\norm{q_h}_{0,h}\norm{\vv_h}_{1,h}}\ge\beta_1,\qquad \inf_{q_h\in   \overline Q_h}\sup_{\vv_h\in   V_h}\frac{b_0(\vv_h,q_h)}{\norm{q_h}_{0,h}\norm{\vv_h}_{1,h}}\ge\beta_0.
	\end{align*}
\end{assumption}	
Section~\ref{stokes_trimming:section:numerical_experiments} includes numerical experiments testing and confirming the validity of Assumption~\ref{stokes_trimming:assumption_bernarda}.
\begin{theorem}\label{stokes_trimming:thm:existence_uniqueness}
	Let us require that Assumption~\ref{stokes_trimming:assumption_bernarda} holds. For $m\in\{0,1\}$, given $\theta\in (0,1]$, there exists a unique solution $\left(\u_h,p_h\right)\in  V_h^\square\times \overline Q_h$ of~\eqref{stokes_trimming:stability:prob6_stab} satisfying the stability estimates~\eqref{stokes_trimming:eq:stab_est}.
\end{theorem}
\begin{proof}
	It suffices to verify the hypotheses of Proposition~\ref{stokes_trimming:prop:ex_uniq_stab}. Conditions~\eqref{stokes_trimming:eq:continuity_a},~\eqref{stokes_trimming:eq:coerc1},~\eqref{stokes_trimming:eq:coerc2} are implied by Lemma~\ref{stokes_trimming:lemma:global_coercivity}. The continuity bounds~\eqref{stokes_trimming:eq:continuity_b1},~\eqref{stokes_trimming:eq:continuity_b0} readily follow from the definitions of $\norm{\cdot}_{1,h}$ and $\norm{\cdot}_{0,h}$. Finally, conditions~\eqref{stokes_trimming:eq:continuity_b1},~\eqref{stokes_trimming:eq:continuity_b0} hold because required by Assumption~\ref{stokes_trimming:assumption_bernarda}. 	
\end{proof}	
\section{\emph{A priori} error estimates}\label{stokes_trimming:sec:apriori}
The goal of this section is to demonstrate that the errors, for both the velocity and pressure fields, achieve optimal \emph{a priori} convergence rates in the topologies induced by the norms $\norm{\cdot}_{1,h}$ and $\norm{\cdot}_{0,h}$, respectively.
\begin{lemma}\label{stokes_trimming:lemma:apriori}
	Let us require that Assumption~\ref{stokes_trimming:assumption_bernarda} holds.	Let $\left(\u,p\right)\in \bm H^{\frac{3}{2}+\eps}(\Omega)\times H^1(\Omega)$, $\eps>0$, and $\left(\u_h,p_h\right)\in  V_h \times \overline Q_h$  be the solutions of~\eqref{stokes_trimming:stokes} and~\eqref{stokes_trimming:stability:prob6_stab} with $m\in\{0,1\}$. Then, for every $\left(\u_I,p_I \right)\in  V_h\times  \overline Q_h$ the following estimates hold.
	\begin{align*}
		\norm{\u_h - \u_I}_{1,h} \le& \frac{1}{\alpha}\left(M_a \norm{\u-\u_I}_{1,h} +\norm{\h^{\frac{1}{2}}D\left(\u-R_h^v(\u_I)\right)\n}_{L^2(\Gamma_D)} + M_{b_1}\norm{p-p_I}_{0,h}  \right)\\
		&+ \frac{1}{\beta_m} \left(1+\frac{M_a}{\alpha}\right) M_{b_m}\norm{\u-\u_I}_{1,h}, \\
		\norm{p_h-p_I}_{0,h} \le & \frac{1}{\beta_1}\left(1+\frac{M_a}{\alpha} \right) \Big(M_a \norm{\u-\u_I}_{1,h} +\norm{\h^{\frac{1}{2}}D\left(\u-R_h^v(\u_I)\right)\n}_{L^2(\Gamma_D)}\\
		& + M_{b_1}\norm{p-p_I}_{0,h} \Big) + \frac{M_a}{\beta_m \beta_1}\left(1+\frac{M_a}{\alpha}\right) M_{b_m}\norm{\u-\u_I}_{1,h}.
	\end{align*}
\end{lemma}
\begin{proof}
	Let $m\in\{0,1\}$ and $\left(\u_I,p_I\right)\in V_h\times\overline Q_h$ be arbitrary. By linearity $\left(\u_h-\u_I,p_h-p_I \right)\in  V_h\times \overline Q_h$ satisfies the saddle point problem
	\begin{equation}\label{stokes_trimming:eq:weak_cons}
		\begin{aligned}
			\overline a_h(\u_h - \u_I,\vv_h) + b_1(\vv_h,p_h-p_I) =  F_I(\vv_h), \qquad&\forall\ \vv_h\in V_h,\\
			b_m(\u_h-\u_I,q_h) =  G_{I,m}(q_h),\qquad&\forall\ q_h\in  \overline Q_h,
		\end{aligned}
	\end{equation}
	where
	\begin{alignat*}{3}
		F_I(\vv_h):= & \int_{\Omega}( D \left( \u-\u_I\right):D \vv_h  - \int_{\Gamma_D}  D\left(\u - R_h^v(\u_I) \right)\n\cdot \vv_h+ b_1(\vv_h,p-p_I)\\
		& - \int_{\Omega}  \left(\u-\u_I\right)\cdot D R_h^v(\vv_h)\n  + \gamma\int_{\Gamma_D}\h^{-1} \left(\u-\u_I\right)\cdot\vv_h,\qquad&&\vv_h\in V_h,\\
		G_{I,m}(q_h):=&b_m(\u-\u_I,q_h),&&q_h\in \overline Q_h.
	\end{alignat*}
	For the sake of completeness, let us show the first line of~\eqref{stokes_trimming:eq:weak_cons}. Note that the second line follows immediately. Recall from~\eqref{stokes_trimming:stokes} and~\eqref{stokes_trimming:stability:prob6} that $F(\vv_h) = \int_\Omega  D\u: D\vv_h -\int_{\Gamma_D}  D\u\n\cdot \vv_h + b_1(\vv_h,p)$ and $\restr{\u}{\Gamma_D}=\u_D$.
	Hence
	\begin{align*}
		\overline a_h(\u_h-\u_I,\vv_h)& + b_1(\vv_h,p_h-p_I) = \overline F_h(\vv_h)-\overline a_h(\u_I,\vv_h)-b_1(\vv_h,p_I)\\
		=& F(\vv_h) - \int_{\Omega} \u_D\cdot D R_h^v(\vv_h)\n
		+\gamma\int_{\Gamma_D} \h^{-1}\u_D\cdot\vv_h - \int_{\Omega}  D\u_I:D\vv_h \\
		& + \int_{\Gamma_D}  DR_h^v(\u_I)\n\cdot \vv_h
		+\int_{\Gamma_D}  \u_I\cdot D R_h^v(\vv_h)\n -\gamma\int_{\Gamma_D}\h^{-1}  \u_I\cdot \vv_h- b_1(\vv_h,p_I)\\
		=&\int_\Omega  D \left( \u-\u_I\right):D \vv_h 
		- \int_{\Gamma_D}   D\left(\u - R_h^v(\u_I) \right)\n\cdot\vv_h
		+ b_1(\vv_h,p-p_I)\\
		& - \int_{\Gamma_D}   \left(\u-\u_I\right)\cdot  D R_h^v(\vv_h)\n+ \gamma\int_{\Gamma_D} \h^{-1} \left(\u-\u_I\right)\cdot \vv_h.
	\end{align*}
	Using the stability estimates~\eqref{stokes_trimming:eq:stab_est}, respectively for $m=0,1$, we get
	\begin{align*}
		\norm{\u_h - \u_I}_{1,h} \le &\frac{1}{\alpha}\norm{F_I}_{-1,h}+ \frac{1}{\beta_m} \left(1+\frac{M_a}{\alpha}\right) \norm{G_{I,m}}_{-0,h}, \\
		\norm{p_h-p_I}_{0,h} \le & \frac{1}{\beta_1}\left(1+\frac{M_a}{\alpha} \right) \norm{F_I}_{-1,h} 
		+ \frac{M_a}{\beta_m \beta_1}\left(1+\frac{M_a}{\alpha}\right) \norm{G_{I,m}}_{-0,h}.
	\end{align*}
	We conclude since, by definition of dual norm, we have
	\begin{align*}
		\norm{F_I}_{-1,h} \le &M_a\norm{\u-\u_I}_{1,h} +\norm{\h^{\frac{1}{2}}D\left(\u-R_h^v(\u_I)\right)\n}_{L^2(\Gamma_D)}+M_{b_1}\norm{p-p_I}_{0,h},\\
		\norm{G_{I,m}}_{-0,h} \le &M_{b_m}\norm{\u-\u_I}_{1,h}.
	\end{align*}
\end{proof}
\begin{theorem}\label{stokes_trimming:thm:apriori}
	Let us require that Assumption~\ref{stokes_trimming:assumption_bernarda} holds.	Let $\left(\u,p\right)\in \bm H^{t}\left(\Omega\right)\times H^{r}\left(\Omega\right)$, $t\ge 2$ and $r\ge 1$, be the solution to problem~\eqref{stokes_trimming:stokes}. Then, the discrete solution $\left(\u_h,p_h\right)\in  V_h^\square\times \overline Q_h$ of the stabilized problem~\eqref{stokes_trimming:stability:prob6_stab} satisfies
	\begin{equation*}
		\norm{\u-\u_h}_{1,h}+\norm{p-p_h}_{0,h}\le C_m h^{\min\{s,\ell\}}\left(  \norm{\u}_{H^{t}(\Omega)}+\norm{p}_{H^{r}(\Omega)}\right),
	\end{equation*}
	where $s:=\min\{k,t-1\}$ if $\square=\mathrm{RT}$, $s:=\min\{k+1,t-1\}$  if $\square=\mathrm{N}$, and $\ell:=\min\{k+1,r\}$, and $C_m>0$ depends on the choice $m\in\{0,1\}$ through the constants appearing in Lemma~\ref{stokes_trimming:lemma:apriori}.
\end{theorem}	
\begin{proof}
	Given $\left(\u_I,p_I \right)\in V_h^\square\times \overline Q_h$, we proceed by triangular inequality:
	\begin{align}\label{stokes_trimming:eq:apriori1}
		\norm{\u-\u_h}_{1,h}\le& \norm{\u-\u_I}_{1,h}+\norm{\u_h-\u_I}_{1,h},\\
		\norm{p-p_h}_{0,h}\le & \norm{p-p_I}_{0,h}+ \norm{p_h-p_I}_{0,h}.	
	\end{align}	
	Using Lemma~\ref{stokes_trimming:lemma:apriori}, we obtain
	\begin{align*}
		\norm{\u-\u_h}_{1,h}\le&  \norm{\u-\u_I}_{1,h}+\frac{1}{\alpha}\Big(M_a \norm{\u-\u_I}_{1,h} +\norm{\h^{\frac{1}{2}}D\left(\u-R_h^v(\u_I) \right)\n}_{L^2(\Gamma_D)}\\
		&+ M_{b_1}\norm{p-p_I}_{0,h}  \Big) + \frac{1}{\beta_m} \left(1+\frac{M_a}{\alpha}\right) M_{b_m}\norm{\u-\u_I}_{1,h},\\
		\norm{p-p_h}_{0,h}\le &\norm{p_h-p_I}_{0,h} +\frac{1}{\beta_1}\left(1+\frac{M_a}{\alpha} \right) \Big(M_a \norm{\u-\u_I}_{1,h} +\norm{\h^{\frac{1}{2}}D\left(\u-R_h^v(\u_I) \right)\n}_{L^2(\Gamma_D)}\\
		&+ M_{b_1}\norm{p-p_I}_{0,h} \Big)+ \frac{M_a}{\beta_m \beta_1}\left(1+\frac{M_a}{\alpha}\right) M_{b_m}\norm{\u-\u_I}_{1,h}.
	\end{align*}
	Let us choose $\u_I:=\Pi_{V_h}^\square\u$ and $p_I:=\overline\Pi_{Q_h} p$ so that, by Proposition~\ref{stokes_trimming:prop:approx_prop_vel} and Theorem~\ref{stokes_trimming:thm:approx}, we obtain
	\begin{align*}
		\norm{\u-\u_h}_{1,h}\le& C_v h^{s} \norm{\u}_{H^{t}(\Omega)}
		+\frac{1}{\alpha}\max\{M_a,1, M_{b_1}\}C h^{\min\{s,\ell\}}\left(\norm{\u}_{H^{t}(\Omega)} +\norm{p}_{H^{r}(\Omega)}\right)\\
		& + \frac{1}{\beta_m} \left(1+\frac{M_a}{\alpha}\right) M_{b_m} C_v h^{s} \norm{\u}_{H^{t}(\Omega)} ,\\
		\norm{p-p_h}_{0,h}\le &C_q h^{\ell}\norm{p}_{H^{r}(\Omega)}\\
		& +\frac{1}{\beta_1}\left(1+\frac{M_a}{\alpha} \right) \max\{M_a,1, M_{b_1}\}C h^{\min\{s,\ell\}}\left(\norm{\u}_{H^{t}(\Omega)} +\norm{p}_{H^{r}(\Omega)}\right) \\
		&+ \frac{M_a}{\beta_m \beta_1}\left(1+\frac{M_a}{\alpha}\right) M_{b_m}C_v h^{s} \norm{\u}_{H^{t}(\Omega)}.
	\end{align*}
\end{proof}	

\section{Numerical examples}\label{stokes_trimming:sec:numerical_experiments}
The main goal of the following numerical experiments is to validate the convergence results of the Theorem~\ref{stokes_trimming:thm:apriori} and to validate the inf-sup condition that we have not been able to prove theoretically.

To prevent the conditioning number of the linear system from being excessively corrupted by the presence of basis functions whose support barely intersects the physical domain, a left-right Jacobi preconditioner is employed. This approach helps for improving the conditioning, but, as previously discussed in~\cite{MR4155233,MR3610101}, it does not completely solve its dependence on the trimming configuration: the interested reader is referred to~\cite{MR3915341} for a more sophisticated approach.
\subsection{Pentagon}\label{stokes_trimming:pentagon_experiment}
Let us consider as computational domain the pentagon $\Omega=\Omega_0\setminus\overline\Omega_1$, where $\Omega_0=\hat\Omega_0$ and $\Omega_1$ is the triangle of vertices $\left(0,0.25+\eps\right)-\left(0,1\right)-\left(0.75-\eps,1\right)$ as illustrated in Figure~\ref{stokes_trimming:lackstab}\subref{stokes_trimming:lackstab:fig1}. Here $\eps=10^{-13}$. The following functions are chosen as manufactured solutions for the velocity and pressure fields:
\begin{align*}
	\u=\left(xy^3,x^4-\frac{y^4}{4} \right),\qquad p = p_{\text{fun}}- \frac{1}{\abs{\Omega}}\int_{\Omega}p_{\text{fun}}, \;\;\text{where}\;\; p_{\text{fun}}=x^3\cos(x)+y^2\sin(x).
\end{align*}
Dirichlet boundary conditions are weakly enforced on the boundary sides unfitted with the mesh, while on the rest, they are imposed in the strong sense (we recall from Remark~\ref{stokes_trimming:remark:bc} that, for the Raviart-Thomas and N\'ed\'elec element, we need to impose the tangential components in a weak sense). We compare, for different isogeometric elements, the well-posedness and accuracy of the non-symmetric, \emph{i.e.}, with $m=0$, non-stabilized and stabilized formulations,~\eqref{stokes_trimming:stability:prob6} and~\eqref{stokes_trimming:stability:prob6_stab} respectively, for $k=2$ and $\gamma=20 \left(k+1 \right)^2$ (the dependency of the penalty parameter on the degree is coherent with~\cite{MR3010180}). The threshold parameter $\theta$ is set equal to $1$,\emph{i.e.}, all cut elements are stabilized.

In Table~\ref{stokes_trimming:num_exp:tab1} we see the values of the inf-sup constants $\beta_0$, $\beta_1$, computed as in~\cite{MR1813187}, in the non-stabilized and stabilized cases (subscripts \emph{ns} and \emph{s} respectively) for the different choices of the isogeometric element (superscripts $\mathrm{RT}$, $\mathrm{N}$ and $\mathrm{TH}$). In the stabilized case, we observe that the inf-sup constants lost their dependence on how the mesh is trimmed. In Figure~\ref{stokes_trimming:num_exp:fig3} the accuracy of the non-stabilized and stabilized formulations are compared. We observe a clear improvement in the pressure error between the non-stabilized and the stabilized case.
\begin{table}[!ht]
	\centering
	\begin{tabular}{c|c c c c c c}
		$h$           & $2^{-1}$  & $2^{-2}$ & $2^{-3}$ & $2^{-4}$ & $2^{-5}$ & $2^{-6}$   \\
		\hline
		$\beta_{0,ns}^{\mathrm{RT}}$      & 0.2437    &   1.6450e-07     &   2.3014e-07   &  3.2638e-07  &  4.6166e-07     & 6.5221e-07    \\
		$\beta_{1,ns}^{\mathrm{RT}}$       &  0.2699 &   2.8358e-07    &     3.9759e-07   & 5.5849e-07      &  7.8738e-07      &  
		1.1108e-06 \\
		$\beta_{0,s}^{\mathrm{RT}}$   & 0.4103  &  0.1740  &  0.2032  &  0.1850  &  0.1588   & 0.1635   \\     
		$\beta_{1,s}^{\mathrm{RT}}$    & 0.3923  &  0.2088 &   0.2397 &   0.2440 &   0.2441  &  0.2442 \\
		\hline
		$\beta_{0,ns}^{\mathrm{N}}$       &  0.2714   &  1.6541e-07   &   2.3212e-07
		& 3.2900e-07 &  4.6583e-07&   6.5811e-07  \\
		$\beta_{1,ns}^{\mathrm{N}}$       &  0.3178    &  3.6259e-07   &  5.0472e-07  &   7.0780e-07    & 9.9752e-07     &  1.4077e-06  \\
		$\beta_{0,s}^{\mathrm{N}}$       &0.4142    &      0.2430      &   0.2902     &   0.2803     &   0.2809    &   0.2676   \\
		$\beta_{1,s}^{\mathrm{N}}$          &  0.4118   & 0.2564     & 0.2979  &  0.3089    &  0.3096    &  0.3096  \\	
		\hline
		$\beta_{0,ns}^{\mathrm{TH}}$      &  0.2672   & 1.6728e-07  &  2.3504e-07 & 3.3295e-07 & 4.7052e-07&  4.7052e-07 \\
		$\beta_{1,ns}^{\mathrm{TH}}$     & 0.2768    & 4.8359e-07   & 6.8222e-07  &   9.6265e-07    & 1.3581e-06   &   1.9189e-06
		\\
		$\beta_{0,s}^{\mathrm{TH}}$     &  0.3374  &  0.2836  &  0.2853 &   0.2853  &  0.2853   & 0.2853   \\
		$\beta_{1,s}^{\mathrm{TH}}$         & 0.2994  &  0.2755  &  0.2789  &  0.2802  &  0.2807  &  0.2809 
	\end{tabular}
	\caption{Inf-sup constant for the \emph{pentagon}: stabilized vs non-stabilized formulations with $k=2$.}\label{stokes_trimming:num_exp:tab1}
\end{table}
\begin{figure}[!ht]
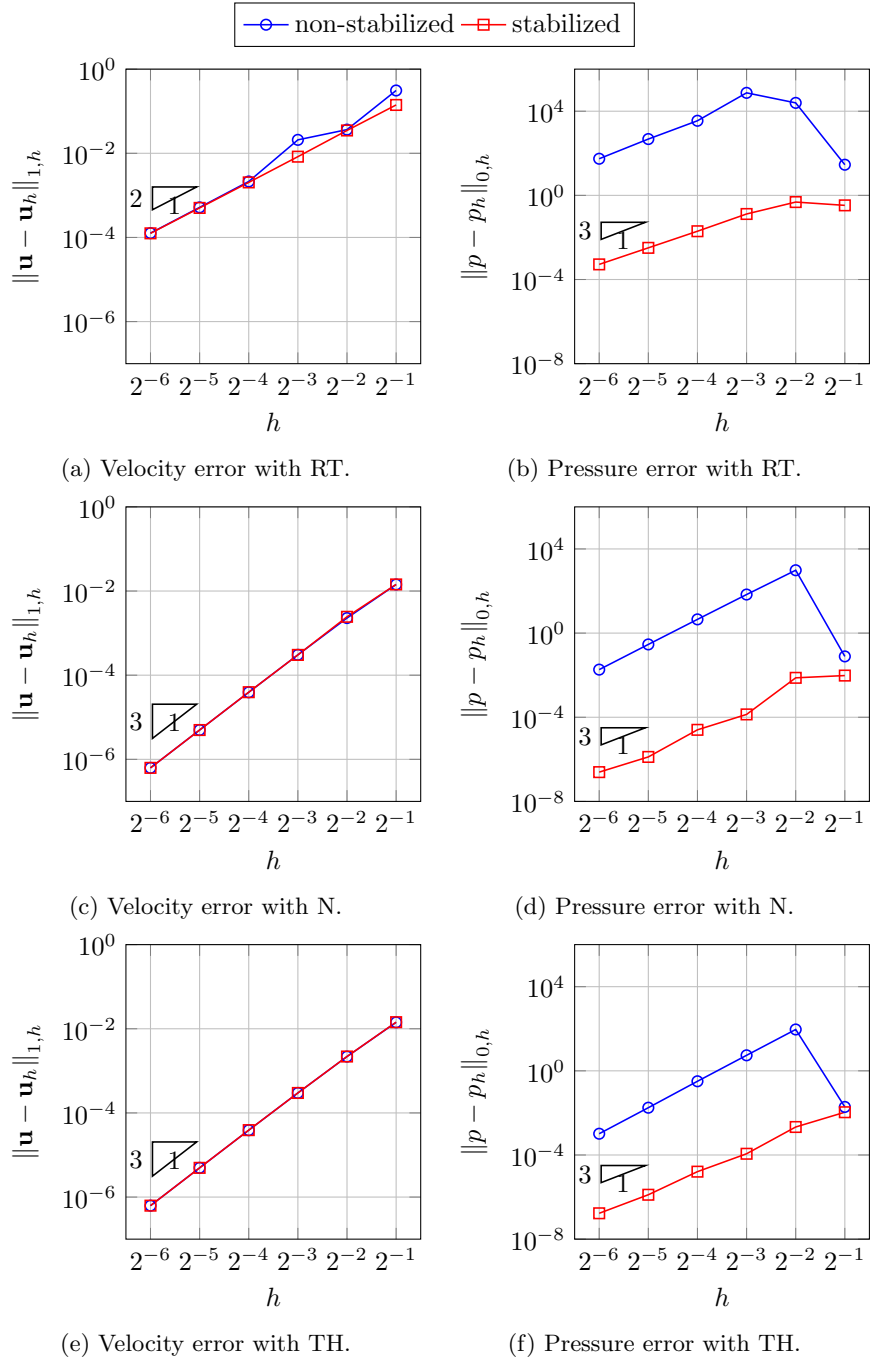

	\centering
	\includestandalone[]{legend_stab_vs_nostab}  \\
	\subfloat[][Velocity error with RT.]
	{
		\includestandalone[]{pentagon_eps_convergence_vel_RT}
	}
	\subfloat[][Pressure error with RT.]
	{
		\includestandalone[]{pentagon_eps_convergence_press_RT}
	} \\
	\subfloat[][Velocity error with N.]
	{
		\includestandalone[]{pentagon_eps_convergence_vel_N}
	}
	\subfloat[][Pressure error with N.]
	{
		\includestandalone[]{pentagon_eps_convergence_press_N}
	} \\
	\subfloat[][Velocity error with TH.]
	{
		\includestandalone[]{pentagon_eps_convergence_vel_TH}
	}
	\subfloat[][Pressure error with TH.]
	{
		\includestandalone[]{pentagon_eps_convergence_press_TH}
	}
	\caption{Convergence errors for the \emph{pentagon} with $k=2$.}\label{stokes_trimming:num_exp:fig3}
\end{figure}

\subsection{Mapped pentagon}
Let us perform an experiment similar to the previous one, this time with a non-linear isogeometric mapping $\F$. We consider $\Omega=\Omega_0\setminus\overline \Omega_1$, where $\Omega_0=\F\left( (0,1)^2\right)$ is the quarter of annulus parametrized by a biquadratic NURBS $\F$, and $\Omega_1=\F(T)$, with $T$ the triangle with vertices $(0,0.25+\eps),(0,1),(0.75-\eps,1)$, see Figure~\ref{stokes_trimming:madre3}\subref{stokes_trimming:fig_mapped_pentagon}. We compare the inf-sup stability of the non-stabilized and the stabilized formulations~\eqref{stokes_trimming:stability:prob6} and~\eqref{stokes_trimming:stability:prob6_stab} respectively, for different degrees and isogeometric elements, $\theta=1$ (we stabilize at all cut elements), and $\eps=10^{-13}$. Dirichlet boundary conditions are imposed on the whole boundary, weakly on the unfitted parts. From Figure~\ref{stokes_trimming:num_exp:fig3bent} we observe that the inf-sup constants of the stabilized formulation behave much better than the ones of the non-stabilized formulation. The order of magnitude of the inf-sup constants in the non-stabilized case are of the same order of the ones in Table~\ref{stokes_trimming:num_exp:tab1}.
\begin{figure}[!ht]
	\centering
	\newcommand{\logLogSlopeTriangleOne}[5]
{

    \pgfplotsextra
    {
        \pgfkeysgetvalue{/pgfplots/xmin}{\xmin}
        \pgfkeysgetvalue{/pgfplots/xmax}{\xmax}
        \pgfkeysgetvalue{/pgfplots/ymin}{\ymin}
        \pgfkeysgetvalue{/pgfplots/ymax}{\ymax}

        \pgfmathsetmacro{\xArel}{#1}
        \pgfmathsetmacro{\yArel}{#3}
        \pgfmathsetmacro{\xBrel}{#1-#2}
        \pgfmathsetmacro{\yBrel}{\yArel}
        \pgfmathsetmacro{\xCrel}{\xArel}

        \pgfmathsetmacro{\lnxB}{\xmin*(1-(#1-#2))+\xmax*(#1-#2)} 
        \pgfmathsetmacro{\lnxA}{\xmin*(1-#1)+\xmax*#1} 
        \pgfmathsetmacro{\lnyA}{\ymin*(1-#3)+\ymax*#3} 
        \pgfmathsetmacro{\lnyC}{\lnyA+#4*(\lnxA-\lnxB)}
        \pgfmathsetmacro{\yCrel}{\lnyC-\ymin)/(\ymax-\ymin)} 

        \coordinate (A) at (rel axis cs:\xArel,\yArel);
        \coordinate (B) at (rel axis cs:\xBrel,\yBrel);
        \coordinate (C) at (rel axis cs:\xCrel,\yCrel);

        \draw[#5]   (A)-- node[pos=0.5,anchor=north] {1} 
                    (B)-- 
                    (C)-- node[pos=0.5,anchor=east] {#4} 
                    cycle;
    }
}


\begin{tikzpicture}

\definecolor{blu}{rgb}{0,0,1}
\definecolor{red}{rgb}{1,0,0}
\definecolor{green}{rgb}{0.3,0.7,0.2}
\definecolor{mygray}{rgb}{0.5,0.5,0.5}
\definecolor{myyellow}{rgb}{0.9,0.7,0.1}
\definecolor{mymagenta}{rgb}{255,0,255}

  \begin{axis}[
  	legend columns=5,
  	width=5.5cm,
  	height=5.5cm,
      hide axis, 
      xmin=10,
      xmax=50,
      ymin=0,
      ymax=0.4,
     legend cell align={left},
      ]

\addlegendimage{semithick, blu, mark=square, mark size=2, mark options={draw=blu}} 
\addlegendentry{$k=1$ stab.};
\addlegendimage{semithick, red, mark=o, mark size=2, mark options={draw=red}} 
\addlegendentry{$k=2$ stab.};
\addlegendimage{semithick, green, mark=triangle, mark size=2, mark options={draw=green}} 
\addlegendentry{$k=3$ stab.};
\addlegendimage{semithick, mygray, mark= star, mark size=2, mark options={draw=mygray}} 
\addlegendentry{$k=4$ stab.};
\addlegendimage{semithick, myyellow, mark=diamond, mark size=2, mark options={draw=myyellow}} 
\addlegendentry{$k=5$ stab.};
\addlegendimage{semithick, blu, mark=square, mark size=2, mark options={draw=blu}} 
\addlegendentry{$k=1$ non-stab.};
\addlegendimage{dashed, red, mark=o, mark size=2, mark options={draw=red}} 
\addlegendentry{$k=2$ non-stab.};
\addlegendimage{dashed, green, mark=triangle, mark size=2, mark options={draw=green}} 
\addlegendentry{$k=3$ non-stab.};
\addlegendimage{dashed, mygray, mark= star, mark size=2, mark options={draw=mygray}} 
\addlegendentry{$k=4$ non-stab.};
\addlegendimage{dashed, myyellow, mark=diamond, mark size=2, mark options={draw=myyellow}} 
\addlegendentry{$k=5$ non-stab.};
\end{axis}

\end{tikzpicture}  \\
		\vspace{0.5cm}
	\subfloat[][Raviart-Thomas.]
{
	\newcommand{\logLogSlopeTriangleOne}[5]
{

\pgfplotsextra
{
\pgfkeysgetvalue{/pgfplots/xmin}{\xmin}
\pgfkeysgetvalue{/pgfplots/xmax}{\xmax}
\pgfkeysgetvalue{/pgfplots/ymin}{\ymin}
\pgfkeysgetvalue{/pgfplots/ymax}{\ymax}

\pgfmathsetmacro{\xArel}{#1}
\pgfmathsetmacro{\yArel}{#3}
\pgfmathsetmacro{\xBrel}{#1-#2}
\pgfmathsetmacro{\yBrel}{\yArel}
\pgfmathsetmacro{\xCrel}{\xArel}

\pgfmathsetmacro{\lnxB}{\xmin*(1-(#1-#2))+\xmax*(#1-#2)} 
\pgfmathsetmacro{\lnxA}{\xmin*(1-#1)+\xmax*#1} 
\pgfmathsetmacro{\lnyA}{\ymin*(1-#3)+\ymax*#3} 
\pgfmathsetmacro{\lnyC}{\lnyA+#4*(\lnxA-\lnxB)}
\pgfmathsetmacro{\yCrel}{\lnyC-\ymin)/(\ymax-\ymin)} 

\coordinate (A) at (rel axis cs:\xArel,\yArel);
\coordinate (B) at (rel axis cs:\xBrel,\yBrel);
\coordinate (C) at (rel axis cs:\xCrel,\yCrel);

\draw[#5]   (A)-- node[pos=0.5,anchor=north] {1} 
            (B)-- 
            (C)-- node[pos=0.5,anchor=east] {#4} 
            cycle;
}
}


\begin{tikzpicture}

\definecolor{color1}{rgb}{255, 0, 255}
\definecolor{blu}{rgb}{0,0,1}
\definecolor{red}{rgb}{1,0,0}
\definecolor{green}{rgb}{0.3,0.7,0.2}
\definecolor{mygray}{rgb}{0.5,0.5,0.5}
\definecolor{myyellow}{rgb}{0.9,0.7,0.1}

\begin{semilogxaxis}[
width=5.5cm,
height=5.5cm,
ymin = 0,
ymax = 0.5,
xlabel={$h$},
ylabel={$\beta_1$},
xtick=data,
xticklabels={$2^{-2}$, $2^{-3}$, $2^{-4}$, $2^{-5}$, $2^{-6}$},
grid=major,
]

\addplot[semithick, blu, mark=square, mark size=2, mark options={draw=blu}] table[ x index = 0, y index = 2] 
{\dataDir/is_b1_pentagon_bent_RT.txt};
\addplot[semithick, red, mark=o, mark size=2, mark options={draw=red}] table[ x index = 0, y index = 4]
{\dataDir/is_b1_pentagon_bent_RT.txt};
\addplot[semithick, green, mark=triangle, mark size=2, mark options={draw=green}] table[ x index = 0, y index = 6] 
{\dataDir/is_b1_pentagon_bent_RT.txt};
\addplot[semithick, mygray, mark= star, mark size=2, mark options={draw=mygray}] table[ x index = 0, y index = 8] 
{\dataDir/is_b1_pentagon_bent_RT.txt};
\addplot[semithick, magenta, mark=diamond, mark size=2, mark options={draw=magenta}] table[ x index = 0, y index = 10] 
{\dataDir/is_b1_pentagon_bent_RT.txt};

\addplot[dashed, blu, mark=square, mark size=2, mark options={draw=blu}] table[ x index = 0, y index = 1] 
{\dataDir/is_b1_pentagon_bent_RT.txt};
\addplot[dashed, red, mark=o, mark size=2, mark options={draw=red}] table[ x index = 0, y index = 3]
{\dataDir/is_b1_pentagon_bent_RT.txt};
\addplot[dashed, green, mark=triangle, mark size=2, mark options={draw=green}] table[ x index = 0, y index = 5] 
{\dataDir/is_b1_pentagon_bent_RT.txt};
\addplot[dashed, mygray, mark= star, mark size=2, mark options={draw=mygray}] table[ x index = 0, y index = 7] 
{\dataDir/is_b1_pentagon_bent_RT.txt};
\addplot[dashed, magenta, mark=diamond, mark size=2, mark options={draw=magenta}] table[ x index = 0, y index = 9] 
{\dataDir/is_b1_pentagon_bent_RT.txt};


\end{semilogxaxis}

\end{tikzpicture}
}
\subfloat[][N\'ed\'elec.]
{
	\newcommand{\logLogSlopeTriangleOne}[5]
{

\pgfplotsextra
{
\pgfkeysgetvalue{/pgfplots/xmin}{\xmin}
\pgfkeysgetvalue{/pgfplots/xmax}{\xmax}
\pgfkeysgetvalue{/pgfplots/ymin}{\ymin}
\pgfkeysgetvalue{/pgfplots/ymax}{\ymax}

\pgfmathsetmacro{\xArel}{#1}
\pgfmathsetmacro{\yArel}{#3}
\pgfmathsetmacro{\xBrel}{#1-#2}
\pgfmathsetmacro{\yBrel}{\yArel}
\pgfmathsetmacro{\xCrel}{\xArel}

\pgfmathsetmacro{\lnxB}{\xmin*(1-(#1-#2))+\xmax*(#1-#2)} 
\pgfmathsetmacro{\lnxA}{\xmin*(1-#1)+\xmax*#1} 
\pgfmathsetmacro{\lnyA}{\ymin*(1-#3)+\ymax*#3} 
\pgfmathsetmacro{\lnyC}{\lnyA+#4*(\lnxA-\lnxB)}
\pgfmathsetmacro{\yCrel}{\lnyC-\ymin)/(\ymax-\ymin)} 

\coordinate (A) at (rel axis cs:\xArel,\yArel);
\coordinate (B) at (rel axis cs:\xBrel,\yBrel);
\coordinate (C) at (rel axis cs:\xCrel,\yCrel);

\draw[#5]   (A)-- node[pos=0.5,anchor=north] {1} 
            (B)-- 
            (C)-- node[pos=0.5,anchor=east] {#4} 
            cycle;
}
}


\begin{tikzpicture}

\definecolor{color1}{rgb}{255, 0, 255}
\definecolor{blu}{rgb}{0,0,1}
\definecolor{red}{rgb}{1,0,0}
\definecolor{green}{rgb}{0.3,0.7,0.2}
\definecolor{mygray}{rgb}{0.5,0.5,0.5}
\definecolor{myyellow}{rgb}{0.9,0.7,0.1}

\begin{semilogxaxis}[
width=5.5cm,
height=5.5cm,
ymin = 0,
ymax = 0.5,
xlabel={$h$},
ylabel={$\beta_1$},
xtick=data,
xticklabels={$2^{-2}$, $2^{-3}$, $2^{-4}$, $2^{-5}$, $2^{-6}$},
grid=major,
]

\addplot[semithick, blu, mark=square, mark size=2, mark options={draw=blu}] table[ x index = 0, y index = 2] 
{\dataDir/is_b1_pentagon_bent_N.txt};
\addplot[semithick, red, mark=o, mark size=2, mark options={draw=red}] table[ x index = 0, y index = 4]
{\dataDir/is_b1_pentagon_bent_N.txt};
\addplot[semithick, green, mark=triangle, mark size=2, mark options={draw=green}] table[ x index = 0, y index = 6] 
{\dataDir/is_b1_pentagon_bent_N.txt};
\addplot[semithick, mygray, mark= star, mark size=2, mark options={draw=mygray}] table[ x index = 0, y index = 8] 
{\dataDir/is_b1_pentagon_bent_N.txt};
\addplot[semithick, magenta, mark=diamond, mark size=2, mark options={draw=magenta}] table[ x index = 0, y index = 10] 
{\dataDir/is_b1_pentagon_bent_N.txt};

\addplot[dashed, blu, mark=square, mark size=2, mark options={draw=blu}] table[ x index = 0, y index = 1] 
{\dataDir/is_b1_pentagon_bent_N.txt};
\addplot[dashed, red, mark=o, mark size=2, mark options={draw=red}] table[ x index = 0, y index = 3]
{\dataDir/is_b1_pentagon_bent_N.txt};
\addplot[dashed, green, mark=triangle, mark size=2, mark options={draw=green}] table[ x index = 0, y index = 5] 
{\dataDir/is_b1_pentagon_bent_N.txt};
\addplot[dashed, mygray, mark= star, mark size=2, mark options={draw=mygray}] table[ x index = 0, y index = 7] 
{\dataDir/is_b1_pentagon_bent_N.txt};
\addplot[dashed, magenta, mark=diamond, mark size=2, mark options={draw=magenta}] table[ x index = 0, y index = 9] 
{\dataDir/is_b1_pentagon_bent_N.txt};


\end{semilogxaxis}

\end{tikzpicture}
}\\
	\vspace{0.5cm}
	\subfloat[][Taylor-Hood.]
{
	\newcommand{\logLogSlopeTriangleOne}[5]
{

\pgfplotsextra
{
\pgfkeysgetvalue{/pgfplots/xmin}{\xmin}
\pgfkeysgetvalue{/pgfplots/xmax}{\xmax}
\pgfkeysgetvalue{/pgfplots/ymin}{\ymin}
\pgfkeysgetvalue{/pgfplots/ymax}{\ymax}

\pgfmathsetmacro{\xArel}{#1}
\pgfmathsetmacro{\yArel}{#3}
\pgfmathsetmacro{\xBrel}{#1-#2}
\pgfmathsetmacro{\yBrel}{\yArel}
\pgfmathsetmacro{\xCrel}{\xArel}

\pgfmathsetmacro{\lnxB}{\xmin*(1-(#1-#2))+\xmax*(#1-#2)} 
\pgfmathsetmacro{\lnxA}{\xmin*(1-#1)+\xmax*#1} 
\pgfmathsetmacro{\lnyA}{\ymin*(1-#3)+\ymax*#3} 
\pgfmathsetmacro{\lnyC}{\lnyA+#4*(\lnxA-\lnxB)}
\pgfmathsetmacro{\yCrel}{\lnyC-\ymin)/(\ymax-\ymin)} 

\coordinate (A) at (rel axis cs:\xArel,\yArel);
\coordinate (B) at (rel axis cs:\xBrel,\yBrel);
\coordinate (C) at (rel axis cs:\xCrel,\yCrel);

\draw[#5]   (A)-- node[pos=0.5,anchor=north] {1} 
            (B)-- 
            (C)-- node[pos=0.5,anchor=east] {#4} 
            cycle;
}
}


\begin{tikzpicture}

\definecolor{color1}{rgb}{255, 0, 255}
\definecolor{blu}{rgb}{0,0,1}
\definecolor{red}{rgb}{1,0,0}
\definecolor{green}{rgb}{0.3,0.7,0.2}
\definecolor{mygray}{rgb}{0.5,0.5,0.5}
\definecolor{myyellow}{rgb}{0.9,0.7,0.1}

\begin{semilogxaxis}[
width=5.5cm,
height=5.5cm,
ymin = 0,
ymax = 0.5,
xlabel={$h$},
ylabel={$\beta_1$},
xtick=data,
xticklabels={$2^{-2}$, $2^{-3}$, $2^{-4}$, $2^{-5}$, $2^{-6}$},
grid=major,
]

\addplot[semithick, blu, mark=square, mark size=2, mark options={draw=blu}] table[ x index = 0, y index = 2] 
{\dataDir/is_b1_pentagon_bent_TH.txt};
\addplot[semithick, red, mark=o, mark size=2, mark options={draw=red}] table[ x index = 0, y index = 4]
{\dataDir/is_b1_pentagon_bent_TH.txt};
\addplot[semithick, green, mark=triangle, mark size=2, mark options={draw=green}] table[ x index = 0, y index = 6] 
{\dataDir/is_b1_pentagon_bent_TH.txt};
\addplot[semithick, mygray, mark= star, mark size=2, mark options={draw=mygray}] table[ x index = 0, y index = 8] 
{\dataDir/is_b1_pentagon_bent_TH.txt};
\addplot[semithick, magenta, mark=diamond, mark size=2, mark options={draw=magenta}] table[ x index = 0, y index = 10] 
{\dataDir/is_b1_pentagon_bent_TH.txt};

\addplot[dashed, blu, mark=square, mark size=2, mark options={draw=blu}] table[ x index = 0, y index = 1] 
{\dataDir/is_b1_pentagon_bent_TH.txt};
\addplot[dashed, red, mark=o, mark size=2, mark options={draw=red}] table[ x index = 0, y index = 3]
{\dataDir/is_b1_pentagon_bent_TH.txt};
\addplot[dashed, green, mark=triangle, mark size=2, mark options={draw=green}] table[ x index = 0, y index = 5] 
{\dataDir/is_b1_pentagon_bent_TH.txt};
\addplot[dashed, mygray, mark= star, mark size=2, mark options={draw=mygray}] table[ x index = 0, y index = 7] 
{\dataDir/is_b1_pentagon_bent_TH.txt};
\addplot[dashed, magenta, mark=diamond, mark size=2, mark options={draw=magenta}] table[ x index = 0, y index = 9] 
{\dataDir/is_b1_pentagon_bent_TH.txt};


\end{semilogxaxis}

\end{tikzpicture}
}
	\caption{Inf-sup constant for the \emph{mapped pentagon} .}\label{stokes_trimming:num_exp:fig3bent}
\end{figure}
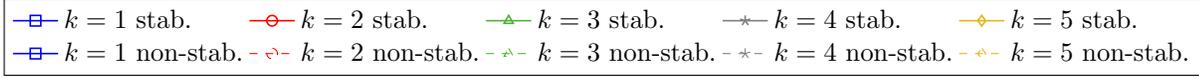
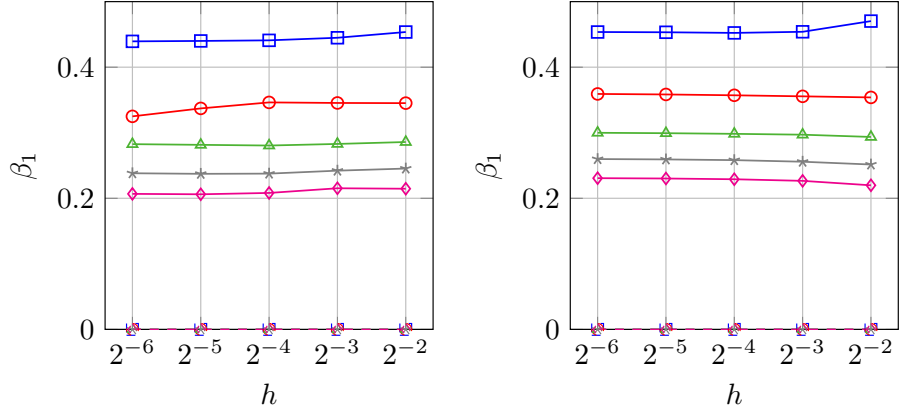
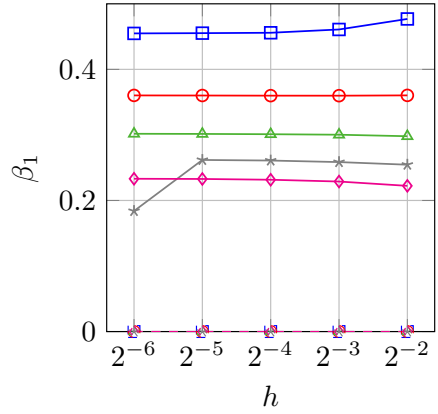
\begin{figure}[!ht]
	\centering
	\subfloat[][\emph{Mapped pentagon}.]
	{
		\begin{tikzpicture}
\definecolor{color2}{rgb}{0.872549019607843,0.019607843137255,0.0819607843137255}
\draw(2,0) -- (4,0) ;
\draw(0,2) -- (0,4) ;
\coordinate (a) at (2.25,2);
\node at (a)   {\textcolor{color2}{\large$\mathbf{\Omega}$}};
\draw[color2,ultra thick]  (0,3.48) -- (0,4);
\begin{scope}
  \clip(0,0) -- (4.1,0) -- (4.1,4.1) -- (0,4.1) -- (0,0);
   \draw[ultra thick,smooth,domain=0:4,variable=\x,dashed,color2]  plot ({\x},{ -0.1922 *\x^4 - 0.3*\x + 3.500 )});
  \draw (0,0) circle (2cm);
\draw[fill = white]  (0,0) circle (2 cm);
  \foreach \x in  {0,25,50,75}
    {\draw[ultra thin, lightgray]  (0,0) circle (2.\x cm);
    \draw[ultra thin, lightgray]  (0,0) circle (3.\x cm);}
    \draw[ultra thick,color2] (0,0) circle (4cm);
    \draw[color2,ultra thick]  (2,0) -- (4,0);
\foreach \ang in {0,2,4,6,8,10,12,14,16,18,20,22,24,26,28,31} {
  \draw [ultra thin, lightgray] (\ang * 180 / 30:4) -- (\ang * 180 / 30:2);
}
 \draw[ultra thick,color2] (2,0)  arc (0:25:2);
\end{scope}

\draw[step= 0.385cm,lightgray,ultra thin] (-5,0.5) grid (-2,3.5);
\draw [->,ultra thick] (-1.75,2) to [bend left=40] node[above,scale=1.5] {\small$F$} (-0.25,2);
\draw[black](-5,0.5) -- (-2,0.5) -- (-2,3.5) -- (-5,3.5) -- (-5,0.5);
\draw[color=color2,ultra thick] (-5,0.5+0.75)--(-5,0.5) -- (-5,0.5) -- (-2,0.5) -- (-2,3.5) -- (-2-0.75,3.5) ;
\draw[color=color2,ultra thick,dashed] (-5,0.5+0.75) -- (-2-0.75,3.5) ;
\node  at (-3.25,2) {\textcolor{color2}{\large$\mathbf{\widehat{\Omega}}$}};
\begin{scope}
  \clip (-5,0) -- (-3,0) -- (-3,3) -- (-5,3) -- (-5,0);
\end{scope}

\end{tikzpicture}\label{stokes_trimming:fig_mapped_pentagon}
	}
	\subfloat[][\emph{Rotating square}\\ for $\alpha\in \{0, \frac{\pi}{10}, \frac{\pi}{5}, \frac{3\pi}{10},\frac{2\pi}{5},\frac{\pi}{2}\}$.]
	{ \raisebox{0ex}{
			\includegraphics[width=0.45\textwidth,keepaspectratio=true]{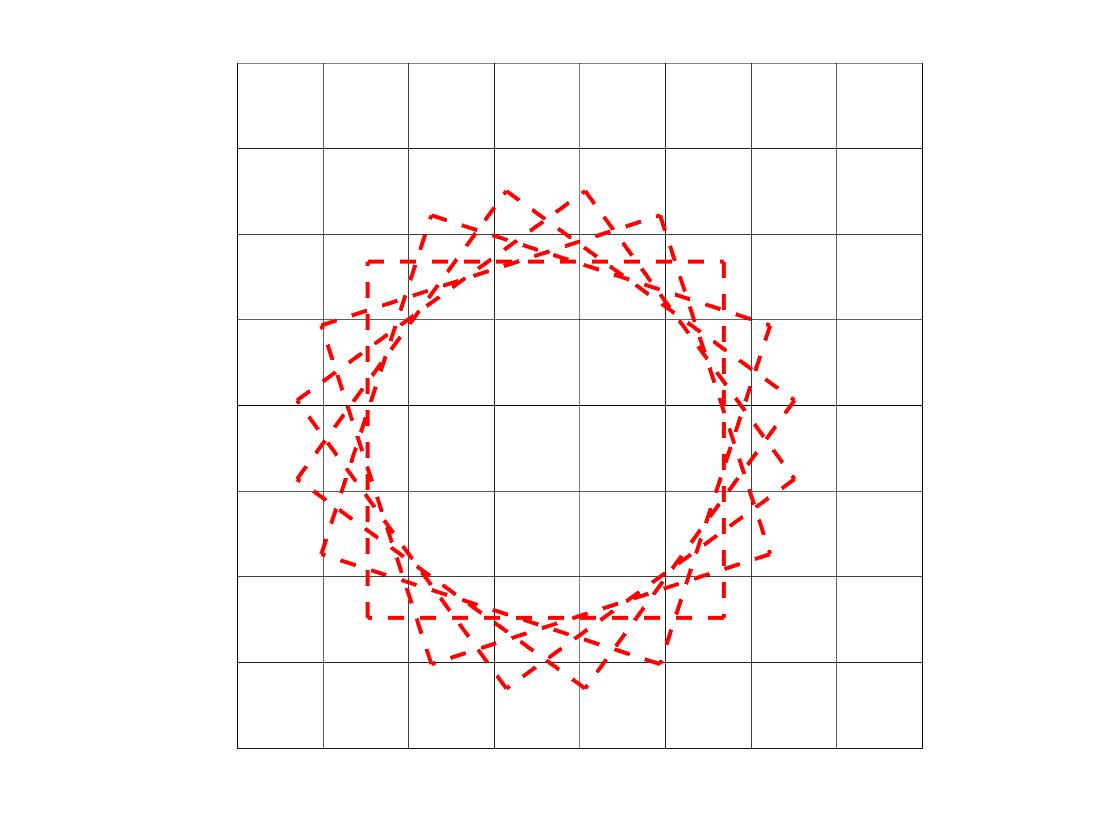}\label{stokes_trimming:rot_square}
		}
	}
	\caption{Trimmed domains.}\label{stokes_trimming:madre3}
\end{figure}

\subsection{Rotating square}
 We embed $\Omega = \left(0.19,0.71\right) \times \left(0.19,0.71\right)$ into $\Omega_0=(0,1)^2$, the latter subdivided with a Cartesian grid of $8$ elements per direction, and rotate $\Omega$ around its barycenter for different angles $\alpha$, as illustrated in Figure~\ref{stokes_trimming:madre3}\subref{stokes_trimming:rot_square}. We choose a threshold parameter $\theta=0.75$, and, for each angle $\alpha\in\{ i\frac{\pi}{200}: i=0,\dots, 100\}$, we compute the inf-sup constants $\beta_0$ and $\beta_1$ in the stabilized and non-stabilized cases. For every configuration we compute $\displaystyle \eta :=\min_{K\in\mathcal M_h}\abs{K\cap \Omega}$ and in Figure~\ref{stokes_trimming:num_exp:figdio6} we plot the inf-sup constants with respect to $\eta$ for the Raviart-Thomas, the N\'ed\'elec, and the Taylor-Hood elements of degree $k=2$. In most configurations, we can observe that the constants corresponding to the stabilized formulation perform better than those of the non-stabilized formulation, specially for small values of $\eta$. 

Furthermore, we notice a greater efficiency, \emph{i.e.}, a greater difference between stabilized and non-stabilized cases, when using the Taylor-Hood element. The configurations corresponding to a smaller $\eta$ do not necessarily give rise to a worse inf-sup constant. Although stabilization does not always seem to ``beat'' the non-stabilized method, we observe that the configurations in which the non-stabilized inf-sup constant is larger than the stabilized one are, in general, the ones with bigger values of $\eta$, which are not the most critical ones.
\begin{figure}[!ht]
	\centering
	\includegraphics[]{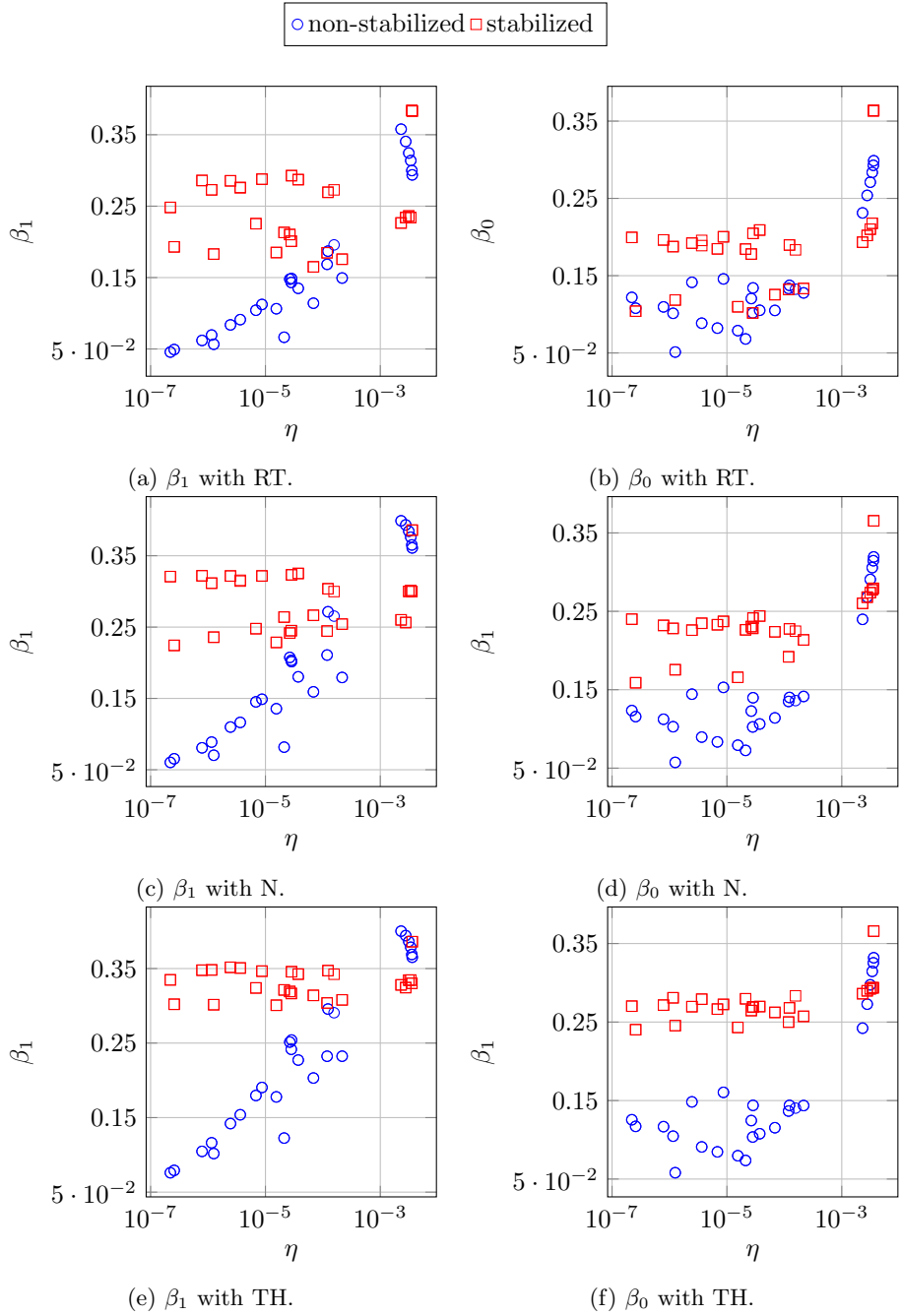}  \\
	\vspace{0.5cm}
	\subfloat[][$\beta_1$ with RT.]
	{
		\includestandalone[]{infsup_beta1_vs_eta_rotating_square_74_RT}
	}
	\subfloat[][$\beta_0$ with RT.]
	{
		\includestandalone[]{infsup_beta0_vs_eta_rotating_square_74_RT}
	}\\
	\subfloat[][$\beta_1$ with N.]
	{
		\includestandalone[]{infsup_beta1_vs_eta_rotating_square_74_N}
	}
	\subfloat[][$\beta_0$ with N.]
	{
		\includestandalone[]{infsup_beta0_vs_eta_rotating_square_74_N}
	}\\
	\subfloat[][$\beta_1$ with TH.]
	{
		\includestandalone[]{infsup_beta1_vs_eta_rotating_square_74_TH}
	}
	\subfloat[][$\beta_0$ with TH.]
	{
		\includestandalone[]{infsup_beta0_vs_eta_rotating_square_74_TH}
	}
	\caption{Inf-sup constants vs. the measure of the ``smallest cut element''.}\label{stokes_trimming:num_exp:figdio6}
\end{figure}		

\subsection{Square with circular trimming}
Let us set up another numerical experience where the physical domain is $\Omega=\Omega_0\setminus\overline\Omega_1$ with $\Omega_0 = \left(0,2\right)\times \left(0,2\right)$ and $\Omega_1 = B(0,r)$, $r = 0.52$, as depicted in Figure~\ref{stokes_trimming:num_exp}\subref{stokes_trimming:num_exp:fig4}. We take as reference solution fields
\begin{align*}
	\u = \left(2y^3\sin(x),x^3\sin(x)- \frac{y^4\cos(x)}{2} -3x^2\cos(x)\right),\qquad p =\frac{x^3y^2}{2} + \frac{y^3}{2},
\end{align*}
where $\u$ is a solenoidal vector field.
We impose Neumann boundary conditions on the straight trimmed sides $\{(0,y):0\le y\le 2\}$, $\{(x,0):0\le x\le 2\}$ and on the rest of the boundary we impose Dirichlet boundary conditions, enforced in a weak sense on $\{\left(r\cos \theta,r \sin\theta \right): 0\le\theta\le\frac{\pi}{2}\}$. 
\begin{figure}[!ht]
	\centering
	\subfloat[B\'ezier mesh.]
	{\raisebox{2ex}
		{\includegraphics[]{plate_with_hole.tex}}  \label{stokes_trimming:num_exp:fig4}
	}
	\subfloat[$\abs{\dive \u_h}$ for $h=2^{-4}$.]
	{
		\includegraphics[width=0.32\textwidth,keepaspectratio=true]{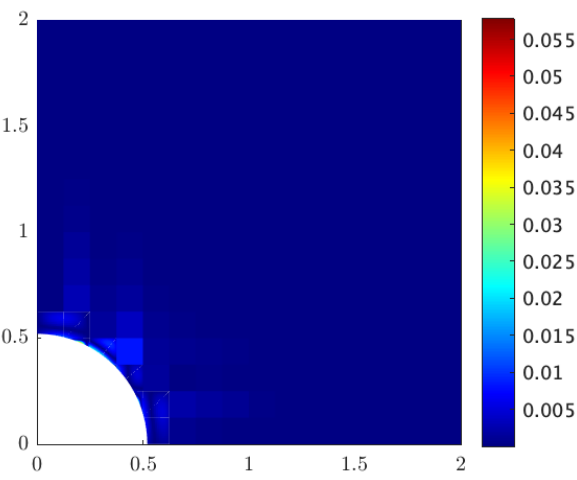} \label{stokes_trimming:num_exp:fig6}
	}
	\caption{\emph{Square with circular trimming}.}\label{stokes_trimming:num_exp}
\end{figure}
We solve using the non-symmetric stabilized formulation~\eqref{stokes_trimming:stability:prob6_stab}, discretized with the Raviart-Thomas element, with different degrees $k=1,2,3,4,5$, penalty parameter $\gamma = 10\left(k+2\right)^2$, and threshold parameter $\theta=1$. The convergence results, validating the error estimates of Theorem~\ref{stokes_trimming:thm:apriori}, are shown in Figure~\ref{stokes_trimming:num_exp:fig5}, while the divergence of the discrete velocity field $\u_h$, for $k=3$ and $h=2^{-4}$, has been plotted in Figure~\ref{stokes_trimming:num_exp}\subref{stokes_trimming:num_exp:fig6}.  As already observed in Remark~\ref{stokes_trimming:rmk:not_commuting}, our numerical scheme does not preserve exactly the incompressibility constraint since $\dive  V^{\mathrm{RT}}_h \not\subset \overline Q_h$. From Figure~\ref{stokes_trimming:num_exp}\subref{stokes_trimming:num_exp:fig6}, we can observe that the divergence of the numerical solution for the velocity is polluted in the vicinity of the trimmed boundary.
\begin{figure}[ht!]
	\centering
	\includegraphics[]{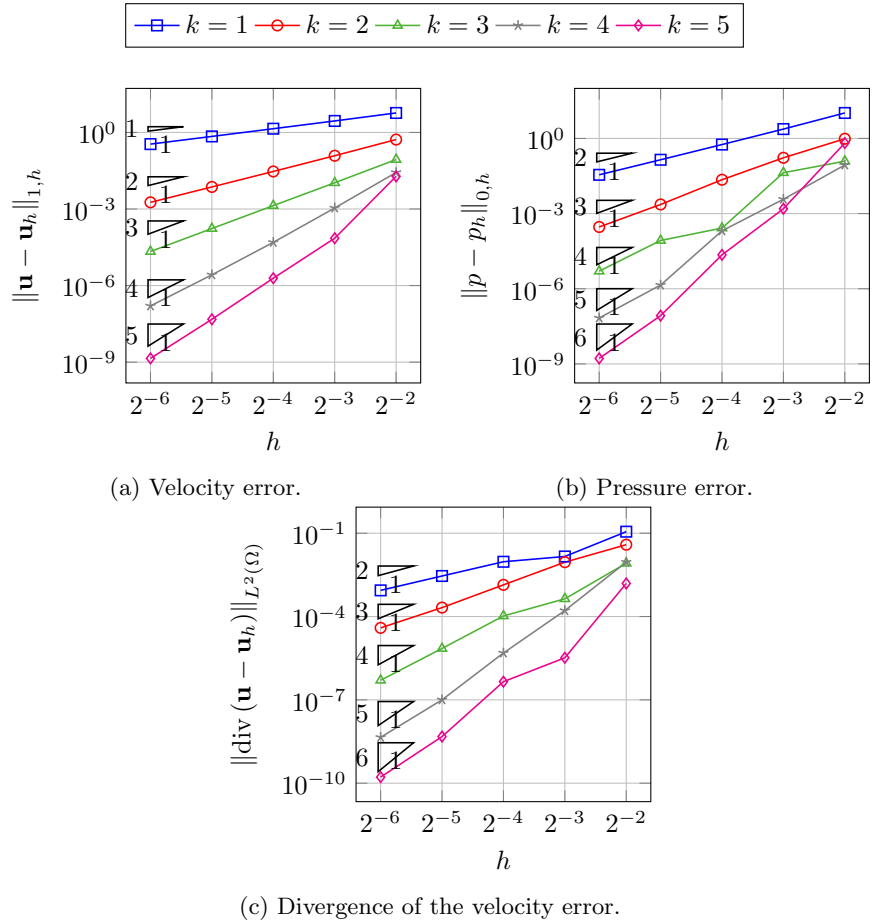}  \\
	\vspace{0.5cm}
	\subfloat[][Velocity error.]
	{
		\newcommand{\logLogSlopeTriangleOne}[5]
{

\pgfplotsextra
{
\pgfkeysgetvalue{/pgfplots/xmin}{\xmin}
\pgfkeysgetvalue{/pgfplots/xmax}{\xmax}
\pgfkeysgetvalue{/pgfplots/ymin}{\ymin}
\pgfkeysgetvalue{/pgfplots/ymax}{\ymax}

\pgfmathsetmacro{\xArel}{#1}
\pgfmathsetmacro{\yArel}{#3}
\pgfmathsetmacro{\xBrel}{#1-#2}
\pgfmathsetmacro{\yBrel}{\yArel}
\pgfmathsetmacro{\xCrel}{\xArel}

\pgfmathsetmacro{\lnxB}{\xmin*(1-(#1-#2))+\xmax*(#1-#2)} 
\pgfmathsetmacro{\lnxA}{\xmin*(1-#1)+\xmax*#1} 
\pgfmathsetmacro{\lnyA}{\ymin*(1-#3)+\ymax*#3} 
\pgfmathsetmacro{\lnyC}{\lnyA+#4*(\lnxA-\lnxB)}
\pgfmathsetmacro{\yCrel}{\lnyC-\ymin)/(\ymax-\ymin)} 

\coordinate (A) at (rel axis cs:\xArel,\yArel);
\coordinate (B) at (rel axis cs:\xBrel,\yBrel);
\coordinate (C) at (rel axis cs:\xCrel,\yCrel);

\draw[#5]   (A)-- node[pos=0.5,anchor=north] {1} 
            (B)-- 
            (C)-- node[pos=0.5,anchor=east] {#4} 
            cycle;
}
}


\begin{tikzpicture}

\definecolor{color1}{rgb}{255, 0, 255}
\definecolor{blu}{rgb}{0,0,1}
\definecolor{red}{rgb}{1,0,0}
\definecolor{green}{rgb}{0.3,0.7,0.2}
\definecolor{mygray}{rgb}{0.5,0.5,0.5}
\definecolor{myyellow}{rgb}{0.9,0.7,0.1}

\begin{loglogaxis}[
width=5.5cm,
height=5.5cm,
xlabel={$h$},
	ylabel={$\norm{\u-\u_h}_{1,h}$},
xtick=data,
xticklabels={$2^{-2}$, $2^{-3}$, $2^{-4}$, $2^{-5}$, $2^{-6}$},
grid=major,
]

\addplot[semithick, blu, mark=square, mark size=2, mark options={draw=blu}] table[ x index = 0, y index = 1] 
{\dataDir/error_vel_plate_hole_stab.txt};
\addplot[semithick, red, mark=o, mark size=2, mark options={draw=red}] table[ x index = 0, y index = 2]
{\dataDir/error_vel_plate_hole_stab.txt};
\addplot[semithick, green, mark=triangle, mark size=2, mark options={draw=green}] table[ x index = 0, y index = 3] 
{\dataDir/error_vel_plate_hole_stab.txt};
\addplot[semithick, mygray, mark= star, mark size=2, mark options={draw=mygray}] table[ x index = 0, y index = 4] 
{\dataDir/error_vel_plate_hole_stab.txt};
\addplot[semithick, magenta, mark=diamond, mark size=2, mark options={draw=magenta}] table[ x index = 0, y index = 5] 
{\dataDir/error_vel_plate_hole_stab.txt};

\logLogSlopeTriangleOne{0.075}{-0.12}{0.87}{1}{black, semithick} ;
\logLogSlopeTriangleOne{0.075}{-0.12}{0.70}{2}{black, semithick} ;
\logLogSlopeTriangleOne{0.075}{-0.12}{0.55}{3}{black, semithick} ;
\logLogSlopeTriangleOne{0.075}{-0.12}{0.35}{4}{black, semithick} ;
\logLogSlopeTriangleOne{0.075}{-0.12}{0.20}{5}{black, semithick} ;

\end{loglogaxis}

\end{tikzpicture}
	}
	\subfloat[][Pressure error.]
	{
		\newcommand{\logLogSlopeTriangleOne}[5]
{

    \pgfplotsextra
    {
        \pgfkeysgetvalue{/pgfplots/xmin}{\xmin}
        \pgfkeysgetvalue{/pgfplots/xmax}{\xmax}
        \pgfkeysgetvalue{/pgfplots/ymin}{\ymin}
        \pgfkeysgetvalue{/pgfplots/ymax}{\ymax}

        \pgfmathsetmacro{\xArel}{#1}
        \pgfmathsetmacro{\yArel}{#3}
        \pgfmathsetmacro{\xBrel}{#1-#2}
        \pgfmathsetmacro{\yBrel}{\yArel}
        \pgfmathsetmacro{\xCrel}{\xArel}

        \pgfmathsetmacro{\lnxB}{\xmin*(1-(#1-#2))+\xmax*(#1-#2)} 
        \pgfmathsetmacro{\lnxA}{\xmin*(1-#1)+\xmax*#1} 
        \pgfmathsetmacro{\lnyA}{\ymin*(1-#3)+\ymax*#3} 
        \pgfmathsetmacro{\lnyC}{\lnyA+#4*(\lnxA-\lnxB)}
        \pgfmathsetmacro{\yCrel}{\lnyC-\ymin)/(\ymax-\ymin)} 

        \coordinate (A) at (rel axis cs:\xArel,\yArel);
        \coordinate (B) at (rel axis cs:\xBrel,\yBrel);
        \coordinate (C) at (rel axis cs:\xCrel,\yCrel);

        \draw[#5]   (A)-- node[pos=0.5,anchor=north] {1} 
                    (B)-- 
                    (C)-- node[pos=0.5,anchor=east] {#4} 
                    cycle;
    }
}


\begin{tikzpicture}

\definecolor{color1}{rgb}{255, 0, 255}
\definecolor{blu}{rgb}{0,0,1}
\definecolor{red}{rgb}{1,0,0}
\definecolor{green}{rgb}{0.3,0.7,0.2}
\definecolor{mygray}{rgb}{0.5,0.5,0.5}
\definecolor{myyellow}{rgb}{0.9,0.7,0.1}

		\begin{loglogaxis}[
	width=5.5cm,
	height=5.5cm,
	xlabel={$h$},
	ylabel={$\norm{p-p_h}_{0,h}$},
	xtick=data,
	xticklabels={$2^{-2}$, $2^{-3}$, $2^{-4}$, $2^{-5}$, $2^{-6}$},
	grid=major,
	]

\addplot[semithick, blu, mark=square, mark size=2, mark options={draw=blu}] table[ x index = 0, y index = 1]
 {\dataDir/error_press_plate_hole_stab.txt};
\addplot[semithick, red, mark=o, mark size=2, mark options={draw=red}] table[ x index = 0, y index = 2] 
{\dataDir/error_press_plate_hole_stab.txt};
\addplot[semithick, green, mark=triangle, mark size=2, mark options={draw=green}] table[ x index = 0, y index = 3] 
{\dataDir/error_press_plate_hole_stab.txt};
\addplot[semithick, mygray, mark= star, mark size=2, mark options={draw=mygray}] table[ x index = 0, y index = 4]
 {\dataDir/error_press_plate_hole_stab.txt};
\addplot[semithick, magenta, mark=diamond, mark size=2, mark options={draw=magenta}] table[ x index = 0, y index = 5] 
{\dataDir/error_press_plate_hole_stab.txt};
\logLogSlopeTriangleOne{0.075}{-0.12}{0.78}{2}{black, semithick} ;
\logLogSlopeTriangleOne{0.075}{-0.12}{0.62}{3}{black, semithick} ;
\logLogSlopeTriangleOne{0.075}{-0.12}{0.46}{4}{black, semithick} ;
\logLogSlopeTriangleOne{0.075}{-0.12}{0.32}{5}{black, semithick} ;
\logLogSlopeTriangleOne{0.075}{-0.12}{0.20}{6}{black, semithick} ;
\end{loglogaxis}

\end{tikzpicture}
	}\\
	\subfloat[][Divergence of the velocity error.]
	{
		\newcommand{\logLogSlopeTriangleOne}[5]
{

    \pgfplotsextra
    {
        \pgfkeysgetvalue{/pgfplots/xmin}{\xmin}
        \pgfkeysgetvalue{/pgfplots/xmax}{\xmax}
        \pgfkeysgetvalue{/pgfplots/ymin}{\ymin}
        \pgfkeysgetvalue{/pgfplots/ymax}{\ymax}

        \pgfmathsetmacro{\xArel}{#1}
        \pgfmathsetmacro{\yArel}{#3}
        \pgfmathsetmacro{\xBrel}{#1-#2}
        \pgfmathsetmacro{\yBrel}{\yArel}
        \pgfmathsetmacro{\xCrel}{\xArel}

        \pgfmathsetmacro{\lnxB}{\xmin*(1-(#1-#2))+\xmax*(#1-#2)} 
        \pgfmathsetmacro{\lnxA}{\xmin*(1-#1)+\xmax*#1} 
        \pgfmathsetmacro{\lnyA}{\ymin*(1-#3)+\ymax*#3} 
        \pgfmathsetmacro{\lnyC}{\lnyA+#4*(\lnxA-\lnxB)}
        \pgfmathsetmacro{\yCrel}{\lnyC-\ymin)/(\ymax-\ymin)} 

        \coordinate (A) at (rel axis cs:\xArel,\yArel);
        \coordinate (B) at (rel axis cs:\xBrel,\yBrel);
        \coordinate (C) at (rel axis cs:\xCrel,\yCrel);

        \draw[#5]   (A)-- node[pos=0.5,anchor=north] {1} 
                    (B)-- 
                    (C)-- node[pos=0.5,anchor=east] {#4} 
                    cycle;
    }
}


\begin{tikzpicture}

\definecolor{color1}{rgb}{255, 0, 255}
\definecolor{blu}{rgb}{0,0,1}
\definecolor{red}{rgb}{1,0,0}
\definecolor{green}{rgb}{0.3,0.7,0.2}
\definecolor{mygray}{rgb}{0.5,0.5,0.5}
\definecolor{myyellow}{rgb}{0.9,0.7,0.1}

\begin{loglogaxis}[
	width=5.5cm,
	height=5.5cm,
	xlabel={$h$},
	ylabel={$\norm{\dive\left(\u-\u_h\right)}_{L^2(\Omega)}$},
	xtick=data,
	xticklabels={$2^{-2}$, $2^{-3}$, $2^{-4}$, $2^{-5}$, $2^{-6}$},
	grid=major,
	]

\addplot[semithick, blu, mark=square, mark size=2, mark options={draw=blu}] table[ x index = 0, y index = 1]
 {\dataDir/error_div_plate_hole_stab.txt};
\addplot[semithick, red, mark=o, mark size=2, mark options={draw=red}] table[ x index = 0, y index = 2] 
{\dataDir/error_div_plate_hole_stab.txt};
\addplot[semithick, green, mark=triangle, mark size=2, mark options={draw=green}] table[ x index = 0, y index = 3] 
{\dataDir/error_div_plate_hole_stab.txt};
\addplot[semithick, mygray, mark= star, mark size=2, mark options={draw=mygray}] table[ x index = 0, y index = 4] 
{\dataDir/error_div_plate_hole_stab.txt};
\addplot[semithick, magenta, mark=diamond, mark size=2, mark options={draw=magenta}] table[ x index = 0, y index = 5] 
{\dataDir/error_div_plate_hole_stab.txt};

\logLogSlopeTriangleOne{0.075}{-0.12}{0.80}{2}{black, semithick} ;
\logLogSlopeTriangleOne{0.075}{-0.12}{0.67}{3}{black, semithick} ;
\logLogSlopeTriangleOne{0.075}{-0.12}{0.53}{4}{black, semithick} ;
\logLogSlopeTriangleOne{0.075}{-0.12}{0.34}{5}{black, semithick} ;
\logLogSlopeTriangleOne{0.075}{-0.12}{0.20}{6}{black, semithick} ;
\end{loglogaxis}

\end{tikzpicture}  
	}
	\caption{Convergence errors for the \emph{Square with circular trimming} with the Raviart-Thomas element.}\label{stokes_trimming:num_exp:fig5}
\end{figure}	

\subsection{Stokes flow around a cylinder}
We consider a classic benchmark example in computational fluid dynamics, \emph{i.e.}, the so-called two-dimensional  \emph{flow around a cylinder}, proposed by~\cite{Bayraktar_benchmarkcomputations} and already seen in the context of immersogeometric methods in~\cite{tuong_thesis}. The incompressible flow of a fluid around a cylinder placed in a channel is studied. The physical domain is $\Omega= \Omega_0\setminus\overline\Omega_1$, where $\Omega_0 = \left(0,L\right)\times \left( 0,H\right)$ and $\Omega_1 = B(x_0, R)$ with $L=2.2$, $H={0.41}$, $x_0=\left(0.2,0.2\right)$ and $R=0.05$. Let us observe that $\Omega_1$ is not symmetric with respect to $\Omega_0$. As Dirichlet boundary condition on the inflow boundary $\{(0,y):0\le y\le H\}$ a parabolic horizontal profile is prescribed:
\begin{align*}
	\u(0,y)=
	\begin{pmatrix}
		4 U_m y\left(H-y\right)/H^2 \\
		0	
	\end{pmatrix},
\end{align*}	
where $U_m=0.3$ is the maximum magnitude of the velocity field. Stress free boundary conditions, \emph{i.e.}, $\u_N=\bm{0}$, are imposed on the outflow boundary $\{(L,y):0\le y\le H\}$, while no slip boundary conditions are imposed on the rest of the boundary. No external forces act on the fluid flow, \emph{i.e.}, $\mathbf f =\bm{0}$.

Let us set $k=3$, $\gamma=10\left(k+1\right)^2$ and consider the mesh configuration depicted in Figure~\ref{stokes_trimming:num_exp:fig1}, with $2^6$ elements in the $x$-direction and $2^4$ elements in the $y$-direction.
\begin{figure}
	\centering
	\includegraphics[]{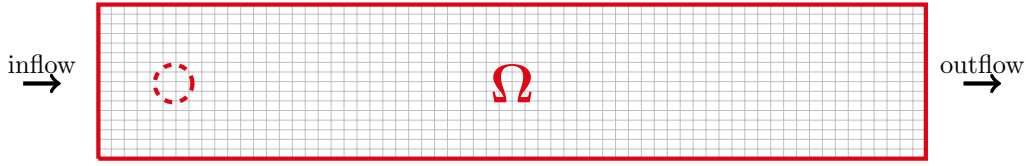}  
	\caption{B\'ezier mesh for the \emph{flow around a cylinder}.}\label{stokes_trimming:num_exp:fig1}
\end{figure}	
In Figure~\ref{stokes_trimming:num_exp:fig8} we show the magnitude of the velocity field and the pressure in the case of the stabilized formulation~\eqref{stokes_trimming:stability:prob6} for $m=0$ with the N\'ed\'elec isogeometric element. Note that both the velocity and pressure fields do not show any spurious oscillations around the trimmed part of the boundary and seem to comply with their physical meaning.

\begin{figure}[!ht]
	\centering
	\subfloat[][Magnitude of the velocity.]
	{
		\includegraphics[width=1\textwidth,keepaspectratio=true]{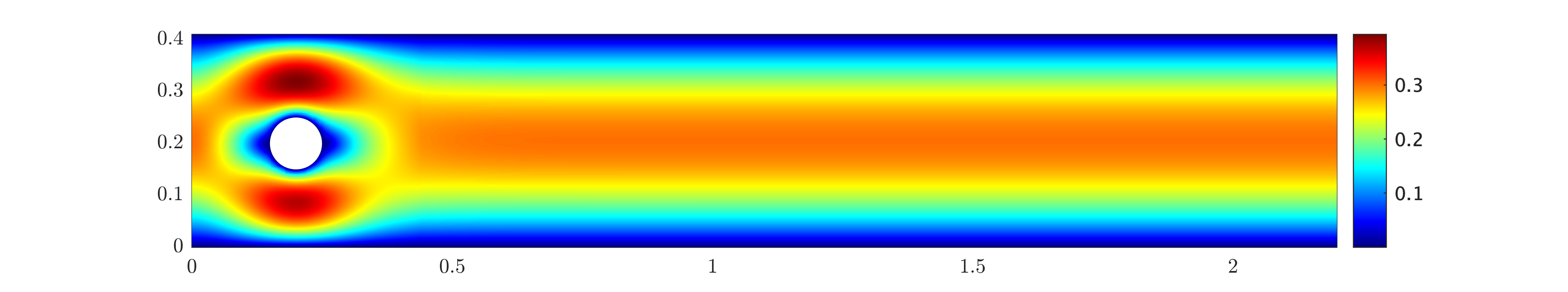}
	}\\
	\subfloat[][Pressure.]
	{
		\includegraphics[width=1\textwidth,keepaspectratio=true]{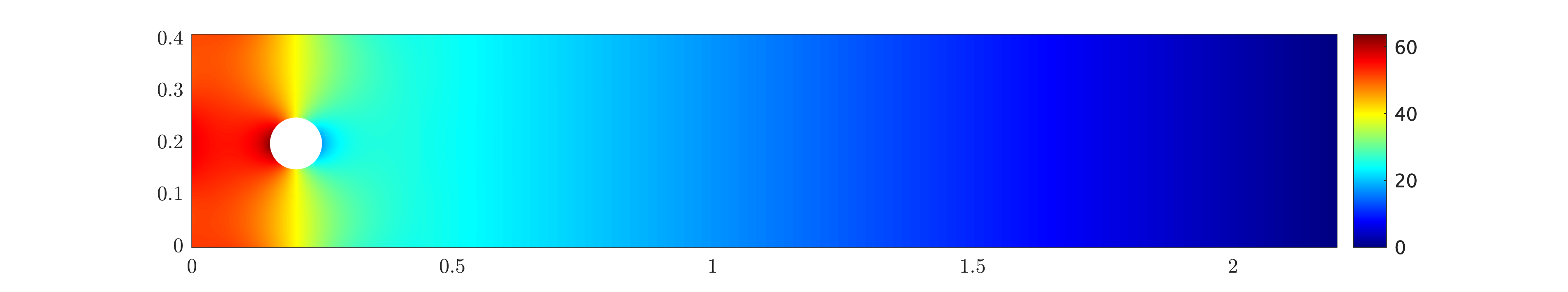}
	}
	\caption{Numerical solutions of the non-symmetric stabilized formulation~\eqref{stokes_trimming:stability:prob6_stab} for the \emph{flow around a cylinder} with the N\'ed\'elec element.}\label{stokes_trimming:num_exp:fig8}
\end{figure}

\subsection{Lid-driven cavity}
The lid-driven cavity is another important benchmark for the Stokes problem where the incompressible flow in a confined volume is driven by the tangential in-plane motion of two opposite bounding walls~\cite{Grcan2003StreamlineTI,Sturges1986StokesFI}. Here, the cavity is represented by the trimmed domain $\Omega=\left(-1,1\right)\times \left(-3,3\right)$ immersed in $\Omega_0$. $\Omega_0$ is the rectangle with vertices $(-3,-3.5)$, $(3,-3.5)$, $(3,3.5)$, $(-3,3.5)$ rotated of $\frac{\pi}{6}$ counterclockwise around the origin. No-slip Dirichlet boundary conditions are imposed on the left and right sides of the cavity, while the top and bottom ones are walls sliding, respectively, to the right and left with unitary velocity magnitude, namely we enforce the non-homogenous Dirichlet boundary conditions 
\begin{align*}
	\u\left(x,3\right) = 
	\begin{pmatrix}
		1 \\
		0
	\end{pmatrix}
	\quad\text{on}\quad \{\left(x,3 \right): x\in\left[-1,1\right]\},\qquad 
	\u\left(x,-3\right) = 
	\begin{pmatrix}
		-1 \\
		0
	\end{pmatrix}
	\quad\text{on}\quad \{\left(x,-3 \right): x\in\left[-1,1\right]\}. 
\end{align*}
Since the Dirichlet boundary conditions have a jump at the corners, the trace of the solution for the velocity does not belong to $\bm H^{\frac{1}{2}}\left( \Gamma\right)$, hence $\u\not\in \bm H^1(\Omega)$. The applied body force is $\f = \bm{0}$.
We solve the problem by using the non-symmetric stabilized formulation~\eqref{stokes_trimming:stability:prob6_stab}, discretized using the Taylor-Hood element, with degree $k=2$, penalty parameter $\gamma=30\left(k+2\right)^2$ and mesh-sizes $h_x=h_y=2^{-5}$ along the first and second parametric directions respectively. The mesh employed for the numerical simulation is depicted in Figure~\ref{stokes_trimming:num_exp:fig12}\subref{stokes_trimming:num_exp:fig2}. In Figures~\ref{stokes_trimming:num_exp:fig12}\subref{stokes_trimming:num_exp:fig12_vel},~\ref{stokes_trimming:num_exp:fig12}\subref{stokes_trimming:num_exp:fig12_press} the numerical solutions for the velocity and the pressures are plotted: our results are qualitatively in accordance with the ones of~\cite{MR2808112,Grcan2003StreamlineTI}.
\begin{figure}[!ht]
	\centering
	\subfloat[][\mbox{B\'ezier mesh for the \emph{lid-driven cavity}}.\label{stokes_trimming:num_exp:fig2}]
	{\raisebox{5ex}{
			\includegraphics[]{lid_driven_cavity_obliquous.tex}  }
	}
	\subfloat[][\mbox{Velocity}.\label{stokes_trimming:num_exp:fig12_vel}]
	{
		\includegraphics[scale=0.5,,keepaspectratio=true]{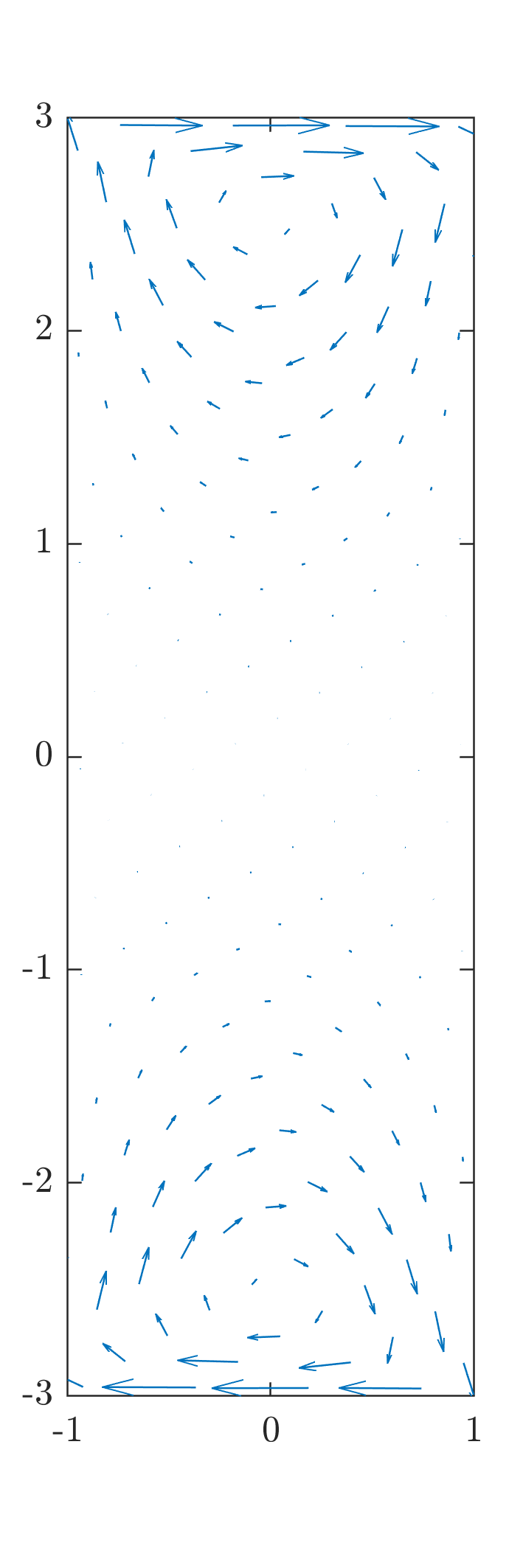}
	}
	\subfloat[][\mbox{Pressure}.\label{stokes_trimming:num_exp:fig12_press}]
	{
		\includegraphics[scale=0.5,,keepaspectratio=true]{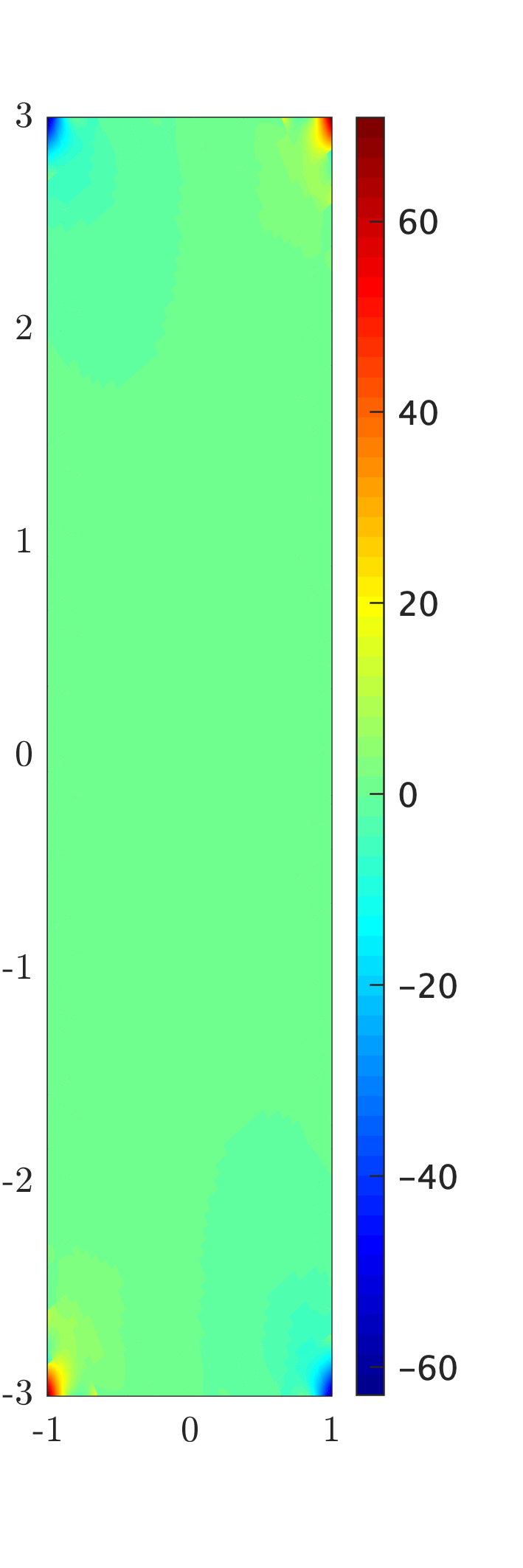}
	}
	\caption{Numerical solutions of the non-symmetric stabilized formulation~\eqref{stokes_trimming:stability:prob6_stab} for the \emph{lid-driven cavity} with the Taylor-Hood element.}\label{stokes_trimming:num_exp:fig12}
\end{figure}

\FloatBarrier

\appendix
\section{Useful inequalities} \label{sec:appendix}
In this section, we collect some technical results repeatedly employed throughout the manuscript. The constants that will appear in the inequalities below, unless otherwise specified, are intended to be robust with respect to the mesh size and mutual position between trimming curve and active physical B\'ezier mesh.
\begin{lemma}\label{appendix:lemma:disc_trace_ineq}
There exists $C>0$. depending on $\Gamma$, such that, for every $K\in\mathcal G_h$,
\begin{equation*}
\norm{v}^2_{L^2(\Gamma_K)}\le C \norm{v}_{L^2(K)}\norm{v}_{H^1(K)},\qquad\forall\ v\in H^1(K).
\end{equation*}	
\end{lemma}	
\begin{proof}
See, for instance, Lemma~3 in~\cite{MR1941489}, Lemma~3 in~\cite{MR2075053}, or Lemma~4.1 of~\cite{MR3407236}.
\end{proof}	
\begin{lemma}\label{appendix:51}
Let $Q,Q'\in\hat{\mathcal{M}}_{h}$ be neighbor elements in the sense of Definition~\ref{poisson_trimming:def_goodandbad}. There exists $C>0$ such that
\begin{equation*}
\norm{\varphi}_{L^\infty(Q)}\le C\norm{\varphi}_{L^\infty(Q')},\qquad\forall\ \varphi\in\mathbb{Q}_p(\R^d),
\end{equation*}
where $C$ depends on $p$, on the shape regularity of the mesh, and on the distance between $Q$ and $Q'$.
\end{lemma}
\begin{proof}
The proof follows by a scaling argument.	
\end{proof}
The next result says that the $L^2$-norm on the cut portion of an element $Q$ controls the $L^\infty$-norm (and hence any other) on the whole element with an equivalence constant depending on the relative measure of the cut portion.
\begin{lemma}\label{appendix:50}
Let $\theta\in (0,1]$. There exists $C>0$ such that, for every $Q\in\hat{\mathcal{M}}_{h}$ and every $S\subset Q$ measurable such that $\abs{S} \ge \theta\abs{Q} $, we have
\begin{equation*}
\norm{\varphi}_{L^\infty(Q)}\le C h^{-\frac{d}{2}}\norm{\varphi}_{L^2(S)},\qquad\forall\ \varphi\in\mathbb{Q}_p(\R^d),
\end{equation*}
where $C$ depends only on $\theta$, $p$, and the mesh regularity.
\end{lemma}
\begin{proof}
See Proposition~1 in~\cite{MR3806650}. 
\end{proof}
Let us recall a more standard inverse inequality with an explicit dependence on the polynomial degree $p$.
\begin{lemma}\label{appendix:lemma:schwab}
There exists $C>0$, depending on the shape regularity of the mesh, such that, for every $Q\in\hat{\mathcal M}_h$,
\begin{equation*}
\norm{\varphi}_{L^\infty(Q)}\le C p^d h_Q^{-\frac{d}{2}}\norm{\varphi}_{L^2(Q)},\qquad\forall\ \varphi \in\mathbb{Q}_p(Q).
\end{equation*}
\end{lemma}
\begin{proof}
We refer the interested reader to~\cite{MR1695813}. 
\end{proof}

\clearpage
\bibliographystyle{plain}
\bibliography{bibliography}

\begin{thebibliography}{10}

\bibitem{MR2424078}
Robert~A. Adams and John J.~F. Fournier.
\newblock {\em Sobolev spaces}, volume 140 of {\em Pure and Applied Mathematics
  (Amsterdam)}.
\newblock Elsevier/Academic Press, Amsterdam, second edition, 2003.

\bibitem{MR3982623}
Pablo Antolin, Annalisa Buffa, and Massimiliano Martinelli.
\newblock Isogeometric analysis on {V}-reps: first results.
\newblock {\em Comput. Methods Appl. Mech. Engrg.}, 355:976--1002, 2019.

\bibitem{MR752608}
John~W. Barrett and Charles~M. Elliott.
\newblock A finite-element method for solving elliptic equations with {N}eumann
  data on a curved boundary using unfitted meshes.
\newblock {\em IMA J. Numer. Anal.}, 4(3):309--325, 1984.

\bibitem{MR853660}
John~W. Barrett and Charles~M. Elliott.
\newblock Finite element approximation of the {D}irichlet problem using the
  boundary penalty method.
\newblock {\em Numer. Math.}, 49(4):343--366, 1986.

\bibitem{MR1813187}
Klaus-J\"{u}rgen Bathe.
\newblock The inf-sup condition and its evaluation for mixed finite element
  methods.
\newblock {\em Comput. \& Structures}, 79(2):243--252, 2001.

\bibitem{Bayraktar_benchmarkcomputations}
Evren Bayraktar, Otto Mierka, and Stefan Turek.
\newblock Benchmark computations of 3d laminar flow around a cylinder with cfx.
\newblock {\em International Journal of Computational Science and Engineering},
  pages 10--1504, 2012.

\bibitem{MR2250029}
Y.~Bazilevs, L.~Beir\~{a}o~da Veiga, J.~A. Cottrell, T.~J.~R. Hughes, and
  G.~Sangalli.
\newblock Isogeometric analysis: approximation, stability and error estimates
  for {$h$}-refined meshes.
\newblock {\em Math. Models Methods Appl. Sci.}, 16(7):1031--1090, 2006.

\bibitem{MR2571349}
Roland Becker, Erik Burman, and Peter Hansbo.
\newblock A {N}itsche extended finite element method for incompressible
  elasticity with discontinuous modulus of elasticity.
\newblock {\em Comput. Methods Appl. Mech. Engrg.}, 198(41-44):3352--3360,
  2009.

\bibitem{MR3202239}
L.~Beir\~{a}o~da Veiga, A.~Buffa, G.~Sangalli, and R.~V\'{a}zquez.
\newblock Mathematical analysis of variational isogeometric methods.
\newblock {\em Acta Numer.}, 23:157--287, 2014.

\bibitem{MR972452}
Christine Bernardi, Claudio Canuto, and Yvon Maday.
\newblock Generalized inf-sup conditions for {C}hebyshev spectral approximation
  of the {S}tokes problem.
\newblock {\em SIAM J. Numer. Anal.}, 25(6):1237--1271, 1988.

\bibitem{MR2846758}
Andrea Bressan.
\newblock Isogeometric regular discretization for the {S}tokes problem.
\newblock {\em IMA J. Numer. Anal.}, 31(4):1334--1356, 2011.

\bibitem{MR3047946}
Andrea Bressan and Giancarlo Sangalli.
\newblock Isogeometric discretizations of the {S}tokes problem: stability
  analysis by the macroelement technique.
\newblock {\em IMA J. Numer. Anal.}, 33(2):629--651, 2013.

\bibitem{MR2808112}
A.~Buffa, C.~de~Falco, and G.~Sangalli.
\newblock Iso{G}eometric {A}nalysis: stable elements for the 2{D} {S}tokes
  equation.
\newblock {\em Internat. J. Numer. Methods Fluids}, 65(11-12):1407--1422, 2011.

\bibitem{MR4155233}
A.~Buffa, R.~Puppi, and R.~V\'{a}zquez.
\newblock A minimal stabilization procedure for isogeometric methods on trimmed
  geometries.
\newblock {\em SIAM J. Numer. Anal.}, 58(5):2711--2735, 2020.

\bibitem{MR2792397}
A.~Buffa, J.~Rivas, G.~Sangalli, and R.~V\'{a}zquez.
\newblock Isogeometric discrete differential forms in three dimensions.
\newblock {\em SIAM J. Numer. Anal.}, 49(2):818--844, 2011.

\bibitem{MR2738930}
Erik Burman.
\newblock Ghost penalty.
\newblock {\em C. R. Math. Acad. Sci. Paris}, 348(21-22):1217--1220, 2010.

\bibitem{MR3416285}
Erik Burman, Susanne Claus, Peter Hansbo, Mats~G. Larson, and Andr\'{e}
  Massing.
\newblock Cut{FEM}: discretizing geometry and partial differential equations.
\newblock {\em Internat. J. Numer. Methods Engrg.}, 104(7):472--501, 2015.

\bibitem{MR3264337}
Erik Burman and Peter Hansbo.
\newblock Fictitious domain methods using cut elements: {III}. {A} stabilized
  {N}itsche method for {S}tokes' problem.
\newblock {\em ESAIM Math. Model. Numer. Anal.}, 48(3):859--874, 2014.

\bibitem{MR3618875}
J.~Austin Cottrell, Thomas J.~R. Hughes, and Yuri Bazilevs.
\newblock {\em Isogeometric analysis}.
\newblock John Wiley \& Sons, Ltd., Chichester, 2009.
\newblock Toward integration of CAD and FEA.

\bibitem{MR3915341}
F.~de~Prenter, C.~V. Verhoosel, and E.~H. van Brummelen.
\newblock Preconditioning immersed isogeometric finite element methods with
  application to flow problems.
\newblock {\em Comput. Methods Appl. Mech. Engrg.}, 348:604--631, 2019.

\bibitem{MR3610101}
F.~de~Prenter, C.~V. Verhoosel, G.~J. van Zwieten, and E.~H. van Brummelen.
\newblock Condition number analysis and preconditioning of the finite cell
  method.
\newblock {\em Comput. Methods Appl. Mech. Engrg.}, 316:297--327, 2017.

\bibitem{MR3010180}
John~A. Evans and Thomas J.~R. Hughes.
\newblock Explicit trace inequalities for isogeometric analysis and parametric
  hexahedral finite elements.
\newblock {\em Numer. Math.}, 123(2):259--290, 2013.

\bibitem{MR3021778}
John~A. Evans and Thomas J.~R. Hughes.
\newblock Isogeometric divergence-conforming {B}-splines for the
  {D}arcy-{S}tokes-{B}rinkman equations.
\newblock {\em Math. Models Methods Appl. Sci.}, 23(4):671--741, 2013.

\bibitem{MR3048532}
John~A. Evans and Thomas J.~R. Hughes.
\newblock Isogeometric divergence-conforming {B}-splines for the steady
  {N}avier-{S}tokes equations.
\newblock {\em Math. Models Methods Appl. Sci.}, 23(8):1421--1478, 2013.

\bibitem{MR3806650}
Michel Fourni\'{e} and Alexei Lozinski.
\newblock Stability and optimal convergence of unfitted extended finite element
  methods with {L}agrange multipliers for the {S}tokes equations.
\newblock In {\em Geometrically unfitted finite element methods and
  applications}, volume 121 of {\em Lect. Notes Comput. Sci. Eng.}, pages
  143--182. Springer, Cham, 2017.

\bibitem{fcd86cec3aaf40f58936b8118fca3f44}
J.~Freund and R.~Stenberg.
\newblock On weakly imposed boundary conditions in the finite element method.
\newblock In {M. Morandi} {Cecchi }, editor, {\em The ninth conference on
  Finite Elements in Fluids, Venezia, 16.-20.10.1995}, pages 327--336, 1995.

\bibitem{Grcan2003StreamlineTI}
F.~G{\"u}rcan.
\newblock Streamline topologies in stokes flow within lid-driven cavities.
\newblock {\em Theoretical and Computational Fluid Dynamics}, 17:19--30, 2003.

\bibitem{MR1941489}
Anita Hansbo and Peter Hansbo.
\newblock An unfitted finite element method, based on {N}itsche's method, for
  elliptic interface problems.
\newblock {\em Comput. Methods Appl. Mech. Engrg.}, 191(47-48):5537--5552,
  2002.

\bibitem{MR2075053}
Anita Hansbo and Peter Hansbo.
\newblock A finite element method for the simulation of strong and weak
  discontinuities in solid mechanics.
\newblock {\em Comput. Methods Appl. Mech. Engrg.}, 193(33-35):3523--3540,
  2004.

\bibitem{MR2923416}
Ralf Hiptmair, Jingzhi Li, and Jun Zou.
\newblock Universal extension for {S}obolev spaces of differential forms and
  applications.
\newblock {\em J. Funct. Anal.}, 263(2):364--382, 2012.

\bibitem{tuong_thesis}
Tuong Hoang.
\newblock {\em Isogeometric and immersogeometric analysis of incompressible
  flow problems}.
\newblock PhD thesis, TU Eindhoven, 2018.

\bibitem{MR3610105}
Tuong Hoang, Clemens~V. Verhoosel, Ferdinando Auricchio, E.~Harald van
  Brummelen, and Alessandro Reali.
\newblock Mixed isogeometric finite cell methods for the {S}tokes problem.
\newblock {\em Comput. Methods Appl. Mech. Engrg.}, 316:400--423, 2017.

\bibitem{MR3801783}
Tuong Hoang, Clemens~V. Verhoosel, Ferdinando Auricchio, E.~Harald van
  Brummelen, and Alessandro Reali.
\newblock Skeleton-stabilized isogeometric analysis: high-regularity
  interior-penalty methods for incompressible viscous flow problems.
\newblock {\em Comput. Methods Appl. Mech. Engrg.}, 337:324--351, 2018.

\bibitem{MR3873865}
Tuong Hoang, Clemens~V. Verhoosel, Chao-Zhong Qin, Ferdinando Auricchio,
  Alessandro Reali, and E.~Harald van Brummelen.
\newblock Skeleton-stabilized immersogeometric analysis for incompressible
  viscous flow problems.
\newblock {\em Comput. Methods Appl. Mech. Engrg.}, 344:421--450, 2019.

\bibitem{MR2152382}
T.~J.~R. Hughes, J.~A. Cottrell, and Y.~Bazilevs.
\newblock Isogeometric analysis: {CAD}, finite elements, {NURBS}, exact
  geometry and mesh refinement.
\newblock {\em Comput. Methods Appl. Mech. Engrg.}, 194(39-41):4135--4195,
  2005.

\bibitem{MR3310316}
D.~Kamensky, M.-C. Hsu, D.~Schillinger, J.~A. Evans, A.~Aggarwal, Y.~Bazilevs,
  M.~S. Sacks, and T.~J.~R. Hughes.
\newblock An immersogeometric variational framework for fluid-structure
  interaction: application to bioprosthetic heart valves.
\newblock {\em Comput. Methods Appl. Mech. Engrg.}, 284:1005--1053, 2015.

\bibitem{MR3589879}
David Kamensky, Ming-Chen Hsu, Yue Yu, John~A. Evans, Michael~S. Sacks, and
  Thomas J.~R. Hughes.
\newblock Immersogeometric cardiovascular fluid-structure interaction analysis
  with divergence-conforming {B}-splines.
\newblock {\em Comput. Methods Appl. Mech. Engrg.}, 314:408--472, 2017.

\bibitem{MR3867694}
Benjamin Marussig and Thomas J.~R. Hughes.
\newblock A review of trimming in isogeometric analysis: challenges, data
  exchange and simulation aspects.
\newblock {\em Arch. Comput. Methods Eng.}, 25(4):1059--1127, 2018.

\bibitem{MR3513034}
Fady Massarwi and Gershon Elber.
\newblock A {B}-spline based framework for volumetric object modeling.
\newblock {\em Comput.-Aided Des.}, 78:36--47, 2016.

\bibitem{MR650055}
R.~A. Nicolaides.
\newblock Existence, uniqueness and approximation for generalized saddle point
  problems.
\newblock {\em SIAM J. Numer. Anal.}, 19(2):349--357, 1982.

\bibitem{puppi_stokes}
Riccardo Puppi.
\newblock {Isogeometric discretizations of the Stokes problem on trimmed
  geometries}.
\newblock arXiv, 2020.

\bibitem{MR1299729}
Alfio Quarteroni and Alberto Valli.
\newblock {\em Numerical approximation of partial differential equations},
  volume~23 of {\em Springer Series in Computational Mathematics}.
\newblock Springer-Verlag, Berlin, 1994.

\bibitem{RANK2012104}
Ernst Rank, Martin Ruess, Stefan Kollmannsberger, Dominik Schillinger, and
  Alexander D{\"u}ster.
\newblock Geometric modeling, isogeometric analysis and the finite cell method.
\newblock {\em Computer Methods in Applied Mechanics and Engineering},
  249-252:104 -- 115, 2012.
\newblock Higher Order Finite Element and Isogeometric Methods.

\bibitem{MR3407236}
Arnold Reusken.
\newblock Analysis of trace finite element methods for surface partial
  differential equations.
\newblock {\em IMA J. Numer. Anal.}, 35(4):1568--1590, 2015.

\bibitem{MR1695813}
Ch. Schwab.
\newblock {\em {$p$}- and {$hp$}-finite element methods}.
\newblock Numerical Mathematics and Scientific Computation. The Clarendon
  Press, Oxford University Press, New York, 1998.
\newblock Theory and applications in solid and fluid mechanics.

\bibitem{MR1365557}
Rolf Stenberg.
\newblock On some techniques for approximating boundary conditions in the
  finite element method.
\newblock volume~63, pages 139--148. 1995.
\newblock International Symposium on Mathematical Modelling and Computational
  Methods Modelling 94 (Prague, 1994).

\bibitem{Sturges1986StokesFI}
L.~D. Sturges.
\newblock Stokes flow in a two‐dimensional cavity with moving end walls.
\newblock {\em Physics of Fluids}, 29:1731--1734, 1986.

\bibitem{MR2328004}
Luc Tartar.
\newblock {\em An introduction to {S}obolev spaces and interpolation spaces},
  volume~3 of {\em Lecture Notes of the Unione Matematica Italiana}.
\newblock Springer, Berlin; UMI, Bologna, 2007.

\end{thebibliography}
\end{document}


\newcommand{\logLogSlopeTriangleOne}[5]
	{
		
		\pgfplotsextra
		{
			\pgfkeysgetvalue{/pgfplots/xmin}{\xmin}
			\pgfkeysgetvalue{/pgfplots/xmax}{\xmax}
			\pgfkeysgetvalue{/pgfplots/ymin}{\ymin}
			\pgfkeysgetvalue{/pgfplots/ymax}{\ymax}
			
			\pgfmathsetmacro{\xArel}{#1}
			\pgfmathsetmacro{\yArel}{#3}
			\pgfmathsetmacro{\xBrel}{#1-#2}
			\pgfmathsetmacro{\yBrel}{\yArel}
			\pgfmathsetmacro{\xCrel}{\xArel}
			
			\pgfmathsetmacro{\lnxB}{\xmin*(1-(#1-#2))+\xmax*(#1-#2)} 
			\pgfmathsetmacro{\lnxA}{\xmin*(1-#1)+\xmax*#1} 
			\pgfmathsetmacro{\lnyA}{\ymin*(1-#3)+\ymax*#3} 
			\pgfmathsetmacro{\lnyC}{\lnyA+#4*(\lnxA-\lnxB)}
			\pgfmathsetmacro{\yCrel}{\lnyC-\ymin)/(\ymax-\ymin)} 
			
			\coordinate (A) at (rel axis cs:\xArel,\yArel);
			\coordinate (B) at (rel axis cs:\xBrel,\yBrel);
			\coordinate (C) at (rel axis cs:\xCrel,\yCrel);
			
			\draw[#5]   (A)-- node[pos=0.5,anchor=north] {1} 
			(B)-- 
			(C)-- node[pos=0.5,anchor=east] {#4} 
			cycle;
		}
	}
	

	\begin{tikzpicture}

		\definecolor{blu}{rgb}{0,0,1}
		\definecolor{red}{rgb}{1,0,0}
		\definecolor{green}{rgb}{0.3,0.7,0.2}
		\definecolor{mygray}{rgb}{0.5,0.5,0.5}
		\definecolor{myyellow}{rgb}{0.9,0.7,0.1}
		\definecolor{mymagenta}{rgb}{255,0,255}
		
		\begin{loglogaxis}[
	width=5.5cm,
	height=5.5cm,
	xlabel={$h$},
	ylabel={$\beta_0$},
	xtick=data,
	xticklabels={$2^{-2}$, $2^{-3}$, $2^{-4}$, $2^{-5}$, $2^{-6}$},
	ymax= 10^0,
	ymin= 10^-6,
	legend pos=outer north east,
	grid=major,
	]

\addplot[semithick, blue, mark=o, mark size=2, mark options={draw=blue}] table[ x index = 0, y index  =1] 
{\dataDir/inf_sup_b0_pentagon_N.txt};
\addplot[semithick, red, mark=square, mark size=2, mark options={draw=red}] table[ x index = 0, y index = 2]
 {\dataDir/inf_sup_b0_pentagon_N.txt};
\addplot[semithick, green, mark=pentagon, mark size=2, mark options={draw=green}] table[ x index = 0, y index = 3] 
{\dataDir/inf_sup_b0_pentagon_N.txt};
\addplot[semithick, mygray, mark=diamond, mark size=2, mark options={draw=mygray}] table[ x index = 0, y index = 4]
 {\dataDir/inf_sup_b0_pentagon_N.txt};
\addplot[semithick, myyellow, mark= star, mark size=2, mark options={draw=myyellow}] table[ x index = 0, y index = 5] 
{\dataDir/inf_sup_b0_pentagon_N.txt};
\addplot[semithick, mymagenta, mark=triangle, mark size=2, mark options={draw=mymagenta}] table[ x index = 0, y index = 6] 
{\dataDir/inf_sup_b0_pentagon_N.txt};

\end{loglogaxis}

\end{tikzpicture}


\newcommand{\logLogSlopeTriangleOne}[5]
{

    \pgfplotsextra
    {
        \pgfkeysgetvalue{/pgfplots/xmin}{\xmin}
        \pgfkeysgetvalue{/pgfplots/xmax}{\xmax}
        \pgfkeysgetvalue{/pgfplots/ymin}{\ymin}
        \pgfkeysgetvalue{/pgfplots/ymax}{\ymax}

        \pgfmathsetmacro{\xArel}{#1}
        \pgfmathsetmacro{\yArel}{#3}
        \pgfmathsetmacro{\xBrel}{#1-#2}
        \pgfmathsetmacro{\yBrel}{\yArel}
        \pgfmathsetmacro{\xCrel}{\xArel}

        \pgfmathsetmacro{\lnxB}{\xmin*(1-(#1-#2))+\xmax*(#1-#2)} 
        \pgfmathsetmacro{\lnxA}{\xmin*(1-#1)+\xmax*#1} 
        \pgfmathsetmacro{\lnyA}{\ymin*(1-#3)+\ymax*#3} 
        \pgfmathsetmacro{\lnyC}{\lnyA+#4*(\lnxA-\lnxB)}
        \pgfmathsetmacro{\yCrel}{\lnyC-\ymin)/(\ymax-\ymin)} 

        \coordinate (A) at (rel axis cs:\xArel,\yArel);
        \coordinate (B) at (rel axis cs:\xBrel,\yBrel);
        \coordinate (C) at (rel axis cs:\xCrel,\yCrel);

        \draw[#5]   (A)-- node[pos=0.5,anchor=north] {1} 
                    (B)-- 
                    (C)-- node[pos=0.5,anchor=east] {#4} 
                    cycle;
    }
}


\begin{tikzpicture}

\definecolor{color1}{rgb}{255, 0, 0}
\definecolor{color2}{rgb}{0,255,0}
\definecolor{color3}{rgb}{0,0,255}
\definecolor{color4}{rgb}{255,0,255}
\definecolor{color5}{rgb}{1,0.9,0}
\definecolor{color6}{rgb}{0,255,180}
\definecolor{color7}{rgb}{0,70,0}

  \begin{semilogxaxis}[
        scale only axis, 
        xlabel={$h$},
        ylabel={inf-sup constant},
      xtick=data,
      xticklabels={$1/4$, $1/8$, $1/16$, $1/32$, $1/64$},
        legend pos=south west,
legend style={font=\Large},
grid=major,
xlabel style={font=\Large},
ylabel style={font=\Large},
xticklabel style={font=\Large},
yticklabel style={font=\Large}
      ]

\addplot[semithick, color1, mark=o, mark size=3, mark options={draw=color1}] table[ x index = 0, y index  =1] {\dataDir/Darcy/inf_sup_b1_pentagon_stab_press.txt};
\addlegendentry{$\varepsilon=10^{-3}$};
\addplot[semithick, color2, mark=square, mark size=3, mark options={draw=color2}] table[ x index = 0, y index = 2] {\dataDir/Darcy/inf_sup_b1_pentagon_stab_press.txt};
\addlegendentry{$\varepsilon=10^{-5}$};
\addplot[semithick, color3, mark=*, mark size=3, mark options={draw=color3}] table[ x index = 0, y index = 3] {\dataDir/Darcy/inf_sup_b1_pentagon_stab_press.txt};
\addlegendentry{$\varepsilon=10^{-7}$};
\addplot[semithick, color4, mark=diamond, mark size=3, mark options={draw=color4}] table[ x index = 0, y index = 4] {\dataDir/Darcy/inf_sup_b1_pentagon_stab_press.txt};
\addlegendentry{$\varepsilon=10^{-9}$};
\addplot[semithick, color5, mark=otimes, mark size=3, mark options={draw=color5}] table[ x index = 0, y index = 5] {\dataDir/Darcy/inf_sup_b1_pentagon_stab_press.txt};
\addlegendentry{$\varepsilon=10^{-11}$};
\addplot[semithick, color6, mark=triangle, mark size=3, mark options={draw=color6}] table[ x index = 0, y index = 6] {\dataDir/Darcy/inf_sup_b1_pentagon_stab_press.txt};
\addlegendentry{$\varepsilon=10^{-13}$};
\end{semilogxaxis}

\end{tikzpicture}